\documentclass[11pt,twoside,reqno]{amsart}
\usepackage{amsmath, amssymb, amsthm, amsfonts, cite, dsfont, enumerate, epsfig, float, infwarerr, mathrsfs, mathtools, mathrsfs, mathtools, relsize, stmaryrd, tabularx, txfonts}
\usepackage[normalem]{ulem}
\usepackage[colorlinks=true,citecolor=blue]{hyperref}

\makeatletter\def\Hy@Warning#1{}\makeatother

\newtheorem{theorem}{Theorem}[section]
\newtheorem{lemma}[theorem]{Lemma}
\newtheorem{proposition}[theorem]{Proposition}
\newtheorem{corollary}[theorem]{Corollary}

\theoremstyle{definition}
\newtheorem{definition}[theorem]{Definition}

\newtheorem{example}[theorem]{Example}

\newtheorem{remark}[theorem]{Remark}
\numberwithin{equation}{section}

\newcommand{\naturals}{\ensuremath{\mathbb{N}}}
\newcommand{\rationals}{\ensuremath{\mathbb{Q}}}
\newcommand{\Reals}{\ensuremath{\mathbb{R}}}

\newcommand{\set}{\ensuremath{\mathcal}}

\newcommand{\OneTo}[1]{[#1]}
\newcommand{\eqdef}{\triangleq}    
\newcommand{\card}[1]{|#1|}

\DeclareMathOperator{\Vertex}{\mathsf{V}}
\DeclareMathOperator{\Edge}{\mathsf{E}}
\DeclareMathOperator{\Adjacency}{\mathbf{A}}
\DeclareMathOperator{\Laplacian}{\mathbf{L}}
\DeclareMathOperator{\SignlessLaplacian}{\mathbf{Q}}
\DeclareMathOperator{\AllOne}{\mathbf{J}}
\DeclareMathOperator{\Identity}{\mathbf{I}}
\DeclareMathOperator{\Independentset}{\set{I}}
\DeclareMathOperator{\Kneser}{\mathsf{K}}
\DeclareMathOperator{\Complete}{\mathsf{K}}
\DeclareMathOperator{\Empty}{\mathsf{E}}

\DeclareMathOperator{\Path}{\mathsf{P}}
\DeclareMathOperator{\Cycle}{\mathsf{C}}
\DeclareMathOperator{\SRG}{\mathsf{srg}}
\DeclareMathOperator{\Sp}{\mathsf{Sp}}

\DeclareMathOperator{\Clique}{\omega}
\DeclareMathOperator{\Chromatic}{\chi}

\newcommand{\Gr}[1]{\mathsf{#1}}                          
\newcommand{\CGr}[1]{\overline{\mathsf{#1}}}              
\newcommand{\V}[1]{\Vertex(#1)}                           
\newcommand{\E}[1]{\Edge(#1)}                             
\newcommand{\A}{\Adjacency}                               
\newcommand{\LM}{\Laplacian}                              
\newcommand{\Q}{\SignlessLaplacian}                       
\newcommand{\D}{\mathbf{D}}                               
\newcommand{\J}[1]{\AllOne_{#1}}                          
\newcommand{\I}[1]{\Identity_{#1}}                        
\newcommand{\indset}[1]{\Independentset(#1)}              
\newcommand{\indsetmax}[1]{\Independentset_{\max}(#1)}    
\newcommand{\indnum}[1]{\alpha(#1)}                       

\newcommand{\findnum}[1]{\alpha_{\mathrm{f}}(#1)}         
\newcommand{\clnum}[1]{\Clique(#1)}                       

\newcommand{\fclnum}[1]{\Clique_{\mathrm{f}}(#1)}         
\newcommand{\chrnum}[1]{\Chromatic(#1)}                   

\newcommand{\fchrnum}[1]{\Chromatic_{\mathrm{f}}(#1)}     

\newcommand{\vchrnum}[1]{\Chromatic_{\mathrm{v}}(#1)}     

\newcommand{\svchrnum}[1]{\Chromatic_{\mathrm{sv}}(#1)}     

\newcommand{\Eigval}[2]{\lambda_{#1}(#2)}  
\newcommand{\CoG}[1]{\Complete_{#1}}       
\newcommand{\CoBG}[2]{\Complete_{#1,#2}}   
\newcommand{\EmG}[1]{\Empty_{#1}}          
\newcommand{\KG}[2]{\Kneser(#1,#2)}        
\newcommand{\PathG}[1]{\Path_{#1}}         

\newcommand{\CG}[1]{\Cycle_{#1}}           
\newcommand{\srg}[4]{\SRG(#1,#2,#3,#4)}    
\newcommand{\LSG}[2]{\mathrm{L}_{#1}(#2)}  
\newcommand{\PLG}[2]{\mathrm{PL}_{#1}(#2)} 

\newcommand{\NS}{\, \underline{\vee} \,}   
\newcommand{\NNS}{\, \uuline{\vee} \,}     

\DeclareMathOperator{\Degree}{\text{d}}
\newcommand{\dgr}[1]{\Degree_{#1}}   

\MHInternalSyntaxOn
\renewcommand{\dcases}
 {
  \MT_start_cases:nnnn
    {\quad}
    {$\m@th\displaystyle##$\hfil}
    {$\m@th\displaystyle##$\hfil}
    {\lbrace}
 }
\MHInternalSyntaxOff

\DeclareMathOperator{\Entr}{H}

\newcommand{\EntCond}[2]{\Entr(#1 | \kern0.1em #2)}
\newcommand{\bigEntCond}[2]{\Entr\bigl(#1 | \kern0.1em #2\bigr)}
\newcommand{\BigEntCond}[2]{\Entr\Bigl(#1 \kern-0.1em \bigm| \kern-0.1em #2 \Bigr)}
\newcommand{\biggEntCond}[2]{\Entr\biggl(#1 \kern-0.1em \Bigm| \kern-0.1em #2 \biggr)}
\newcommand{\BiggEntCond}[2]{\Entr\Biggl(#1 \kern-0.1em \biggm| \kern-0.1em #2 \Biggr)}

\makeatletter\c@MaxMatrixCols=15\makeatother   

\begin{document}
\setcounter{page}{1}

\vspace*{2cm}
\title[Observations on Graph Invariants with the Lov\'{a}sz $\vartheta$-Function]
{Observations on Graph Invariants with the Lov\'{a}sz $\vartheta$-Function}
\author[IGAL SASON]{Igal Sason}
\maketitle

\vspace*{-0.6cm}
\begin{center}
{\footnotesize

The Viterbi Faculty of Electrical and Computer Engineering,
and the Department of Mathematics, \\Technion - Israel Institute of Technology, Haifa 3200003, Israel.}
\end{center}

\begin{center}
{\footnotesize Dedicated to Professor Emeritus Abraham (Avi) Berman for his eightieth birthday.}
\end{center}

\vskip 4mm {\footnotesize \noindent {\bf Abstract.}
This paper delves into three research directions, leveraging the Lov\'{a}sz $\vartheta$-function
of a graph. First, it focuses on the Shannon capacity of graphs, providing new results that
determine the capacity for two infinite subclasses of strongly regular graphs, and extending prior
results. The second part explores cospectral and nonisomorphic graphs, drawing on a work by Berman
and Hamud (2024), and it derives related properties of two types of joins of graphs. For every even
integer such that $n \geq 14$, it is constructively proven that there exist connected, irregular, cospectral,
and nonisomorphic graphs on $n$ vertices, being jointly cospectral with respect to their adjacency,
Laplacian, signless Laplacian, and normalized Laplacian matrices, while also sharing identical independence,
clique, and chromatic numbers, but being distinguished by their Lov\'{a}sz $\vartheta$-functions.
The third part focuses on establishing bounds on graph invariants, particularly emphasizing strongly
regular graphs and triangle-free graphs, and compares the tightness of these bounds to existing ones.
The paper derives spectral upper and lower bounds on the vector and strict vector chromatic numbers of
regular graphs, providing sufficient conditions for the attainability of these bounds. Exact closed-form
expressions for the vector and strict vector chromatic numbers are derived for all strongly regular graphs
and for all graphs that are vertex- and edge-transitive, demonstrating that these two types of chromatic
numbers coincide for every such graph. This work resolves a query regarding the variant of the
$\vartheta$-function by Schrijver and the identical function by McEliece {\em et al.} (1978).
It shows, by a counterexample, that the $\vartheta$-function variant by Schrijver does not possess the property of
the Lov\'{a}sz $\vartheta$-function of forming an upper bound on the Shannon capacity of a graph. This research
paper also serves as a tutorial of mutual interest in zero-error information theory and
algebraic graph theory.


}

\renewcommand{\thefootnote}{}
\footnotetext{
E-mail address: eeigal@technion.ac.il (I. Sason). Cite: I. Sason, Observations on graph invariants with the Lov\'{a}sz $\vartheta$-function,
{\em AIMS-Mathematics}, 9 (2024), no. 6, pp.~15385--15468, April 2024. \url{https://doi.org/10.3934/math.2024747}}

\section{Introduction}
\label{section: introduction}

The concept of the graph capacity, as introduced in Shannon's seminal 1956 paper on zero-error communication
\cite{Shannon56}, plays a key role for understanding the synergy and interaction between
the fields of information theory and graph theory. This importance is highlighted in various
survey papers \cite{Alon02,Alon19,Jurkiewicz14,KornerO98}.

The behavior of the Shannon capacity of a graph, along with its convergence, is quite erratic as is demonstrated
by various studies \cite{Alon98,AlonL06,GuoW90}. The Shannon capacity of a graph and the zero-error
capacity of a discrete memoryless channel are generally difficult to compute or even approximate,
and notions of their computability have been explored \cite{AlonL06,BocheD20,BocheD21}. Despite these studies,
some computability aspects of the capacities in question remain unresolved. Using computationally feasible
upper bounds on the Shannon capacity of a graph therefore provides valuable insights
\cite{AlipourG23,BiT19,BukhC19,Haemers79,HuTS18,Knuth94,Lovasz79_IT,GuruswamiR_ISIT21}.
These bounds, including the easily computable Lov\'{a}sz $\vartheta$-function \cite{Lovasz79_IT,Knuth94}
and others that may be more challenging to compute, offer informative results. In some cases, these bounds
also prove to be tight.

The concept of the asymptotic spectrum of graphs, introduced by Zuiddam \cite{Zuiddam19}, delineates a
space of graph parameters that remain invariant under graph isomorphisms. This space is characterized by
the following unique properties: additivity under disjoint union of graphs, multiplicativity under
strong product of graphs, normalization for a simple graph with a single vertex, and monotonicity
under graph complement homomorphisms. Building upon Strassen's theory of the asymptotic spectra \cite{Strassen88},
a novel dual characterization of the Shannon capacity of a graph is derived in \cite{Zuiddam19}, expressed
as the minimum over the elements of its asymptotic spectrum. By confirming that various graph invariants,
including the Lov\'{a}sz $\vartheta$-function \cite{Lovasz79_IT} and the fractional Haemers bound \cite{BukhC19},
are elements of the asymptotic spectrum of a graph (spectral points), it can be deduced that these elements
indeed serve as upper bounds on the Shannon capacity of a graph. For further exploration, the comprehensive
paper by Wigderson and Zuiddam \cite{WigdersonZ23} provides a survey on Strassen’s theory of the asymptotic
spectra and its application areas, including the Shannon capacity of graphs.

The Lov\'{a}sz $\vartheta$-function of a graph, introduced in \cite{Lovasz79_IT}, has important and fascinating
properties. It serves as an upper bound on the independence number, and even on the Shannon capacity
of a graph, while concurrently functioning as a lower bound on the chromatic number and even on the
fractional chromatic number of the complement graph. It found many other applications in graph theory,
such as providing a novel lower bound on the max-cut of the complement graph \cite{BallaJS24}. Remarkably,
the Lov\'{a}sz $\vartheta$-function, which is a combinatorial characteristic of graphs, reveals intriguing
connections with probabilistic properties, particularly in lower bounding the decoding error probability in
classical and quantum information theory, as explored in \cite{Dalai_IT13, Dalai_ISIT13}.
Additionally, this function was independently extended in \cite{DuanSW_IT13,BorelandTW21} within the framework
of zero-error communication via general quantum channels.

The Lov\'{a}sz $\vartheta$-function of a graph is efficiently computable by the numerical solution
of semidefinite optimization problems \cite{GrotschelLS81,GrotschelLS84,Knuth94,Lovasz79_IT,Lovasz19}.
Its computational efficiency is notable, particularly in light of the NP-hard complexity inherent
in computing other graph invariants, such as the independence, clique, and chromatic numbers of a
graph, as well as their fractional counterparts \cite{GareyJ79,GrotschelLS81}.

This paper explores three research directions, using the properties of the Lov\'{a}sz
$\vartheta$-function of a graph. Its main results and structure are outlined as follows.
\begin{itemize}
\item Section~\ref{section: preliminaries} presents introductory content and notation.
Additionally, every subsequent section (Sections~\ref{section: on the Shannon capacity of graphs}
to \ref{section: bounds on graph invariants}) is equipped with specialized preliminaries tailored
to its unique focus, offered at the outset of each respective section.

\item Section~\ref{section: on the Shannon capacity of graphs} relies on the Lov\'{a}sz
$\vartheta$-function for the derivation of new results on the Shannon capacity of graphs
(Theorems~\ref{thm:extension of Thm. 12 by Lovasz}--\ref{thm:on the symplectic graphs}).
This includes the calculation of the Shannon capacity for two infinite subclasses
of strongly regular graphs, along with extensions of previously established findings. The
subsequent discussions and illustrative examples further elucidate these results.

\item Section~\ref{section: cospectral and nonisomorphic graphs} presents new results
on cospectral and nonisomorphic irregular graphs. For every even integer $n \geq 14$, it offers
a construction of connected, irregular, cospectral, and nonisomorphic graphs on $n$ vertices,
sharing identical independence, clique, and chromatic numbers, but being distinguished by their
Lov\'{a}sz $\vartheta$-functions (Theorem~\ref{theorem: existence of NICS graphs}). The result
relies in part on \cite{vanDamH03,HamudB24}, and the constructed graphs have identical spectra
with respect to each of the following matrices: the adjacency, Laplacian, signless Laplacian,
and normalized Laplacian matrices. This section further provides exact expressions or bounds
on graph invariants such as the independence number, clique number, chromatic number,
Shannon capacity, and the Lov\'{a}sz $\vartheta$-function for the two types of joins of graphs
introduced in \cite{Hamud23,HamudB24} (see Theorems~\ref{theorem: on Graph Invariants of NS/NNS Joins of Graphs}
and~\ref{theorem: Shannon Capacity of NS and NNS Join Graphs} in this paper).
Section~\ref{section: cospectral and nonisomorphic graphs} concludes with discussions
and numerical experiments elaborating on these results.

\item Section~\ref{section: bounds on graph invariants} relies on the properties of the Lov\'{a}sz
$\vartheta$-function for the derivation of bounds on graph invariants. The tightness of these bounds is then
compared to existing ones. In light of the analysis presented in \cite{Sason23}, closed-form bounds on
graph invariants for strongly regular graphs are provided in Corollary~\ref{corollary:bounds on parameters of SRGs}.
This section also presents spectral upper and lower bounds on the vector and strict vector chromatic
numbers of regular graphs, along with sufficient conditions for their attainability, as stated in
Theorem~\ref{theorem: bounds on vchrnum and svchrnum for regular graphs}. Further, exact closed-form
expressions for the vector and strict vector chromatic numbers are derived in
Theorems~\ref{theorem: vchrnum and svchrnum of srg} and~\ref{theorem: vchrmum and svchrnum of vt+et graphs},
covering all strongly regular graphs and all graphs that are both vertex- and edge-transitive.
This establishes that these two types of chromatic numbers coincide for all such regular graphs.
The fractional chromatic number of a graph is known to be lower-bounded by the Lov\'{a}sz $\vartheta$-function
of the complement graph. That lower bound on the fractional chromatic number is studied in the context of
perfect graphs (Theorem~\ref{theorem: graph invariants of perfect graphs}), and triangle-free graphs. For
triangle-free graphs on $n$ vertices, a lower bound on the Lov\'{a}sz $\vartheta$-function is provided in
Theorem~\ref{theorem: LB for triangle-free graphs}, serving as a counterpart to an upper bound for a subfamily
of regular and triangle-free graphs in \cite{Alon94}. These bounds match up to a constant scaling factor
and scale proportionally to $n^{\frac{2}{3}}$. Theorem~\ref{thm: lower bound on the independence number} further
provides lower bounds on the independence and clique numbers of a graph, derived from the Lov\'{a}sz $\vartheta$-function of
the graph or its complement. Additionally, a counterexample demonstrates that, in contrast to the Lov\'{a}sz
$\vartheta$-function, which provides an upper bound on the Shannon capacity of a graph, the variant of the $\vartheta$-function
by Schrijver does not possess that property (Example~\ref{example: counterexample}). This resolves a query regarding
the variant of the $\vartheta$-function proposed by Schrijver and the identical function presented by McEliece {\em et al.} (1978)
\cite{McElieceRR78,Schrijver79}, which was also posed as an open question in \cite{BiT19}.
Throughout, the presented results in Section~\ref{section: bounds on graph invariants}
are substantiated by numerical findings and are discussed in connection with other existing bounds.

\item Section~\ref{section: summary} furnishes a concise overview and delineates related open problems.
\end{itemize}

\section{Preliminaries}
\label{section: preliminaries}

\subsection{Graphs and Matrices}
\label{subsection: preliminaries on graphs}
Essential terms and standard notation are presented as follows.

\subsubsection{General Terminology and Notation}
\label{subsubsection: general}
Let $\Gr{G} = (\Vertex, \Edge)$ be a graph.
\begin{itemize}
\item $\Vertex=\V{\Gr{G}}$ is the {\em vertex set} of $\Gr{G}$, and $\Edge=\E{\Gr{G}}$ is the {\em edge set} of $\Gr{G}$.
\item An {\em undirected graph} is a graph whose edges are undirected.
\item A {\em self-loop} is an edge that connects a vertex to itself.
\item A {\em simple graph} is a graph having no self-loops and no multiple edges between any pair of vertices.
\item A {\em finite graph} is a graph with a finite number of vertices.
\item The {\em order} of a finite graph is the number of its vertices, $\card{\V{\Gr{G}}} = n$.
\item The {\em size} of a finite graph is the number of its edges, $\card{\E{\Gr{G}}} = m$.
\item Vertices $i, j \in \V{\Gr{G}}$ are {\em adjacent} if they are the endpoints of an edge
in $\Gr{G}$; it is denoted by $\{i,j\} \in \E{\Gr{G}}$ or $i \sim j$.
\item An {\em empty graph} is a graph without edges, so its size is equal to zero.
\item The {\em degree of a vertex} $v$ in $\Gr{G}$ is the
number of adjacent vertices to $v$ in $\Gr{G}$, denoted by $\dgr{v} = \dgr{v}(\Gr{G})$.
\item A graph is {\em regular} if all its vertices have an identical degree.
\item A {\em $d$-regular} graph is a regular graph whose all vertices have a fixed degree $d$.
\item A {\em walk} in a graph $\Gr{G}$ is a sequence of vertices in $\Gr{G}$, where
every two consecutive vertices in the sequence are adjacent in $\Gr{G}$.
\item A {\em trail} in a graph is a walk with no repeated edges.
\item A {\em path} in a graph is a walk with no repeated vertices; consequently, a path
has no repeated edges, so every path is a trail but a trail is not necessarily a path.
\item A {\em cycle} $\Cycle$ in a graph $\Gr{G}$ is obtained by adding an edge
to a path $\Path$~such that it gives a closed walk.
\item The {\em length of a path or a cycle} is equal to its number of edges. A {\em triangle}
is a cycle of length~3.
\item The length of a shortest cycle in a graph $\Gr{G}$ is called the {\em girth} of $\Gr{G}$.
\item A {\em connected graph} is a graph where every two distinct vertices are connected by a path.
\item A {\em tree} is a connected graph that has no cycles (i.e., it is a connected and {\em acyclic} graph).
\item An {\em $r$-partite graph} is a graph whose vertex set is a disjoint union of $r$ subsets
such that no two vertices in the same subset are adjacent. If $r=2$, then $\Gr{G}$ is a {\em bipartite graph}.
\item A {\em complete graph} on $n$ vertices, denoted by $\CoG{n}$, is a graph whose all $n$
distinct vertices are pairwise adjacent. Hence, $\CoG{n}$ is an $(n-1)$-regular graph of order $n$.
\item A {\em path graph} on $n$ vertices is denoted by $\PathG{n}$, and its size is equal to $n-1$.
\item A {\em cycle graph} on $n$ vertices is called an $n$-cycle, and it
is denoted by $\CG{n}$ with an integer $n \geq 3$. The order and size of $\CG{n}$ are equal to $n$,
and $\CG{n}$ is a bipartite graph if and only if $n \geq 4$ is even.
\item A {\em complete $r$-partite graph}, denoted by $\Complete_{n_1, \ldots, n_r}$ with
$n_1, \ldots n_r \in \naturals$, is an $r$-partite graph
whose vertex set is partitioned into $r$ disjoint subsets of cardinalities $n_1, \ldots, n_r$,
such that every two vertices in the same subset are not adjacent, and every two vertices in
distinct subsets are adjacent.
\item The {\em line graph} of $\Gr{G}$, denoted by $\ell(\Gr{G})$,
is a graph whose vertices are the edges in $\Gr{G}$, and two vertices are adjacent in
$\ell(\Gr{G})$ if the corresponding edges are incident in $\Gr{G}$.
\item The {\em Mycielskian graph} of a simple graph $\Gr{G}$, denoted by $\mathrm{M}(\Gr{G})$,
is defined on the vertex set $\V{\mathrm{M}(\Gr{G})} \eqdef \bigl\{\V{\Gr{G}} \times \{0,1\}\bigr\} \cup \{z_{\mathrm{M}}(\Gr{G})\}$,
with the edge set
$$\hspace*{0.6cm} \vspace*{-0.14cm} \E{\mathrm{M}(\Gr{G})} \eqdef \bigl\{\{(v,0), (w,i)\}: \{v,w\} \in \E{\Gr{G}}, i \in \{0,1\}\bigr\} \cup
\bigl\{\{z_{\mathrm{M}(\Gr{G})}, (v,1)\}: v \in \V{\Gr{G}}\bigr\}.$$
\end{itemize}
Throughout this paper, the graphs under consideration are finite, simple, and undirected.
The standard notation $\OneTo{n} \eqdef \{1, \ldots, n\}$, for every $n \in \naturals$, is also used.

\begin{definition}[Subgraphs and graph connectivity]
\label{definition:subgraphs}
A graph $\Gr{F}$ is a {\em subgraph} of a graph $\Gr{G}$, and it is
denoted by $\Gr{F} \subseteq \Gr{G}$, if
$\V{\Gr{F}} \subseteq \V{\Gr{G}}$ and $\E{\Gr{F}} \subseteq \E{\Gr{G}}$.
\begin{itemize}
\item A {\em spanning subgraph} of $\Gr{G}$ is obtained by edge deletions
from $\Gr{G}$, while its vertex set is left unchanged.
A {\em spanning tree} in $\Gr{G}$ is a spanning subgraph of $\Gr{G}$ that forms a tree.
\item An {\em induced subgraph} is obtained by removing vertices
from the original graph, followed by the deletion of their incident edges.
\item A {\em component} of $\Gr{G}$ is a maximal connected induced subgraph of $\Gr{G}$;
it is $\Gr{G}$ if connected.
\item A connected graph is said to be {\em $k$-connected}, with $k \in \naturals$, if
the induced subgraph resulting from the deletion of any subset of $k-1$ vertices
remains connected.
\end{itemize}
\end{definition}

\begin{definition}[Maximum-cut problem]
\label{definition: maximum-cut problem}
Let $\Gr{G}$ be a finite, simple, and undirected graph. Then,
\begin{itemize}
\item A {\em maximum cut} of $\Gr{G}$ is a partition of the vertex
set $\V{\Gr{G}}$ into two disjoint subsets $\set{S}$ and $\set{T}$ that maximizes the number
of edges between $\set{S}$ and $\set{T}$. The {\em max-cut} of $\Gr{G}$,
denoted by $\mathrm{mc}(\Gr{G})$, is the maximum number of edges between $\set{S}$ and $\set{T}$
for any such partition of $\V{\Gr{G}}$. In other words, $\mathrm{mc}(\Gr{G})$ is the maximum
number of edges among all bipartite spanning subgraphs of $\Gr{G}$.
\item A simple greedy algorithm shows that $\mathrm{mc}(\Gr{G}) \geq \tfrac12 \, \card{\E{\Gr{G}}}$
holds for every graph $\Gr{G}$.
\item The {\em surplus} of a graph $\Gr{G}$, denoted by $\mathrm{sp}(\Gr{G})$, is given by
$\mathrm{sp}(\Gr{G}) \eqdef \mathrm{mc}(\Gr{G}) - \tfrac12 \, \card{\E{\Gr{G}}} \geq 0$.
\end{itemize}
\end{definition}

\begin{definition}[Isomorphic graphs]
\label{definition:isomorphic graphs}
Graphs $\Gr{G}$ and $\Gr{H}$ are {\em isomorphic} if there exists a bijection
$f \colon \V{\Gr{G}} \to \V{\Gr{H}}$ (i.e., a one-to-one and onto mapping) such that
$\{i,j\} \in \E{\Gr{G}}$ if and only if $\{f(i), \, f(j)\} \in \E{\Gr{H}}$.
It is denoted by $\Gr{G} \cong \Gr{H}$, and $f$ is said to be an {\em isomorphism}
from $\Gr{G}$ to $\Gr{H}$.
\end{definition}

\begin{definition}[Complement and self-complementary graphs]
\label{definition:complement and s.c. graphs}
The {\em complement} of a graph $\Gr{G}$, denoted by $\CGr{G}$, is a graph
whose vertex set is $\V{\Gr{G}}$, and its edge set is the complement set $\overline{\E{\Gr{G}}}$.
Every vertex in $\V{\Gr{G}}$ is nonadjacent to itself in $\Gr{G}$ and $\CGr{G}$, so
$\{i,j\} \in \E{\CGr{G}}$ if and only if $\{i, j\} \notin \E{\Gr{G}}$ with $i \neq j$.
A graph $\Gr{G}$ is {\em self-complementary} if $\Gr{G} \cong \CGr{G}$ (i.e., $\Gr{G}$
is isomorphic to $\CGr{G}$).
\end{definition}

\begin{example}
\label{example: self-complementary}
It can be verified that $\PathG{4}$ and $\CG{5}$ are self-complementary graphs.
\end{example}

\begin{definition}[Disjoint union of graphs]
\label{def:disjoint_union_graphs}
Let $\Gr{G}_1, \ldots, \Gr{G}_k$ be graphs.
If the vertex sets in these graphs are not pairwise disjoint,
let $\Gr{G}'_2, \ldots, \Gr{G}'_k$ be isomorphic copies of
$\Gr{G}_2, \ldots, \Gr{G}_k$, respectively, such that none
of the graphs $\Gr{G}_1, \Gr{G}'_2, \ldots \Gr{G}'_k$ have
a vertex in common. The disjoint union of these graphs,
denoted by $\Gr{G} = \Gr{G}_1 + \ldots + \Gr{G}_k$,
is a graph whose vertex and edge sets are equal to the disjoint
unions of the vertex and edge sets of
$\Gr{G}_1, \Gr{G}'_2, \ldots, \Gr{G}'_k$
[$\Gr{G}$ is defined up to an isomorphism].
\end{definition}

\subsubsection{Graph invariants under isomorphism}
\label{subsubsection:graph invariants under isomorphism}

\begin{definition}[Independent sets and independence number]
\label{definition: independent sets}
An {\em independent set} in a graph $\Gr{G}$ is a subset of its vertices such that no two vertices
in that subset are adjacent in $\Gr{G}$ (i.e., it is an induced empty subgraph of $\Gr{G}$).
The largest number of vertices in an independent set of $\Gr{G}$ is called the {\em independence number}
of $\Gr{G}$, and it is denoted by $\indnum{\Gr{G}}$.
\end{definition}

\begin{definition}[Cliques and clique number]
\label{definition: cliques}
A {\em clique} in a graph $\Gr{G}$ is a subset of its vertices such that every two vertices
in that subset are adjacent in $\Gr{G}$ (i.e., it is an induced complete subgraph of $\Gr{G}$).
The largest number of vertices in a clique of $\Gr{G}$ is called the {\em clique number} of $\Gr{G}$,
and it is denoted by $\clnum{\Gr{G}}$. Consequently, every independent set in $\Gr{G}$ is a clique
in the complement $\CGr{G}$, and every clique in $\Gr{G}$ is an independent set in $\CGr{G}$; in
particular, it follows that $\indnum{\Gr{G}} = \clnum{\CGr{G}}$.
\end{definition}

\begin{definition}[Chromatic number]
\label{definition: chromatic number}
A {\em proper vertex coloring} in a graph $\Gr{G}$ is an assignment of colors to its vertices
such that no two adjacent vertices get the same color.
The smallest number of colors required for such a vertex coloring is called the
{\em chromatic number} of $\Gr{G}$, denoted by $\chrnum{\Gr{G}}$.
\end{definition}
The chromatic number of a graph $\Gr{G}$ is equal to the smallest number of independent sets
in $\Gr{G}$ that partition the vertex set $\V{\Gr{G}}$. Indeed, this holds since all vertices
in an independent set can be assigned the same color. Consequently, for every graph $\Gr{G}$,
\begin{align}
\label{eq:18.04.2024}
\indnum{\Gr{G}} \, \chrnum{\Gr{G}} \geq \card{\V{\Gr{G}}}.
\end{align}

\begin{proposition}[Graph invariants]
\label{proposition: graph invariants}
If $\Gr{G} \cong \Gr{H}$, then $\indnum{\Gr{G}} = \indnum{\Gr{H}}$, $\clnum{\Gr{G}} = \clnum{\Gr{H}}$,
$\chrnum{\Gr{G}} = \chrnum{\Gr{H}}$.
\end{proposition}

Fractional graph theory converts integer-based definitions of graph invariants as above
into their fractional analogues, which are shown to be useful in solving theoretical
and practical problems \cite{ScheinermanU08}.
Essential definitions and results on fractional graph theory are next provided.
\begin{definition}[Fractional graph coloring]
\label{definition: $b$-fold coloring}
A {\em fractional graph coloring} assigns a set of colors to each vertex in a graph such
that adjacent vertices have no colors in common. The following terms are used:
\begin{enumerate}
\item A {\em $b$-fold coloring} of a graph is an assignment of a set of $b$ colors to each
vertex such that adjacent vertices have no colors in common.
\item An $a \colon \hspace*{-0.1cm} b$-coloring is a $b$-fold coloring out of $a$
available colors.
It is a {\em homomorphism to the Kneser graph} $\KG{a}{b}$ since its vertices
are in one-to-one correspondence with all the $b$-element subsets of the given $a$ colors, and two
vertices are adjacent if their corresponding sets of colors are disjoint.
\item The {\em $b$-fold chromatic number} of $\Gr{G}$, denoted by $\chi_b(\Gr{G})$, is
the smallest natural number $a$ such that an $a \colon \hspace*{-0.1cm} b$-coloring of $\Gr{G}$ exists
(by Definition~\ref{definition: chromatic number}, $\chi_1(\Gr{G}) = \chrnum{\Gr{G}}$).
\item The {\em fractional chromatic number} of $\Gr{G}$, denoted by $\fchrnum{\Gr{G}}$, is
defined as
\begin{align}
\label{eq1: fchrnum}
\fchrnum{\Gr{G}} &\eqdef \inf_{b \in \naturals} \frac{\chi_b(\Gr{G})}{b} \\
\label{eq2: fchrnum}
&= \lim_{b \to \infty} \frac{\chi_b(\Gr{G})}{b},
\end{align}
where \eqref{eq2: fchrnum} holds since $\{\chi_b(\Gr{G})\}_{b=1}^{\infty}$
is a subadditive sequence, i.e.,
\begin{align}
\label{eq: subadditivity}
\chi_{b_1+b_2}(\Gr{G}) \leq \chi_{b_1}(\Gr{G}) + \chi_{b_2}(\Gr{G}), \; \forall \, b_1, b_2 \in \naturals,
\end{align}
and, by Fekete's lemma, if $\{x_k\}_{k \geq 1}$ is a nonnegative sequence such that
$x_{m+n} \leq x_m +x_n$ for all $m,n \in \naturals$, then the equality
$\underset{n \to \infty}{\lim} \, \dfrac{x_n}{n} = \underset{n \in \naturals}{\inf} \, \dfrac{x_n}{n}$ holds.
\end{enumerate}
\end{definition}

\begin{theorem}[The fractional chromatic number of a graph]
\label{theorem: fractional chromatic number}
Let $\Gr{G}$ be a graph, and let $\indset{\Gr{G}}$ and $\indsetmax{\Gr{G}}$ denote, respectively, the sets of all
the independent sets and maximal independent sets in $\Gr{G}$.
Then, the fractional chromatic number of $\Gr{G}$ is the solution of the linear programming (LP) problem
\begin{eqnarray}
\label{eq: fractional chromatic number}
\mbox{\fbox{$
\begin{array}{l}
\text{minimize} \; \; \underset{\set{I} \in \indset{\Gr{G}}}{\sum} x_{\set{I}} \\
\text{subject to} \\[0.1cm]
\begin{dcases}
\, \forall \, v \in \V{\Gr{G}}, \quad \underset{\set{I} \in \indset{\Gr{G}}: \, v \in \set{I}}{\sum} \, x_{\set{I}} \geq 1, \\
\, \forall \, \set{I} \in \indset{\Gr{G}}, \quad x_{\set{I}} \in [0,1].
\end{dcases}
\end{array}$}}
\end{eqnarray}
Furthermore,
\begin{enumerate}[(1)]
\item
The solution of the LP problem \eqref{eq: fractional chromatic number} is not affected
by restricting $\indset{\Gr{G}}$ to $\indsetmax{\Gr{G}}$,
which requires that only maximal independent sets in $\Gr{G}$ can get positive weights.
\item The fractional chromatic number of $\Gr{G}$ is a rational number, i.e., $\fchrnum{\Gr{G}} \in \rationals$.
\item The following inequality holds by \cite{Lovasz75}{\em :}
\begin{align}
\label{eq: Lovasz75}
\frac{1}{1 + \ln \indnum{\Gr{G}}} \leq \frac{\fchrnum{\Gr{G}}}{\chrnum{\Gr{G}}} \leq 1.
\end{align}
\end{enumerate}
\end{theorem}

The LP problem \eqref{eq: fractional chromatic number} serves as a relaxation of the linear integer-programming problem
aimed at determining the chromatic number $\chrnum{\Gr{G}}$. The chromatic number represents the smallest number
of (maximal) independent sets required to partition the vertex set of graph $\Gr{G}$. This relaxation is justified
by the observation that all vertices belonging to an independent set in $\Gr{G}$ can be assigned identical colors.
Consequently, the optimization variables $\{x_{\set{I}}\}$ in \eqref{eq: fractional chromatic number},
corresponding to each independent set $\set{I} \in \indset{\Gr{G}}$,
are relaxed to lie within the interval $[0,1]$ rather than being binary variables. This relaxation defines the fractional
chromatic number $\fchrnum{\Gr{G}}$. Notably, the minimization in the LP problem \eqref{eq: fractional chromatic number}
is achieved by constraining the independent sets in $\Gr{G}$ to be maximal.
The computational complexity of determining the fractional chromatic number of a graph is NP-hard, as demonstrated in Section~7
of \cite{GrotschelLS81}. Additionally, it is noteworthy that the number of maximal independent sets in a graph of fixed order
$n$ grows exponentially with $n$. Specifically, Theorem~1 in \cite{MoonM65} bounds this number between $3^{\lfloor n/3 \rfloor}$
and $2 \cdot 3^{\lfloor n/3 \rfloor}$.

The dual LP of \eqref{eq: fractional chromatic number}
is a maximization of the sum of the nonnegative weights that are assigned to the vertices in $\Gr{G}$ such
that the total weight of the vertices in every maximal independent set in $\Gr{G}$ is at most~1.
It is given by the LP problem
\begin{eqnarray}
\label{eq: fractional clique number}
\mbox{\fbox{$
\begin{array}{l}
\text{maximize} \; \; \underset{v \in \V{\Gr{G}}}{\sum} x_v \\[-0.1cm]
\text{subject to} \\[0.1cm]
\begin{cases}
\forall \, \set{I} \in \indsetmax{\Gr{G}}, \quad \underset{v \in \set{I}}{\sum} x_v \leq 1, \\
\forall \, v \in \V{\Gr{G}}, \quad x_v \geq 0.
\end{cases}
\end{array}$}}
\end{eqnarray}
The dual LP problem \eqref{eq: fractional clique number} forms a relaxation of the integer programming problem for
the clique number $\clnum{\Gr{G}}$. It is defined to be the {\em fractional clique number} of the graph $\Gr{G}$,
denoted by $\fclnum{\Gr{G}}$. By the strong duality in linear programming, which states that the optimal values
of the primal and dual LP problems are identical provided that these LP problems are feasible, it follows that
\begin{align}
\label{eq1:14.10.23}
\fclnum{\Gr{G}} = \fchrnum{\Gr{G}}.
\end{align}

\begin{definition}[Fractional independence number]
\label{definition: fractional independence number}
The {\em fractional independence number} of a graph $\Gr{G}$ is defined to be the fractional
clique number of the complement $\CGr{G}$, i.e.,
\begin{align}
\label{eq2:14.10.23}
\findnum{\Gr{G}} \eqdef \fclnum{\CGr{G}} = \fchrnum{\CGr{G}}.
\end{align}
\end{definition}

\begin{corollary}
\label{corollary: rational fractional graph invariants}
The fractional independence, clique, and chromatic numbers of every graph $\Gr{G}$ are rational numbers.
These are graph invariants, i.e., if $\Gr{G} \cong \Gr{H}$, then
\begin{align}
\label{eq3:14.10.23}
\findnum{\Gr{G}} = \findnum{\Gr{H}}, \quad \fclnum{\Gr{H}} = \fclnum{\Gr{G}} = \fchrnum{\Gr{G}} = \fchrnum{\Gr{H}}.
\end{align}
\end{corollary}

\subsubsection{Matrices Associated with Graphs}
\label{subsubsection: Matrices associated with graphs}

Spectral graph theory delves into relations between the structure of a graph and the eigenvalues
of matrices that are associated with the graph. This field constitutes significant
aspects of algebraic graph theory, as studied in textbooks such as
\cite{BrouwerH12,ChartrandLZ15,CioabaM22,CvetkovicRS09,GodsilR,Lovasz19,Nica18,Stanic15}.
Next, we provide essential background in spectral graph theory for this paper.
The interested reader is further referred to a recent survey paper \cite{LiuN23}, which presents
twenty open problems in spectral graph theory, along with accompanying historical notes.

\begin{definition}[Adjacency matrix]
\label{definition: adjacency matrix}
Let $\Gr{G}$ be a simple undirected graph on $n$ vertices. The {\em adjacency matrix} of $\Gr{G}$,
denoted by $\A = \A(\Gr{G})$, is an $n \times n$ symmetric matrix $\A = (\mathrm{A}_{i,j})$
where $\mathrm{A}_{i,j} = 1$ if $\{i,j\} \in \E{\Gr{G}}$, and $\mathrm{A}_{i,j}=0$ otherwise
(so, the entries in the principal diagonal of $\A$ are zeros).
\end{definition}

Let $d_i$ denote, for $i \in \OneTo{n}$, the degree of the vertex $i \in \V{\Gr{G}}$,
and let $\D = \D(\Gr{G})$ be the diagonal matrix with the diagonal entries $d_1, \ldots, d_n$
(hence, for a $d$-regular graph on $n$ vertices, $\D = d \I{n}$).
\begin{definition}[Laplacian and signless Laplacian matrices]
\label{definition: Laplacian and Signless Laplacian Matrices}
Let $\Gr{G}$ be a simple undirected graph on $n$ vertices with an adjacency matrix $\A$ and degree
matrix $\D$. Then, the {\em Laplacian} and {\em signless Laplacian} matrices of $\Gr{G}$
are the symmetric matrices that are, respectively, given by
\begin{align}
\label{eq: laplacian}
& \LM \eqdef \D - \A, \\
\label{eq: signless laplacian}
& \Q \eqdef \D + \A.
\end{align}
\end{definition}

\begin{definition}[Normalized Laplacian matrix]
\label{definition: Normalized Laplacian matrix}
The normalized Laplacian matrix ${\bf{\mathcal{L}}} = {\bf{\mathcal{L}}}(\Gr{G})$ of a simple undirected graph
$\Gr{G}$ on $n$ vertices, with the adjacency and degree matrices $\A$ and $\D$, respectively,
is defined to be
\begin{align}
\label{eq: Normalized Laplacian matrix}
{\bf{\mathcal{L}}} \eqdef \D^{-\frac12} \LM \D^{-\frac12}
= \I{n}- \D^{-\frac12} \A \D^{-\frac12}.
\end{align}
The entries of the normalized Laplacian matrix ${\bf{\mathcal{L}}} = (\mathcal{L}_{i,j})$ are given by
\begin{align}
\label{eq2: Normalized Laplacian matrix}
\mathcal{L}_{i,j} =
\begin{dcases}
\begin{array}{cl}
1, \quad & \mbox{if $i=j$ and $d_i \neq 0$,} \\[0.2cm]
-\dfrac{1}{\sqrt{d_i d_j}}, \quad & \mbox{if $i \neq j$ and $\{i,j\} \in \E{\Gr{G}}$}, \\[0.5cm]
0, \quad & \mbox{otherwise},
\end{array}
\end{dcases}
\end{align}
with the convention that if $i \in \V{\Gr{G}}$ is an isolated vertex in $\Gr{G}$ (i.e., $d_i=0$), then $d_i^{-\frac12} = 0$.
\end{definition}

The next theorems show that the spectra of a graph with respect to the above four matrices
(i.e., the sets of eigenvalues of these matrices) give information about the structure of the graph.

\begin{theorem}[Number of walks of a given length]
\label{thm: number of walks of a given length}
Let $\Gr{G} = (\Vertex, \Edge)$ be a finite, simple, and undirected graph with an adjacency matrix $\A = \A(\Gr{G})$,
let $i,j \in \Vertex$, and let $\ell \in \naturals$. Then, the number
of walks of length $\ell$ in $\Gr{G}$, with the fixed endpoints $i$ and $j$, is equal to $(\A^\ell)_{i,j}$.
\end{theorem}

\begin{corollary}[Number of closed walks of a given length]
\label{corollary: Number of Closed Walks of a Given Length}
Let $\Gr{G} = (\Vertex, \Edge)$ be a simple undirected graph on $n$ vertices with an adjacency matrix
$\A = \A(\Gr{G})$, and let its spectrum (with respect to $\A$) be given by $\{\lambda_j\}_{j=1}^n$. Then,
for all $\ell \in \naturals$, the number of closed walks of length $\ell$ in $\Gr{G}$ is equal to
$\sum_{j=1}^n \lambda_j^{\ell}$.
\end{corollary}

A graph $\Gr{G}$ is bipartite if and only if it does not include odd cycles. In light of
Corollary~\ref{corollary: Number of Closed Walks of a Given Length}, a graph $\Gr{G}$ is bipartite
if and only if its spectrum $\{\lambda_j\}_{j=1}^n$ is symmetric around zero.

\begin{corollary}[Number of edges and triangles in a graph]
\label{corollary: number of edges and triangles in a graph}
Let $\Gr{G}$ be a simple undirected graph with $n = \card{\V{\Gr{G}}}$ vertices, $e = \card{\E{\Gr{G}}}$ edges, and $t_3$
triangles. Let $\A = \A(\Gr{G})$ be the adjacency matrix of $\Gr{G}$, and let $\{\lambda_j\}_{j=1}^n$ be its spectrum. Then,
\begin{align}
\label{eq:14.09.23-c1}
& \sum_{j=1}^n \lambda_j = \mathrm{tr}(\A) = 0,  \\
\label{eq:14.09.23-c2}
& \sum_{j=1}^n \lambda_j^2 = \mathrm{tr}(\A^2) = 2 e, \\
\label{eq:14.09.23-c3}
& \sum_{j=1}^n \lambda_j^3 = \mathrm{tr}(\A^3) = 6 t_3.
\end{align}
\end{corollary}

\begin{theorem}[On the Laplacian matrix]
\label{theorem: On the Laplacian matrix of a graph}
Let $\Gr{G}$ be a finite, simple, and undirected graph, and let $\LM$ be
the Laplacian matrix of $\Gr{G}$. Then,
\begin{enumerate}
\item \label{Item 1: Laplacian matrix of a graph}
The matrix $\LM$ is a positive semidefinite matrix.
\item \label{Item 2: Laplacian matrix of a graph}
The smallest eigenvalue of $\, \LM$ is equal to zero, and its multiplicity is equal
to the number of components in $\Gr{G}$.
\item \label{Item 3: Laplacian matrix of a graph}
The size of $\Gr{G}$, $\card{\E{\Gr{G}}}$, is equal to one-half the sum of the eigenvalues of $\, \LM$ (with multiplicities).
\end{enumerate}
\end{theorem}

\begin{theorem}[On the signless Laplacian matrix]
\label{theorem: On the signless Laplacian matrix of a graph}
Let $\Gr{G}$ be a finite, simple, and undirected graph, and let $\Q$ be
the signless Laplacian matrix of $\Gr{G}$. Then,
\begin{enumerate}
\item \label{Item 1: signless Laplacian matrix of a graph}
The matrix $\Q$ is a positive semidefinite matrix.
\item \label{Item 2: signless Laplacian matrix of a graph}
The least eigenvalue of the matrix $\Q$ is equal to zero if and only if $\Gr{G}$ is a bipartite graph.
\item \label{Item 3: signless Laplacian matrix of a graph}
The multiplicity of 0 as an eigenvalue of $\Q$ is equal to the number of bipartite components in $\Gr{G}$.
\item \label{Item 4: signless Laplacian matrix of a graph}
The size of $\Gr{G}$ is equal to one-half the sum of the eigenvalues of~$\Q$ (with multiplicities).
\end{enumerate}
\end{theorem}

\begin{theorem}[On the normalized Laplacian matrix]
\label{theorem: On the normalized Laplacian matrix of a graph}
Let $\Gr{G}$ be a finite, simple, and undirected graph, and let ${\bf{\mathcal{L}}}$ be
the normalized Laplacian matrix of $\Gr{G}$. Then,
\begin{enumerate}
\item \label{Item 1: normalized Laplacian matrix of a graph}
The eigenvalues of ${\bf{\mathcal{L}}}$ lie in the interval $[0,2]$.
\item \label{Item 2: normalized Laplacian matrix of a graph}
The number of components in $\Gr{G}$ is equal to the multiplicity of~0 as an eigenvalue of ${\bf{\mathcal{L}}}$.
\item \label{Item 3: normalized Laplacian matrix of a graph}
The number of the bipartite components in $\Gr{G}$ is equal to the multiplicity of~2 as an eigenvalue of~${\bf{\mathcal{L}}}$.
\end{enumerate}
\end{theorem}

\begin{definition}[Characteristic polynomial]
\label{definition: Characteristic Polynomial}
The characteristic polynomial of an $n \times n$ matrix ${\bf{M}}$ is given by
$f_{{\bf{M}}}(x) \eqdef \det(x \I{n} - {\bf{M}})$,
where $\I{n}$ denotes the identity matrix of order $n$. Additionally,
\end{definition}
\begin{enumerate}
\item  \label{Item 1: Characteristic Polynomial}
$f_{X}(\cdot)$ denotes the $X$-characteristic polynomial of a graph $\Gr{G}$, with $X \in \{A, L, Q, \mathcal{L}\}$.
\item  \label{Item 2: Characteristic Polynomial}
The zeros of the $X$-characteristic polynomial of a graph $\Gr{G}$ are the $X$-eigenvalues of $\Gr{G}$.
\item  \label{Item 3: Characteristic Polynomial}
The collection of $X$-eigenvalues of $\Gr{G}$, including multiplicities, is the {\em $X$-spectrum} of $\Gr{G}$.
\end{enumerate}
Let $\Gr{G}$ be a graph on $n$ vertices, and let
\begin{align}
\label{eq2:26.09.23}
& \Eigval{1}{\Gr{G}} \geq \Eigval{2}{\Gr{G}} \geq \ldots \geq \Eigval{n}{\Gr{G}}, \\
\label{eq3:26.09.23}
& \mu_1(\Gr{G}) \leq \mu_2(\Gr{G}) \leq \ldots \leq \mu_n(\Gr{G}), \\
\label{eq4:26.09.23}
& \nu_1(\Gr{G}) \geq \nu_2(\Gr{G}) \geq \ldots \geq \nu_n(\Gr{G}), \\
\label{eq5:26.09.23}
& \delta_1(\Gr{G}) \leq \delta_2(\Gr{G}) \leq \ldots \leq \delta_n(\Gr{G})
\end{align}
be, respectively, the $\{A, L, Q, \mathcal{L}\}$-eigenvalues of $\Gr{G}$ (including multiplicities). Then,
\begin{align}
\label{eq6:11.10.23}
& |\lambda_\ell(\Gr{G})| \leq \lambda_1(\Gr{G}), \quad \forall \, \ell \in \{2, \ldots, n\},\\
\label{eq7:11.10.23}
& \mu_1(\Gr{G}) = 0, \\
\label{eq8:11.10.23}
& \nu_n(\Gr{G}) \geq 0,  \\
\label{eq9:11.10.23}
& \delta_1(\Gr{G}) = 0, \\
\label{eq9b:11.10.23}
& \delta_n(\Gr{G}) \leq 2,
\end{align}
with equality in \eqref{eq9b:11.10.23} if and only if $\Gr{G}$ contains a bipartite component. The multiplicity
of~2 as an eigenvalue of the normalized Laplacian matrix is equal to the number of bipartite components of $\Gr{G}$
(see Item~\ref{Item 3: normalized Laplacian matrix of a graph} of Theorem~\ref{theorem: On the normalized Laplacian matrix of a graph}).

\begin{theorem}[The number of spanning trees]
\label{theorem: number of spanning trees}
The number of spanning trees in a graph $\Gr{G}$ on $n$ vertices is determined by the eigenvalues
of the Laplacian matrix, and it is equal to $\frac{1}{n} \overset{n}{\underset{\ell=2}{\prod}} \, \mu_\ell(\Gr{G})$.
\end{theorem}

\begin{remark}[The number of spanning trees]
\label{remark: number of spanning trees}
There are no spanning trees in a disconnected graph, which is consistent with
Theorem~\ref{theorem: number of spanning trees} and the fact that if $\Gr{G}$
is disconnected, then the multiplicity of 0 as the smallest eigenvalue
of the Laplacian matrix of $\Gr{G}$ is at least~2 (by Item~\ref{Item 2: Laplacian matrix of a graph}
of Theorem~\ref{theorem: On the Laplacian matrix of a graph}). Moreover, Cayley's
formula, which states that the number of trees on $n$ vertices is equal to $n^{n-2}$
(see, e.g., proofs of this renowned formula in Chapter~33 of \cite{AignerZ18}),
can be derived from Theorem~\ref{theorem: number of spanning trees}.  This is facilitated
by the fact that $\mu_1(\CoG{n}) = 0$ and $\mu_2(\CoG{n}) = \ldots = \mu_n(\CoG{n}) = n$.
\end{remark}

The primary focus in spectral graph theory revolves around examining the spectra of graphs with respect
to their adjacency and Laplacian matrices. Furthermore, \cite{Butler16} and \cite{CvetkovicRS07} provide
surveys on the properties of the spectra of finite graphs with respect to their normalized Laplacian
and signless Laplacian matrices, respectively. An intriguing connection is next given between the
$A$-eigenvalues of a graph's line graph and the $Q$-eigenvalues of the original graph,
as presented in Proposition~1.4.1 of \cite{BrouwerH12}.

\begin{theorem}[$A$-eigenvalues of a line graph]
\label{theorem: A-eigenvalues of a line graph}
Let $\Gr{G}$ be a graph with $n$ vertices and $m$ edges, and let $\ell(\Gr{G})$
be the line graph of $\Gr{G}$. The $A$-eigenvalues of $\ell(\Gr{G})$ satisfy the
following properties:
\begin{enumerate}
\item \label{Item 1: A-eigenvalues of a line graph}
If $m \geq n$, then
$\lambda_i\bigl(\ell(\Gr{G})\bigr) = \nu_i(\Gr{G})-2$ for all $i \in \OneTo{n}$, and
$\lambda_i\bigl(\ell(\Gr{G})\bigr) = -2$ for all $i \in \{n+1, \ldots, m\}$.
\item \label{Item 2: A-eigenvalues of a line graph}
Otherwise, if $n>m$, then it is needed to remove $n-m$ numbers of $(-2)$ from the
numbers listed first to get the $A$-spectrum of $\ell(\Gr{G})$.
\end{enumerate}
\end{theorem}

\begin{remark}[Size of a graph]
\label{remark: size of a graph}
In contrast to the $A$-spectrum, $L$-spectrum, and $Q$-spectrum of a graph $\Gr{G}$,
the $\mathcal{L}$-spectrum of $\Gr{G}$ does not determine the size of $\Gr{G}$ (see
\eqref{eq:14.09.23-c2}, Item~\ref{Item 3: Laplacian matrix of a graph} in Theorem~\ref{theorem: On the Laplacian matrix of a graph},
and Item~\ref{Item 4: signless Laplacian matrix of a graph} in Theorem~\ref{theorem: On the signless Laplacian matrix of a graph}
for the first three spectra).
\end{remark}

Unless explicitly specified, the spectrum refers to the $A$-spectrum of a graph \eqref{eq2:26.09.23}.

\subsubsection{Vertex- and edge-transitive graphs}
\label{subsubsection: vertex- and edge-transitive graphs}

Vertex- and edge-transitivity, defined as follows, play an important role in characterizing graphs.

\begin{definition}[Automorphism]
\label{definition:automorphism}
An {\em automorphism} of a graph $\Gr{G}$ is an isomorphism from $\Gr{G}$ to itself.
\end{definition}

\begin{definition}[Vertex-transitivity]
\label{definition:vertex-transitive graphs}
A graph $\Gr{G}$ is said to be {\em vertex-transitive} if, for every two vertices $i, j \in \V{\Gr{G}}$,
there is an automorphism $f \colon \V{\Gr{G}} \to \V{\Gr{G}}$ such that $f(i) = j$.
\end{definition}

\begin{definition}[Edge-transitivity]
\label{definition:edge-transitive graphs}
A graph $\Gr{G}$ is {\em edge-transitive} if, for every two edges $e_1, e_2 \in \E{\Gr{G}}$,
there is an automorphism $f \colon \V{\Gr{G}} \to \V{\Gr{G}}$ that maps the endpoints of $e_1$ to
the endpoints of $e_2$.
\end{definition}

\begin{example}
\label{example: vertex- and edge-transitivity}
A vertex-transitive graph is necessarily regular, but the converse is false.
For example, the Frucht graph is a 3-regular graph on 12 vertices that lacks
vertex-transitivity (and it also lacks edge-transitivity) \cite{SageMath}.
In contrast, an edge-transitive graph is not necessarily regular. As an example, consider a
star graph on any number $n \geq 3$ of vertices, which is an irregular and edge-transitive graph.
\end{example}

\subsubsection{Strongly regular graphs}
\label{subsubsection: strongly regular graphs}
Strongly regular graphs form an important subfamily of the class of regular graphs \cite{BrouwerM22}.
Essential definitions and properties of these graphs are next given.

\begin{definition}[Strongly regular graphs]
\label{definition: strongly regular graphs}
A regular graph $\Gr{G}$ that is neither complete nor empty is called a
{\em strongly regular} graph, with parameters $\srg{n}{d}{\lambda}{\mu}$
where $\lambda$ and $\mu$ are nonnegative integers, if the following conditions hold:
\begin{enumerate}
\item \label{Item 1 - definition of SRG}
$\Gr{G}$ is a $d$-regular graph on $n$ vertices.
\item \label{Item 2 - definition of SRG}
Every two adjacent vertices in $\Gr{G}$ have exactly $\lambda$ common neighbors.
\item \label{Item 3 - definition of SRG}
Every two distinct and nonadjacent vertices in $\Gr{G}$ have exactly $\mu$ common neighbors.
\end{enumerate}
\end{definition}

\begin{definition}[Primitive and imprimitive strongly regular graphs]
\label{definition:primitive and imprimitive SRGs}
A strongly regular graph $\Gr{G}$ is called {\em primitive} if $\Gr{G}$ and
its complement $\CGr{G}$ are connected graphs. Otherwise, it is called an
{\em imprimitive} strongly regular graph.
\end{definition}

\begin{theorem}[On the imprimitive strongly regular graphs]
\label{thm: disconnected srg}
Let $\Gr{G}$ be a strongly regular graph with parameters $\srg{n}{d}{\lambda}{\mu}$.
Then, the following conditions are equivalent:
\begin{enumerate}
\item \label{Item 1 - imprimitive SRG}
$\Gr{G}$ is a disconnected graph;
\item \label{Item 2 - imprimitive SRG}
$\mu=0$;
\item \label{Item 3 - imprimitive SRG}
$\lambda = d-1$;
\item \label{Item 4 - imprimitive SRG}
$\Gr{G}$ is isomorphic to $m$ disjoint copies of $\CoG{d+1} \,$ for some $m \geq 2$.
\end{enumerate}
\end{theorem}

\begin{theorem}[On the parameters of strongly regular graphs]
\label{theorem: parameters of srg}
The four parameters of strongly regular graphs satisfy the following properties:
\begin{enumerate}
\item \label{item: complement of SRG}
The complement of a strongly regular graph with parameters $\SRG(n,d,\lambda,\mu)$
is a strongly regular graph with parameters
$\SRG(n_{\mathrm{c}},d_{\mathrm{c}},\lambda_{\mathrm{c}},\mu_{\mathrm{c}})$ where
\begin{align}
\label{eq: n-complement}
& n_{\mathrm{c}} = n, \\
\label{eq: d-complement}
& d_{\mathrm{c}} = n-d-1, \\
\label{eq: lambda-complement}
& \lambda_{\mathrm{c}} = n-2d+\mu-2, \\
\label{eq: mu-complement}
& \mu_{\mathrm{c}} = n-2d+\lambda.
\end{align}
\item \label{item: relation between parameters of SRG}
The four parameters of a strongly regular graph $\SRG(n,d,\lambda,\mu)$
satisfy the equality
\begin{align}
\label{eq:relation-srg}
(n-d-1) \, \mu = d \, (d-\lambda-1).
\end{align}
\end{enumerate}
\end{theorem}

\begin{theorem}[The eigenvalues of strongly regular graphs]
\label{theorem: eigenvalues of srg}
The following spectral properties are satisfied by the family of strongly regular graphs:
\begin{enumerate}
\item \label{Item 1: eigenvalues of srg}
A strongly regular graph has three distinct eigenvalues.
\item \label{Item 2: eigenvalues of srg}
Let $\Gr{G}$ be a connected strongly regular graph, and let its parameters be $\SRG(n,d,\lambda,\mu)$.
Then, the largest eigenvalue of its adjacency matrix is $\Eigval{1}{\Gr{G}} = d$ with multiplicity~1,
and the other two distinct eigenvalues of its adjacency matrix are given by
\begin{align}
\label{eigs-SRG}
p_{1,2} = \tfrac12 \, \Biggl( \lambda - \mu \pm \sqrt{ (\lambda-\mu)^2 + 4(d-\mu) } \, \Biggr),
\end{align}
with the respective multiplicities
\begin{align}
\label{eig-multiplicities-SRG}
m_{1,2} = \tfrac12 \, \Biggl( n-1 \mp \frac{2d+(n-1)(\lambda-\mu)}{\sqrt{(\lambda-\mu)^2+4(d-\mu)}} \, \Biggr).
\end{align}
\item \label{Item 3: eigenvalues of srg}
A connected regular graph with exactly three distinct eigenvalues is strongly regular.
\end{enumerate}
\end{theorem}

\begin{definition}[Conference graphs]
\label{definition: conference graphs}
A conference graph on $n$ vertices is a strongly regular graph with the parameters
$\srg{n}{\tfrac12(n-1)}{\tfrac14(n-5)}{\tfrac14(n-1)}$ (with $n \geq 5$).
\end{definition}
By Item~\ref{item: complement of SRG} of Theorem~\ref{theorem: parameters of srg}, if $\Gr{G}$ is a
conference graph on $n$ vertices, then so is its complement $\CGr{G}$; it is, however, not necessarily
self-complementary. By Theorem~\ref{theorem: eigenvalues of srg}, the distinct eigenvalues of
the adjacency matrix of $\Gr{G}$ are given by $\tfrac12 (n-1)$, $\tfrac12 (\hspace*{-0.1cm}
\sqrt{n}-1)$, and $-\tfrac12 (\hspace*{-0.1cm} \sqrt{n}+1)$ with multiplicities
$1, \tfrac12 (n-1)$, and $\tfrac12 (n-1)$, respectively. A conference graph is a primitive
strongly regular graph.

\subsection{The Shannon Capacity of Graphs}
\label{subsection: the Shannon capacity of a graph}

The concept of the Shannon capacity of a graph $\Gr{G}$ was introduced in \cite{Shannon56}
to consider the largest information rate that can be achieved with zero-error communication.
A discrete memoryless channel consists of a finite {\em input set} $\set{X}$, a
finite or a countably infinite {\em output set} $\set{Y}$, and a nonempty {\em fan-out set}
$\set{S}_x \subseteq \set{Y}$, for every input $x \in \set{X}$. The set $\set{S}_x$
is the set of all possible output symbols that can be received at the channel output
with positive probability, in each channel use, if the transmitted symbol is $x \in \set{X}$.
The study of the maximum amount (rate) of information that a channel can communicate without
error is of great interest, and it turns to be a problem in graph theory. To that end, the
channel is represented by a {\em confusion graph} $\Gr{G}$ whose set of vertices $\V{\Gr{G}}$
represents the elements of the input set $\set{X}$, and its set of edges $\E{\Gr{G}}$ is
constructed such that any two distinct vertices in $\Gr{G}$ are adjacent if and only if the
two input symbols are not distinguishable by the channel, so
\begin{eqnarray}
\label{eq5:11.10.23}
\V{\Gr{G}} = \set{X}, \quad
\E{\Gr{G}} = \bigl\{ \{x, x'\}: \; x, x' \in \set{X}, \; x \neq x', \; \set{S}_x \cap
\set{S}_{x'} \neq \varnothing \bigr\}.
\end{eqnarray}
(Two distinct input symbols are not distinguishable by the channel if and only if they can result
in the same output symbol with some positive probability).
The largest number of input symbols that a channel can communicate without error in a single use
is $\indnum{\Gr{G}}$ (the independence number of the graph $\Gr{G}$). It is obtained as follows:
The sender and the receiver agree in advance on an independent set $\set{I}$ of a
maximum size $\indnum{\Gr{G}}$, the sender transmits only input symbols in $\set{I}$, and
every received output is in the fan-out set of exactly one input symbol in $\set{I}$,
so the receiver can correctly determine the transmitted input symbol.
Before we proceed, it is required to define the notion of a {\em strong product} of graphs.
\begin{definition}[Strong product of graphs]
\label{def:strong product of graphs}
Let $\Gr{G}$ and $\Gr{H}$ be graphs.
The strong product $\Gr{G} \boxtimes \Gr{H}$ is a graph with
a vertex set $\V{\Gr{G} \boxtimes \Gr{H}} = \V{\Gr{G}} \times \V{\Gr{H}}$ (the Cartesian product),
and distinct vertices $(g, h)$ and $(g', h')$ are adjacent in $\Gr{G} \boxtimes \Gr{H}$ if
and only if one of the following three conditions hold: \newline
(1) $g = g'$ and $\{h, h'\} \in \E{\Gr{H}}$, \quad (2) $\{g, g'\} \in \E{\Gr{G}}$ and $h = h'$,
\quad (3) $\{g, g'\} \in \E{\Gr{G}}$ and $\{h, h'\} \in \E{\Gr{H}}$.
\newline
Strong products are therefore commutative and associative up to graph isomorphisms.
\end{definition}

Consider the transmission of $k$-length strings over a channel that is
used $k \geq 1$ times. The sender transmits a sequence $x_1 \ldots x_k$, and the receiver
gets a sequence $y_1 \ldots y_k$, where $y_i \in \set{S}_{x_i}$ for all $i \in \OneTo{k}$.
The $k$ channel uses are viewed as a single use of an extended channel whose input set
is $\set{X}^k$, its output set is $\set{Y}^k$, and the fan-out set of
$(x_1, \ldots, x_k) \in \set{X}^k$ is the Cartesian product $\set{S}_{x_1} \times \ldots \times \set{S}_{x_k}$.
The {\em $k$-th confusion graph}, which refers to the confusion graph of the extended channel, is the {\em $k$-fold
strong power} of $\Gr{G}$ that is given by $\Gr{G}^{\boxtimes \, k} \eqdef \Gr{G} \boxtimes \ldots \boxtimes \Gr{G}$.
The largest number of $k$-length input strings that are distinguishable by the channel
is then equal to $\indnum{\Gr{G}^{\boxtimes \, k}}$. Using distinguishable strings asserts
error-free communication, so the largest achievable {\em information rate per symbol} is given by
\begin{eqnarray}
\label{eq2:11.10.23}
\frac1k \, \log \indnum{\Gr{G}^{\boxtimes \, k}} = \log \sqrt[k]{\indnum{\Gr{G}^{\boxtimes \, k}}}, \quad k \in \naturals.
\end{eqnarray}
Under the assumption that the length $k$ of the input strings can be made arbitrarily large,
the largest information rate per symbol is given by the supremum of
the right-hand side of \eqref{eq2:11.10.23} over all $k \in \naturals$. The resulting amount is defined to be the
(logarithm of the) {\em Shannon capacity of the graph $\Gr{G}$} or also the {\em zero-error Shannon capacity of
the channel}. By supremizing the exponent of the right-hand side of \eqref{eq2:11.10.23} over all
$k \in \naturals$, the Shannon capacity of a graph $\Gr{G}$ is given by
\begin{align}
\label{eq1:graph capacity}
\Theta(\Gr{G}) &\eqdef \underset{k \in \naturals}{\sup}
\sqrt[k]{\indnum{\Gr{G}^{\boxtimes \, k}}} \\
\label{eq2:graph capacity}
&= \underset{k \to \infty}{\lim} \sqrt[k]{\indnum{\Gr{G}^{\boxtimes \, k}}},
\end{align}
where \eqref{eq2:graph capacity} holds by Fekete's lemma. Indeed, since
$\indnum{\Gr{G} \boxtimes \Gr{H}} \geq \indnum{\Gr{G}} \, \indnum{\Gr{H}}$ holds for every two
graphs $\Gr{G}$ and $\Gr{H}$, it follows that the inequality $\indnum{\Gr{G}^{\boxtimes \, (k_1 + k_2)}}
\geq \indnum{\Gr{G}^{\boxtimes \, k_1}} \; \indnum{\Gr{G}^{\boxtimes \, k_2}}$ holds for every
graph $\Gr{G}$ and for all $k_1, k_2 \in \naturals$.
This validates the application of Fekete's lemma in \eqref{eq2:graph capacity}.
The supremum on the right-hand side of \eqref{eq1:graph capacity} is not necessarily achieved
by a finite $k \in \naturals$ \cite{GuoW90,XuR13}, and the convergence to the Shannon capacity
on the right-hand side of \eqref{eq2:graph capacity} exhibits erratic behavior \cite{AlonL06}.
Moreover, several notions of computability of the Shannon capacity of a graph and the
zero-error capacity of discrete memoryless channels, known for their computational difficulty,
have been addressed in \cite{AlonL06,BocheD20,BocheD21}, yet certain aspects of their
computability remain unresolved. Except for some specific graphs, the Shannon capacity of a graph
remains undetermined, even for graphs with a very simple structure. For instance, the Shannon
capacity of the cycle graph $\CG{n}$ or its complement, with odd $n \geq 7$, remains unresolved,
with only bounds and an asymptotic limit theorem currently available
\cite{BohmanH03,Bohman03a,Bohman05b,BohmanHN09,BohmanHN13,Hales73,Lovasz79_IT,PolakS19,SonnemannK74,Zhu24}
(for other graph analogues of odd-cycle graphs, see \cite{HellR82}).
Specifically, the Shannon capacity of the 7-cycle graph $\CG{7}$ is bounded as follows \cite{Lovasz79_IT,PolakS19}:
\begin{align}
3.2578 < \sqrt[5]{367} \leq \Theta(\CG{7}) \leq \dfrac{7}{1+ \sec \frac{\pi}{7}} < 3.3177.
\end{align}
The Shannon capacity of disjoint unions, strong products of graphs, and the complement of the categorical
(tensor) product of graphs is studied in \cite{Alon98,KeevashL16,Schrijver23,Shannon56,Simonyi21,WigdersonZ23}.

\subsection{Orthonormal Representations of Graphs}
\label{subsection: Orthonormal Representations of Graphs}
In \cite{Lovasz79_IT}, Lov\'{a}sz introduced the concept of orthonormal representations of graphs
to derive an upper bound on the Shannon capacity of graphs in information theory.
Orthonormal representations of graphs are related to several fundamental graph properties.
That concept is defined and exemplified as follows.

\begin{definition}[Orthogonal and orthonormal representations of a graph]
\label{def: orthogonal representation}
Let $\Gr{G}$ be a finite, simple, and undirected graph, and let $d \in \naturals$.
\begin{itemize}
\item
An {\em orthogonal representation} of the graph $\Gr{G}$ in the $d$-dimensional
Euclidean space $\Reals^d$ assigns to each vertex $i \in \V{\Gr{G}}$
a nonzero vector ${\bf{u}}_i \in \Reals^d$ such that
${\bf{u}}_i^{\mathrm{T}} {\bf{u}}_j = 0$ for every $\{i, j\} \notin \E{\Gr{G}}$
with $i \neq j$. In other words, for every two distinct and nonadjacent vertices in the graph,
their assigned nonzero vectors should be orthogonal in $\Reals^d$.
\item
An {\em orthonormal representation} of $\Gr{G}$ is additionally represented by unit vectors,
i.e., $\| {\bf{u}}_i \| = 1$ for all $i \in \V{\Gr{G}}$.
\item
In an orthogonal (orthonormal) representation of $\Gr{G}$, nonadjacent vertices in $\Gr{G}$
are mapped into orthogonal (orthonormal) vectors, but adjacent vertices may not necessarily
be mapped into nonorthogonal vectors. If ${\bf{u}}_i^{\mathrm{T}} {\bf{u}}_j \neq 0$ for all
$\{i, j\} \in \E{\Gr{G}}$, then such a representation of $\Gr{G}$ is called {\em faithful}.
\end{itemize}
\end{definition}

Every graph $\Gr{G}$ on $n$ vertices can be trivially represented by an orthonormal representation in $\Reals^n$,
by assigning each vertex $i \in \OneTo{n}$ to the standard unit-vector ${\bf{e}}_i$ (with a single $\tt{1}$ at the $i$-th
coordinate, and zeros elsewhere). That orthonormal representation is not faithful, unless $\Gr{G}$ is an empty graph.
In order to obtain an $n$-dimensional faithful orthogonal representation, let $\LM = (L_{i,j})$ be the Laplacian matrix of
$\Gr{G}$, which is a positive semidefinite matrix, and let ${\bf{u}}_i$ ($i \in \OneTo{n}$) be the $i$-th column
of the matrix $\LM^{1/2}$. Then, ${\bf{u}}_i^{\mathrm{T}} {\bf{u}}_j = L_{i,j}$ for all $i,j \in \OneTo{n}$, which implies
that $\{{\bf{u}}_i\}$ is a faithful orthogonal representation of~$\Gr{G}$. This idea extends to the signless Laplacian
matrix $\Q$ and the normalized Laplacian matrix ${\bf{\mathcal{L}}}$ since these are positive semidefinite matrices
by Theorems~\ref{theorem: On the Laplacian matrix of a graph}--\ref{theorem: On the normalized Laplacian matrix of a graph}).

\begin{example}[\hspace*{-0.15cm} \cite{Lovasz19}, Example~10.5]
\label{example: orthonormal representation} In an orthogonal representation of a graph $\Gr{G}$,
every vertex in a clique of $\Gr{G}$ can be represented by an identical vector.
Let $k \eqdef \chrnum{\CGr{G}}$ be the chromatic number of the complement graph $\CGr{G}$.
Then, there is a family of $k$ disjoint cliques in $\Gr{G}$, denoted by $\set{B}_1, \ldots, \set{B}_k$, covering the
vertex set $\V{\Gr{G}}$. Mapping all the vertices in the subset $\set{B}_i$, for $i \in \OneTo{k}$, into
${\bf{e}}_i \in \Reals^k$ gives an orthonormal representation of $\Gr{G}$ in the $k$-dimensional Euclidean space.
That orthogonal representation is not guaranteed to be faithful. Indeed, unless $\Gr{G}$ is a disjoint union of complete
graphs, it is possible for two adjacent vertices in $\Gr{G}$ to belong to different cliques; however, their corresponding
vectors are orthogonal.
\end{example}

The construction of low-dimensional orthogonal representations of graphs, particularly those with
unique properties, is a significant area of study in graph theory. Detailed discussions on the minimum
dimension of orthogonal representations of general graphs, especially those lacking specific subgraphs,
can be found in \cite{LovaszSS89, Balla23, BallaLS20, Lovasz79_IT} and Chapter 10 of \cite{Lovasz19}.
Some of the bounds on the minimum dimension of such representations
involve the Lov\'{a}sz $\vartheta$-function, which is presented in Section~\ref{subsection: Lovasz theta-function}
(see Item~\ref{item: min. dimension} of that section).
The minimum dimension of orthonormal representations of a graph $\Gr{G}$ on $n$ vertices is equal to the
minimum rank over all positive semidefinite matrices ${\bf{M}} = (M_{i,j})$ such that $M_{i,i}=1$ for all
$i \in \OneTo{n}$, and $M_{i,j} = 0$ for all $\{i,j\} \not\in \E{\Gr{G}}$ \cite{BallaLS20}.
Here, ${\bf{M}}$ serves as the {\em Gram matrix} of the representing vectors ${\bf{u}}_1, \ldots, {\bf{u}}_n$,
with $M_{i,j} \eqdef {\bf{u}}_i^{\mathrm{T}} {\bf{u}}_j$ for all $i,j \in \OneTo{n}$.
Consequently, the minimum dimension among all orthonormal representations of a graph $\Gr{G}$
is also termed the {\em minimum semidefinite rank} of $\Gr{G}$, and it is denoted by $\mathrm{msr}(\Gr{G})$.
A notable result by Lov\'{a}sz {\em et al.} \cite{LovaszSS89} demonstrates a relationship between graph
connectivity and the minimum dimension $d$ required for an orthogonal representation of a graph
to ensure that any $d$ of the representing vectors are linearly independent in $\Reals^d$.
Theorem~1 of \cite{LovaszSS89} states that a connected, undirected, and simple graph $\Gr{G}$ on
$n$ vertices has such an orthogonal representation in $\Reals^d$ if and only if $\Gr{G}$ is $(n-d)$-connected.
The latter result is also shown in Section 10.3 of \cite{Lovasz19}.

\subsection{Vector and Strict Vector Chromatic Numbers of Graphs}
\label{subsection: Vector chromatic number}
The vector and strict vector colorings of graphs, and their associated chromatic numbers, were
defined by Karger {\em et al.} \cite{KargerMS98}. These notions are introduced as follows, and
Section~\ref{subsection: Lovasz theta-function} addresses in part their relations to the Lov\'{a}sz
$\vartheta$-function and its variants, as well as their relations to other graph invariants.
\begin{definition}[Vector chromatic number]
\label{definition: vector chromatic number}
Let $\Gr{G}$ be a nonempty graph on $n$ vertices, and let $t \geq 2$ be a real number. A {\em vector $t$-coloring}
of $\Gr{G}$ in $\Reals^d$, with $d \in \naturals$, is an assignment of a unit vector ${\bf{u}}_i \in \Reals^d$
to each vertex $i \in \V{\Gr{G}}$ such that, for every two adjacent vertices $\{i, j\} \in \E{\Gr{G}}$,
\begin{align}
\label{eq: vector t-coloring}
{\bf{u}}_i^{\mathrm{T}} \, {\bf{u}}_j \leq -\frac{1}{t-1} \, .
\end{align}
The {\em vector chromatic number} of a nonempty graph $\Gr{G}$, denoted by $\vchrnum{\Gr{G}}$, is the smallest
real number $t \geq 2$ for which a vector $t$-coloring of $\Gr{G}$ exists in $\Reals^n$ (namely, the vector
$t$-coloring of $\Gr{G}$ is assumed here to be of dimension $n$). The vector chromatic number
of an empty graph is defined to be equal to~1.
\end{definition}
Let $\Gr{G}$ be a graph on $n$ vertices, and let ${\bf{A}} = (A_{i,j})$ be its $n \times n$
adjacency matrix. For a real symmetric $n \times n$ matrix ${\bf{M}}$, the standard
notation ${\bf{M}} \succeq 0$ means that the matrix ${\bf{M}}$ is positive semidefinite.
By Definition~\ref{definition: vector chromatic number}, and by setting the $n \times n$
matrix ${\bf{M}} = (M_{i,j})$ with $M_{i,j} \eqdef (t-1) {\bf{u}}_i^{\mathrm{T}} \, {\bf{u}}_j$
for all $i, j \in \OneTo{n}$,
the vector chromatic number of $\Gr{G}$ can be expressed to be equal to the value of the
following semidefinite programming (SDP) problem:
\vspace*{0.1cm}
\begin{eqnarray}
\label{eq: SDP problem - vector chromatic number}
\mbox{\fbox{$
\begin{array}{l}
\text{minimize} \; \; t  \\
\text{subject to} \\
\begin{cases}
{\bf{M}} \succeq 0, \\
M_{i,i} = t-1, \; \forall \, i \in \OneTo{n}, \\
A_{i,j} = 1 \; \Rightarrow \;  M_{i,j} \leq -1, \quad i,j \in \OneTo{n}.
\end{cases}
\end{array}$}}
\end{eqnarray}
Let ${\bf{J}}_n$ be the all-ones $n \times n$ matrix.
The dual of the SDP problem \eqref{eq: SDP problem - vector chromatic number} is given by
\vspace*{0.1cm}
\begin{eqnarray}
\label{eq: dual SDP problem - vector chromatic number}
\mbox{\fbox{$
\begin{array}{l}
\text{maximize} \; \; \mathrm{Tr}({\bf{B}} \, {\bf{J}}_n)  \\
\text{subject to} \\
\begin{cases}
{\bf{B}} \succeq 0, \\
\mathrm{Tr}({\bf{B}}) = 1, \\
B_{i,j} \geq 0, \; \; \forall \; i, j \in \OneTo{n}, \\
A_{i,j} = 0, \; i \neq j  \; \Rightarrow \;  B_{i,j} = 0, \quad i,j \in \OneTo{n}.
\end{cases}
\end{array}$}}
\end{eqnarray}
Strong duality holds here since the SDP problems are feasible and bounded, so the
solutions of the dual SDP problems in \eqref{eq: SDP problem - vector chromatic number}
and \eqref{eq: dual SDP problem - vector chromatic number} coincide. Applying the SDP
problem in \eqref{eq: dual SDP problem - vector chromatic number} to the complement
$\CGr{G}$ gives the {\em Schrijver $\vartheta$-function} of the graph $\Gr{G}$, which is denoted
by $\vartheta'(\Gr{G})$. In light of \eqref{eq: dual SDP problem - vector chromatic number},
with $\Gr{G}$ replaced by the complement $\CGr{G}$, the value of $\vartheta'(\Gr{G})$
is obtained by solving the SDP problem (see equation~(23) in \cite{Schrijver79}):
\begin{eqnarray}
\label{eq: SDP problem - Schrijver's theta-function}
\mbox{\fbox{$
\begin{array}{l}
\text{maximize} \; \; \mathrm{Tr}({\bf{B}} \, {\bf{J}}_n)  \\
\text{subject to} \\
\begin{cases}
{\bf{B}} \succeq 0, \\
\mathrm{Tr}({\bf{B}}) = 1, \\
B_{i,j} \geq 0, \; \; \forall \; i, j \in \OneTo{n}, \\
A_{i,j} = 1 \; \Rightarrow \;  B_{i,j} = 0, \quad i,j \in \OneTo{n}.
\end{cases}
\end{array}$}}
\end{eqnarray}
The vector chromatic number and Schrijver's $\vartheta$-function therefore satisfy the equality
\begin{align}
\label{eq1:17.11.23}
\vartheta'(\Gr{G}) = \vchrnum{\CGr{G}}
\end{align}
for every graph $\Gr{G}$.
By \eqref{eq: SDP problem - Schrijver's theta-function}, it follows that
\begin{align}
\label{eq2:17.11.23}
\vartheta'(\Gr{G}) \geq \indnum{\Gr{G}},
\end{align}
which holds by selecting a feasible solution in \eqref{eq: SDP problem - Schrijver's theta-function}
as follows. Let $\mathcal{I}$ be a largest independent set in $\Gr{G}$, and let
$\mathcal{I} = \{i_1, \ldots, i_\ell\} \subseteq \OneTo{n}$ with $\ell = \indnum{\Gr{G}}$.
Define ${\bf{B}}$ to be the $n \times n$ symmetric matrix whose elements are given by
$B_{i,j} \eqdef \frac{1}{\indnum{\Gr{G}}}$ whenever $i,j \in \mathcal{I}$, and $B_{i,j} \eqdef 0$ otherwise.
Then, ${\bf{B}}$ is indeed a positive semidefinite matrix whose trace is equal to~1, and the objective
function in \eqref{eq: SDP problem - Schrijver's theta-function} is then equal to $\indnum{\Gr{G}}$. Due to
the maximization in \eqref{eq: SDP problem - Schrijver's theta-function}, inequality \eqref{eq2:17.11.23} holds.

Schrijver's $\vartheta$-function of a graph was studied independently by McEliece {\em et al.}
\cite{McElieceRR78} (with a different notation, where $\vartheta'(\Gr{G})$ in \cite{Schrijver79}
is replaced by $\alpha_{\mathrm{L}}(\Gr{G})$ in \cite{McElieceRR78}). In a subsequent study by
Szegedy \cite{Szegedy94}, $\vartheta'(\Gr{G})$ is denoted by $\vartheta_{1/2}(\Gr{G})$
(an additional variant, denoted by $\vartheta_2(\Gr{G})$, was also introduced in \cite{Szegedy94}). To that end,
the SDP problem in \eqref{eq: SDP problem - Schrijver's theta-function} was modified by replacing
the last two lines in \eqref{eq: SDP problem - Schrijver's theta-function} with the requirement that
$B_{i,j} \leq 0$ for all $i,j \in \OneTo{n}$ such that $A_{i,j} = 1$ (see Eq.~(8) in \cite{Szegedy94}).
The Schrijver $\vartheta$-function of a graph $\Gr{G}$, $\vartheta'(\Gr{G})$, constitutes a
variant of the Lov\'{a}sz $\vartheta$-function $\vartheta(\Gr{G})$, which is subsequently introduced in
Section~\ref{subsection: Lovasz theta-function}.

Spectral characterizations of the vector chromatic number of a graph, which is equal to
the Schrijver $\vartheta$-function of the graph complement by \eqref{eq1:17.11.23}, were studied in \cite{Bilu06,Galtman00},
and more recently in \cite{GodsilRRSV2020,WocjanEA23}. We next define the strict vector chromatic
number of a graph, which is also shown in Section~\ref{subsection: Lovasz theta-function} to be
related to the Lov\'{a}sz $\vartheta$-function in a similar way to \eqref{eq1:17.11.23}
(see \eqref{eq: svchrnum vs. Lovasz theta-function}).
\begin{definition}[Strict vector chromatic number]
\label{definition: strict vector chromatic number}
Let $\Gr{G}$ be a nonempty graph on $n$ vertices, and let $t \geq 2$ be a real number. A {\em strict vector $t$-coloring}
of $\Gr{G}$ is an assignment of a unit vector ${\bf{u}}_i \in \Reals^n$ to each vertex $i \in \V{\Gr{G}}$
such that, for every two adjacent vertices $\{i,j\} \in \E{\Gr{G}}$, the condition in \eqref{eq: vector t-coloring}
holds with equality.
The {\em strict vector chromatic number} of a nonempty graph $\Gr{G}$, denoted by $\svchrnum{\Gr{G}}$, is the
smallest real number $t \geq 2$ for which a strict vector $t$-coloring of $\Gr{G}$ exists. The strict vector
chromatic number of an empty graph is defined to be equal to~1.
\end{definition}
Clearly, by Definitions~\ref{definition: vector chromatic number} and~\ref{definition: strict vector chromatic number}, the inequality
\begin{align}
\label{eq1:05.11.23}
\vchrnum{\Gr{G}} \leq \svchrnum{\Gr{G}}
\end{align}
holds for every graph $\Gr{G}$.
By Definition~\ref{definition: strict vector chromatic number}, the strict vector chromatic number is
expressed as a solution of the SDP problem that is almost similar to \eqref{eq: SDP problem - vector chromatic number};
the only difference is that the last inequality constraints in \eqref{eq: SDP problem - vector chromatic number}
turn into equality constraints. Then, similarly to the transition from the SDP problem
in \eqref{eq: SDP problem - vector chromatic number} to the dual SDP problem in \eqref{eq: dual SDP problem - vector chromatic number},
strong duality gives that the strict vector chromatic number $\svchrnum{\Gr{G}}$ can be obtained by solving the following SDP problem:
\begin{eqnarray}
\label{eq: dual SDP problem - strict vector chromatic number}
\mbox{\fbox{$
\begin{array}{l}
\text{maximize} \; \; \mathrm{Tr}({\bf{B}} \, {\bf{J}}_n)  \\
\text{subject to} \\
\begin{cases}
{\bf{B}} \succeq 0, \\
\mathrm{Tr}({\bf{B}}) = 1, \\
A_{i,j} = 0, \; i \neq j  \; \Rightarrow \;  B_{i,j} = 0, \quad i,j \in \OneTo{n}.
\end{cases}
\end{array}$}}
\end{eqnarray}
In comparison to \eqref{eq: dual SDP problem - vector chromatic number}, the
condition that all entries of ${\bf{B}}$ are nonnegative is
absent in \eqref{eq: dual SDP problem - strict vector chromatic number}.

\subsection{The Lov\'{a}sz $\vartheta$-Function of Graphs}
\label{subsection: Lovasz theta-function}
The Lov\'{a}sz $\vartheta$-function of a graph, as introduced in \cite{Lovasz79_IT}, exhibits intriguing
connections to diverse graph parameters such as the independence number, clique number,
chromatic number, and Shannon capacity of the graph. Its efficient computability renders
it a potent tool in information theory, graph theory, and combinatorial optimization.

\begin{definition}[Lov\'{a}sz $\vartheta$-function]
\label{definition: Lovasz theta function}
Let $\Gr{G}$ be a finite, simple, and undirected graph. Then, the {\em Lov\'{a}sz
$\vartheta$-function of $\Gr{G}$} is defined as
\begin{eqnarray}
\label{eq: Lovasz theta function}
\vartheta(\Gr{G}) \triangleq \min_{\bf{u}, \bf{c}} \, \max_{i \in \V{\Gr{G}}} \,
\frac1{\bigl( {\bf{c}}^{\mathrm{T}} {\bf{u}}_i \bigr)^2} \, ,
\end{eqnarray}
where the minimum is taken over all orthonormal representations $\{{\bf{u}}_i: i \in \V{\Gr{G}} \}$ of $\Gr{G}$,
and all unit vectors ${\bf{c}}$.
The unit vector $\bf{c}$ is called the {\em handle} of the orthonormal representation. The minimization
on the right-hand side of \eqref{eq: Lovasz theta function} is implicitly performed over the dimension of the vectors.
However, without any loss of generality, it suffices to restrict their dimension to be equal to $n = \card{\V{\Gr{G}}}$.
\end{definition}

The Lov\'{a}sz $\vartheta$-function can be expressed as a solution of an SDP
problem. To that end, let ${\bf{A}} = (A_{i,j})$ be the $n \times n$ adjacency matrix of $\Gr{G}$ with
$n \triangleq \card{\V{\Gr{G}}}$. The Lov\'{a}sz
$\vartheta$-function $\vartheta(\Gr{G})$ can be expressed by the following convex optimization problem:
\vspace*{0.1cm}
\begin{eqnarray}
\label{eq: SDP problem - Lovasz theta-function}
\mbox{\fbox{$
\begin{array}{l}
\text{maximize} \; \; \mathrm{Tr}({\bf{B}} \, {\bf{J}}_n)  \\
\text{subject to} \\
\begin{cases}
{\bf{B}} \succeq 0, \\
\mathrm{Tr}({\bf{B}}) = 1, \\
A_{i,j} = 1  \; \Rightarrow \;  B_{i,j} = 0, \quad i,j \in \OneTo{n}.
\end{cases}
\end{array}$}}
\end{eqnarray}
The SDP formulation in \eqref{eq: SDP problem - Lovasz theta-function} yields
the existence of an algorithm that computes $\vartheta(\Gr{G})$, for every graph $\Gr{G}$, with a
precision of $r$ decimal digits, and a computational complexity that is polynomial in $n$ and $r$.
A comparison of the optimization problems in \eqref{eq: dual SDP problem - vector chromatic number} and
\eqref{eq: SDP problem - Lovasz theta-function} gives
\begin{align}
\label{eq: vchrnum vs. Lovasz theta-function}
\vchrnum{\Gr{G}} \leq \vartheta(\CGr{G}),
\end{align}
which holds for every graph $\Gr{G}$. Indeed, \eqref{eq: vchrnum vs. Lovasz theta-function} follows
from the additional constraint in \eqref{eq: dual SDP problem - vector chromatic number} that all the
entries of the positive semidefinite matrix ${\bf{B}}$ are nonnegative, whereas the latter condition
is absent in \eqref{eq: SDP problem - Lovasz theta-function}. A comparison of the SDP
problems in \eqref{eq: dual SDP problem - strict vector chromatic number} and
\eqref{eq: SDP problem - Lovasz theta-function} also gives (see Theorem~8.2 in \cite{KargerMS98})
\begin{align}
\label{eq: svchrnum vs. Lovasz theta-function}
\svchrnum{\Gr{G}} = \vartheta(\CGr{G}).
\end{align}

The following properties of the Lov\'{a}sz $\vartheta$-function and its variant by Schrijver are collected from
\cite{Acin17, Alon19, AlonK98, Balla23, BallaJS24, BallaLS20, CsonkaS23, Feige97, GrotschelLS81, Knuth94, Lovasz75, Lovasz79_IT, Lovasz19, Sason23, Schrijver79,Szegedy94},
presented herein for reference and convenience.
Some of these properties are employed in the analysis presented in this paper. It is assumed throughout that the
graphs are finite, simple, and undirected.
\begin{enumerate}
\item Upper bound on the Shannon capacity of graphs: For every graph $\Gr{G}$
\begin{equation}
\label{eq: Lovasz 79 - theorem 1}
\Theta(\Gr{G}) \leq \vartheta(\Gr{G}).
\end{equation}
Although the upper bound in \eqref{eq: Lovasz 79 - theorem 1} is tight in some cases, e.g., for the family of
Kneser graphs and the family of self-complementary vertex-transitive graphs \cite{Lovasz19}, it is not a tight
bound in general.
More explicitly, there exists a sequence of graphs $\{\Gr{G}_n\}$ where $\Gr{G}_n$ is a graph on $n$ vertices
such that $\Theta(\Gr{G}_n) \leq 3$ and $\vartheta(\Gr{G}_n) > \sqrt[4]{n}$ for all $n \in \naturals$.
\item Sandwich theorem:
\begin{align}
\label{eq1a: sandwich}
& \indnum{\Gr{G}} \leq \vchrnum{\CGr{G}} = \vartheta'(\Gr{G}) \leq \vartheta(\Gr{G})
= \svchrnum{\CGr{G}} \leq \findnum{\Gr{G}} = \fchrnum{\CGr{G}} \leq \chrnum{\CGr{G}}, \\
\label{eq1b: sandwich}
& \clnum{\Gr{G}} \leq \vchrnum{\Gr{G}} = \vartheta'(\CGr{G}) \leq \vartheta(\CGr{G})
= \svchrnum{\Gr{G}} \leq \fclnum{\Gr{G}} = \fchrnum{\Gr{G}} \leq \chrnum{\Gr{G}}.
\end{align}
\item
The inequality $\vartheta(\Gr{G}) \geq 1$ holds in general, with an equality if and only if $\Gr{G}$ is a complete graph.
\item \label{item: computational complexity}
Computational complexity of graph invariants:
\begin{itemize}
\item Computing $\indnum{\Gr{G}}$, $\findnum{\Gr{G}}$, $\clnum{\Gr{G}}$, $\fclnum{\Gr{G}}$,
$\fchrnum{\Gr{G}}$, and $\chrnum{\Gr{G}}$ constitutes NP-hard problems.
\item
Computing bounds on these graph invariants, specifically $\vartheta(\Gr{G})$, $\vartheta'(\Gr{G})$, $\vartheta(\CGr{G})$,
and $\vartheta'(\CGr{G})$ (by \eqref{eq1a: sandwich} and \eqref{eq1b: sandwich}), in any desired precision,
is feasible. The computational complexity scales polynomially in $n$, and it is obtained by numerically solving
the SDP problems presented in Section~\ref{subsection: Vector chromatic number} and~\eqref{eq: SDP problem - Lovasz theta-function}.
\end{itemize}
\item For every graph $\Gr{G}$, combining \eqref{eq: Lovasz75} and \eqref{eq1a: sandwich} gives
\begin{align}
\label{eq2: Lovasz75}
\frac{1}{1 + \ln \vartheta(\Gr{G})} \leq \frac{1}{1 + \ln \vartheta'(\Gr{G})} \leq \frac{\fchrnum{\Gr{G}}}{\chrnum{\Gr{G}}} \leq 1.
\end{align}
By Item~\ref{item: computational complexity}, \eqref{eq2: Lovasz75} gives computable lower bounds on the ratio $\frac{\fchrnum{\Gr{G}}}{\chrnum{\Gr{G}}}$.
\item In regard to \eqref{eq1a: sandwich}, the ratios $\frac{\chrnum{\CGr{G}}}{\vartheta(\Gr{G})}$ and
$\frac{\vartheta(\Gr{G})}{\indnum{\Gr{G}}}$ can be made arbitrarily large as follows:
\begin{itemize}
\item There exist a constant $c>0$ and an infinite sequence of graphs $\{\Gr{G}_\ell\}$, with increasing
orders $n_\ell \eqdef \card{\V{\Gr{G}_\ell}}$, for which
$\indnum{\Gr{G}_\ell} < 2^{\sqrt{\log n_\ell}}$ and $\vartheta(\Gr{G}_\ell) > n_\ell \, 2^{-c \sqrt{\log n_\ell}}$ hold for all $\ell$.
\item There exist a constant $c>0$ and an infinite sequence of graphs $\{\Gr{G}_\ell\}$, with increasing
orders $n_\ell \eqdef \card{\V{\Gr{G}_\ell}}$, for which
$\vartheta(\Gr{G}_\ell) < 2^{\sqrt{\log n_\ell}}$ and $\chrnum{\CGr{G}_\ell} > n_\ell \, 2^{-c \sqrt{\log n_\ell}}$ hold for all $\ell$.
\end{itemize}
\item
A counterexample shows that, unlike \eqref{eq: Lovasz 79 - theorem 1}, $\Theta(\Gr{G}) \not\leq \vartheta'(\Gr{G})$
(see Example~\ref{example: counterexample} here).
\item \label{item: min. dimension}
If a graph $\Gr{G}$ has an orthonormal representation in the $d$-dimensional Euclidean space $\Reals^d$, then
$d \geq \vartheta(\Gr{G})$ by Theorem~11 of \cite{Lovasz79_IT}, and $d \geq \tfrac12 \log_2 \chrnum{\CGr{G}}$
by Proposition 10.8 of \cite{Lovasz79_IT}. Additionally, by
Example~\ref{example: orthonormal representation}, there exists an othornormal representation
of $\Gr{G}$ in $\Reals^d$ with $d = \chrnum{\CGr{G}}$. The {\em minimum semidefinite rank} of $\Gr{G}$, denoted
by $\mathrm{msr}(\Gr{G})$, is equal to the minimum dimension $d$ of an orthonormal representation
of $\Gr{G}$ in $\Reals^d$ (see Section~\ref{subsection: Orthonormal Representations of Graphs}), so
\begin{align}
\label{eq1: min. semidefinite rank}
\max \Bigl\{\vartheta(\Gr{G}), \, \tfrac12 \log_2 \chrnum{\CGr{G}} \Bigr\} \leq \mathrm{msr}(\Gr{G}) \leq \chrnum{\CGr{G}}.
\end{align}
This indicates that the minimum dimension of an orthonormal representation of a graph $\Gr{G}$
is somewhat related to the chromatic number of the graph complement $\CGr{G}$.
\item Factorization under strong products: For all graphs $\Gr{G}$ and $\Gr{H}$,
\begin{align}
\label{eq: factorization}
\vartheta(\Gr{G} \boxtimes \Gr{H}) = \vartheta(\Gr{G}) \; \vartheta(\Gr{H}).
\end{align}
\item For every graph $\Gr{G}$ on $n$ vertices,
\begin{align}
\label{eq10:11.10.23}
\vartheta(\Gr{G}) \; \vartheta(\CGr{G}) \geq n,
\end{align}
with an equality in \eqref{eq10:11.10.23} if the graph $\Gr{G}$ is vertex-transitive or strongly regular.
\item Let $\Gr{G}$ be a $d$-regular graph on $n$ vertices. Then,
\begin{align}
\label{eq: Lovasz79 - Theorem 9}
\vartheta(\Gr{G}) \leq -\frac{n \, \Eigval{n}{\Gr{G}}}{d - \Eigval{n}{\Gr{G}}},
\end{align}
with an equality in \eqref{eq: Lovasz79 - Theorem 9} if $\Gr{G}$ is an edge-transitive graph.
\item More generally, if $\Gr{G}$ is a $d$-regular graph on $n$ vertices, then
for every symmetric nonzero $n \times n$ matrix ${\mathbf{M}}$ such that $M_{i,j}=0$
for $\{i,j\} \in \E{\CGr{G}}$ and also for $i=j$, and with ${\mathbf{M}}$ having
equal row-sums, the following holds:
\begin{align}
\label{eq11:11.10.23}
\vartheta(\Gr{G}) \leq \frac{-n \, \Eigval{\min}{\mathbf{M}}}{\Eigval{\max}{\mathbf{M}} - \Eigval{\min}{\mathbf{M}}}.
\end{align}
If $\Gr{G}$ is a vertex-transitive graph, then there exists such a matrix ${\mathbf{M}}$ attaining
equality in \eqref{eq11:11.10.23}. If $\Gr{G}$ is an edge-transitive graph, then \eqref{eq11:11.10.23}
holds with equality for ${\mathbf{M}} = \A(\Gr{G})$.
\item Lov\'{a}sz $\vartheta$-function of subgraphs:
\begin{itemize}
\item \label{item: subgraphs}
If $\, \Gr{H}$ is a spanning subgraph of a graph $\Gr{G}$, then
$\vartheta(\Gr{H}) \geq \vartheta(\Gr{G})$.
\item If $\, \Gr{H}$ is an induced subgraph of a graph $\Gr{G}$, then
$\vartheta(\Gr{H}) \leq \vartheta(\Gr{G})$.
\end{itemize}
\item For two graphs $\Gr{G}$ and $\Gr{H}$, let $\Gr{G} + \Gr{H}$ denote their disjoint union (see
Definition~\ref{def:disjoint_union_graphs}). Then,
\begin{align}
\label{eq: Lovasz function of a disjoint union of graphs}
\vartheta(\Gr{G} + \Gr{H}) = \vartheta(\Gr{G}) + \vartheta(\Gr{H}).
\end{align}
\item Let $\Gr{G}$ be a graph on $n$ vertices, and let $\indnum{\Gr{G}} = k$. Then,
\begin{align}
\label{eq12:11.10.23}
\vartheta(\Gr{G}) \leq 16 \, n^{\; \tfrac{k-1}{k+1}}.
\end{align}
\item Let $\Gr{G}$ be a graph on $n$ vertices, and let $\vartheta(\CGr{G}) = t \geq 2$. Then,
\begin{align}
\label{eq0:17.10.23}
\indnum{\Gr{G}} \geq \frac{n^{\; \tfrac{3}{t+1}}}{10 \sqrt{\ln n}}.
\end{align}
\item The surplus of a graph $\Gr{G}$ (see Definition~\ref{definition: maximum-cut problem}),
denoted by $\mathrm{sp}(\Gr{G})$, satisfies
\begin{align}
\label{eq1: Sudakov1}
\mathrm{sp}(\Gr{G}) \geq  \frac1{\pi} \, \frac{\card{\E{\Gr{G}}}}{\vchrnum{\Gr{G}}-1}
\geq \frac1{\pi} \, \frac{\card{\E{\Gr{G}}}}{\vartheta(\CGr{G})-1}.
\end{align}
The utility of \eqref{eq1: Sudakov1} in lower bounding the surplus of a graph $\Gr{G}$ follows from the fact
that the computational complexity of the max-cut of a graph is NP-complete, whereas the computation of the
vector chromatic number in the middle term of \eqref{eq1: Sudakov1}, as well as of the Lov\'{a}sz
$\vartheta$-function in the rightmost term of~\eqref{eq1: Sudakov1}, is feasible (see
Section~\ref{subsection: Vector chromatic number} and Item~\ref{item: computational complexity} above).
\item For every graph $\Gr{G}$,
\begin{eqnarray}
\label{eq2: Acin17 - Theorem 2}
\sup_{\Gr{H}} \frac{\indnum{\Gr{G} \boxtimes \Gr{H}}}{\vartheta(\Gr{G} \boxtimes \Gr{H})} = 1,
\end{eqnarray}
where the supremum is taken over all (simple, finite, and undirected) graphs $\Gr{H}$.
\item For every graph $\Gr{G}$ on $n$ vertices, let
$\tilde{\vartheta}(\Gr{G}) \eqdef \underset{\Gr{H}}{\max} \, \dfrac{\card{\V{\Gr{H}}}}{\vartheta(\Gr{H})}$,
where the maximization is taken over all
the induced subgraphs $\Gr{H}$ of $\Gr{G}$ (see equation~(10) in \cite{Szegedy94}). Then, the following holds:
\begin{align}
\label{Szegedy - Theorem 2.1}
\tfrac18 \, \vartheta(\CGr{G}) \, (\ln n)^{-2} \leq \tilde{\vartheta}(\Gr{G}) \leq \vartheta(\CGr{G}).
\end{align}
\item For every graph $\Gr{G}$ on $n$ vertices, the following two equivalent chain of inequalities hold:
\begin{align}
\label{eq:Balla 23 - Theorem 8}
\frac{1}{n} \, \vartheta(\CGr{G})^2 \leq \vchrnum{\Gr{G}} \leq \vartheta(\CGr{G}),  \qquad
\frac{1}{n} \, \vartheta(\Gr{G})^2 \leq \vartheta'(\Gr{G}) \leq \vartheta(\Gr{G}).
\end{align}
\item Let $\Gr{G}$ be a graph on $n$ vertices, and let $\ell(\Gr{G})$ be the maximum of
$\| {\bf{u}}_1 + \ldots + {\bf{u}}_n \|$ over all orthonormal representations
$\{{\bf{u}}_i\}$ of $\Gr{G}$. Then,
\begin{align}
\label{eq: Balla et al. '20}
\frac{n}{\sqrt{\vartheta(\Gr{G})}} \leq \ell(\Gr{G}) \leq \sqrt{n \, \vartheta(\CGr{G})},
\end{align}
and the upper and lower bounds in \eqref{eq: Balla et al. '20} coincide if $\Gr{G}$ is a vertex-transitive
or a strongly regular graph. [By Claim~1.2 in \cite{BallaLS20}, \eqref{eq: Balla et al. '20} holds with
equalities if $\Gr{G}$ is vertex-transitive. It also holds with equalities if $\Gr{G}$ is a strongly regular
graph, as indicated by the sufficient condition for equality in \eqref{eq10:11.10.23} (see Corollary~1 in \cite{Sason23}
or Corollary~\ref{corollary:identity for the Lovasz function of srg and its complement} here)].
\item The Lov\'{a}sz $\vartheta$-function of the complement of the Mycielskian $\mathrm{M}(\Gr{G})$ of a
graph $\Gr{G}$ is uniquely determined by $\vartheta(\CGr{G})$ (see Theorem~2 in \cite{CsonkaS23} for an explicit
equation).
\item For every graph $\Gr{G}$ on $n$ vertices, the following holds:
\begin{enumerate}
\item
\begin{align}
\label{eq0:optimization}
\vartheta(\Gr{G}) = \max \sum_{i=1}^n \bigl({\mathbf{d}}^{\mathrm{T}} \, {\mathbf{v}}_i \bigr)^2,
\end{align}
where the maximization on the right-hand side of \eqref{eq0:optimization} is taken over all the
unit vectors $\mathbf{d}$ and orthonormal representations $\{{\mathbf{v}}_i\}_{i=1}^n$ of
the complement graph $\CGr{G}$.
\item
\begin{align}
\label{eq1:optimization}
\vartheta(\Gr{G}) = \min_{{\mathbf{U}}} \lambda_{\max}({\mathbf{U}}),
\end{align}
where the minimization on the right-hand side of \eqref{eq1:optimization} is taken over all
$n \times n$ symmetric matrices ${\mathbf{U}} = (U_{i,j})$ with $U_{i,j} = 1$ for all $i,j \in \OneTo{n}$
such that $\{i,j\} \in \E{\CGr{G}}$ or $i=j$.
\item
\begin{align}
\label{eq2:optimization}
\vartheta(\Gr{G}) = \max_{{\bf{W}}} \lambda_{\max}({\mathbf{W}}),
\end{align}
where the maximization on the right-hand side of \eqref{eq2:optimization} is taken over all positive
semidefinite $n \times n$ matrices ${\mathbf{W}} \succeq 0$ with $W_{i,j} = 0$ for all
$\{i,j\} \in \E{\Gr{G}}$, and $W_{i,i} = 1$ for all $i \in \OneTo{n}$.
\item
\begin{align}
\label{eq3:optimization}
\vartheta(\Gr{G}) = 1 + \max_{{\bf{T}}} \frac{\lambda_{\max}({\bf{T}})}{\bigl|\lambda_{\min}({\bf{T}})\bigr|},
\end{align}
where the maximization on the right-hand side of \eqref{eq3:optimization} is taken over all symmetric
nonzero $n \times n$ matrices ${\bf{T}} = (T_{i,j})$ with $T_{i,j}=0$ for all $i,j \in \OneTo{n}$ such that
$\{i,j\} \in \E{\Gr{G}}$ or $i=j$.
\item The SDP problem in \eqref{eq: SDP problem - Lovasz theta-function} is given by
\begin{align}
\label{eq4:optimization}
\vartheta(\Gr{G}) = \max_{{\bf{B}}} \, \mathrm{Tr}({\bf{B}} \, {\bf{J}}_n),
\end{align}
where the maximization on the right-hand side of \eqref{eq4:optimization} is taken over all positive
semidefinite $n \times n$ matrices ${\bf{B}} \succeq 0$ with $\mathrm{Tr}({\bf{B}}) = 1$,
and $B_{i,j} = 0$ for all $i,j \in \OneTo{n}$ such that $\{i,j\} \in \E{\Gr{G}}$.
\item
\begin{align}
\label{eq4-dual:optimization}
\vartheta(\Gr{G}) = 1 + \min_{{\bf{Y}}} \, \max_{i \in \OneTo{n}} Y_{i,i},
\end{align}
where the minimization on the right-hand side of \eqref{eq4-dual:optimization} is taken over
all positive semidefinite $n \times n$ matrices ${\bf{Y}} \succeq 0$ with $Y_{i,j} = -1$ for
all $i,j \in \OneTo{n}$ such that $\{i,j\} \not\in \E{\Gr{G}}$.
\end{enumerate}
\end{enumerate}

\section{Observations on the Shannon Capacity of Graphs}
\label{section: on the Shannon capacity of graphs}

Shannon's problem of zero-error communication and the notion of the Shannon capacity of graphs
\cite{Shannon56} had a major impact on the development of information theory,
extremal combinatorics, and graph theory. This section provides new observations on the Shannon
capacity of graphs.

\subsection{Specialized Preliminaries for Section~\ref{section: on the Shannon capacity of graphs}}
\label{subsection: preliminaries - Shannon capacity of graphs}
The present section starts with the introduction of specialized preliminaries that are essential
for elucidating the subsequent results. These preliminaries encompass bounds and exact findings
related to the Lov\'{a}sz $\vartheta$-function for regular and strongly regular graphs. We also
include results pertaining to subclasses of self-complementary regular graphs.

\subsubsection{On the Lov\'{a}sz $\vartheta$-function of regular and strongly regular graphs}
\label{subsubsection: On the Lovasz theta-function of regular and strongly regular graphs}

\begin{theorem}[Bounds on the Lov\'{a}sz function of regular graphs, \cite{Sason23}]
\label{thm:bounds on the Lovasz function for regular graphs}
Let $\Gr{G}$ be a $d$-regular graph of order $n$, which is a noncomplete
and nonempty graph. Then, the following bounds hold for the Lov\'{a}sz
$\vartheta$-function of $\Gr{G}$ and its complement $\CGr{G}${\em :}

\begin{enumerate}
\item
\begin{eqnarray}
\label{eq:21.10.22a1}
\frac{n-d+\Eigval{2}{\Gr{G}}}{1+\Eigval{2}{\Gr{G}}} \leq \vartheta(\Gr{G})
\leq -\frac{n \Eigval{n}{\Gr{G}}}{d - \Eigval{n}{\Gr{G}}}.
\end{eqnarray}
\begin{itemize}
\item Equality holds in the leftmost inequality of \eqref{eq:21.10.22a1} if $\CGr{G}$
is both vertex-transitive and edge-transitive, or if $\Gr{G}$ is a strongly regular graph;
\item Equality holds in the rightmost inequality of \eqref{eq:21.10.22a1} if $\Gr{G}$
is edge-transitive, or if $\Gr{G}$ is a strongly regular graph.
\end{itemize}

\item
\begin{eqnarray}
\label{eq:21.10.22a2}
1 - \frac{d}{\Eigval{n}{\Gr{G}}} \leq \vartheta(\CGr{G})
\leq \frac{n \bigl(1+\Eigval{2}{\Gr{G}}\bigr)}{n-d+\Eigval{2}{\Gr{G}}}.
\end{eqnarray}
\begin{itemize}
\item Equality holds in the leftmost inequality of \eqref{eq:21.10.22a2}
if $\Gr{G}$ is both vertex-transitive and edge-transitive, or if $\Gr{G}$ is
a strongly regular graph;
\item Equality holds in the rightmost inequality of \eqref{eq:21.10.22a2}
if $\CGr{G}$ is edge-transitive, or if $\Gr{G}$ is a strongly regular graph.
\end{itemize}
\end{enumerate}
\end{theorem}

\begin{corollary}
\label{corollary: a common sufficient condition}
{\em All inequalities in Theorem~\ref{thm:bounds on the Lovasz function for regular graphs}
hold with equality if $\Gr{G}$ is a strongly regular graph.}
\end{corollary}

\begin{proof}
This follows from the sufficient conditions of the inequalities in
Theorem~\ref{thm:bounds on the Lovasz function for regular graphs}
to hold with equality.
Recall that a graph $\Gr{G}$ is strongly regular if and only if $\CGr{G}$ is so.
\end{proof}

\begin{theorem}[Lov\'{a}sz function of strongly regular graphs, \cite{Sason23}]
\label{thm:Lovasz function of srg}
Let $\Gr{G}$ be a strongly regular graph with parameters $\srg{n}{d}{\lambda}{\mu}$.
Then,
\begin{align}
\label{eq: Lovasz-theta srg}
& \vartheta(\Gr{G}) = \dfrac{n \, (t+\mu-\lambda)}{2d+t+\mu-\lambda}, \\[0.2cm]
\label{eq: Lovasz-theta srg comp.}
& \vartheta(\CGr{G}) = 1 + \dfrac{2d}{t+\mu-\lambda},
\end{align}
where
\begin{align}
\label{eq: t parameter - srg}
t \eqdef \sqrt{(\mu-\lambda)^2 + 4(d-\mu)}.
\end{align}
Furthermore, if $2d + (n-1)(\lambda-\mu) \neq 0$, then $\vartheta(\Gr{G}), \, \vartheta(\CGr{G}) \in \rationals$
(i.e., rational numbers).
\end{theorem}

\begin{corollary}[An identity for vertex-transitive or strongly regular graphs]
\label{corollary:identity for the Lovasz function of srg and its complement}
The equality
\begin{align}
\label{eq3:02.08.23}
\vartheta(\Gr{G}) \; \vartheta(\CGr{G}) = n
\end{align}
holds for all vertex-transitive graphs and for all strongly regular graphs.
\end{corollary}
\begin{proof}
Equality \eqref{eq3:02.08.23} holds for vertex-transitive graphs by Theorem~8 in \cite{Lovasz79_IT}.
For strongly regular graphs, \eqref{eq3:02.08.23} follows readily from the expressions for
the Lov\'{a}sz $\vartheta$-functions in \eqref{eq: Lovasz-theta srg} and \eqref{eq: Lovasz-theta srg comp.}.
\end{proof}
\begin{remark}
By Theorem~5 in \cite{Schrijver79}, equality~\eqref{eq3:02.08.23} holds more generally for all graphs
derived from symmetric association schemes, specifically including the family of strongly regular graphs.
\end{remark}

\begin{remark}[Strongly regular graphs are not necessarily vertex- or edge-transitive graphs]
\label{remark: srg are not necessarily v.t.}
In light of Corollary~\ref{corollary:identity for the Lovasz function of srg and its complement},
it should be noted that many strongly regular graphs are neither vertex- nor edge-transitive.
One such example is Gritsenko's graph \cite{Gritsenko01}, which is a conference graph on~65 vertices;
it is a strongly regular graph with parameters $\srg{65}{32}{15}{16}$ that is neither vertex- nor
edge-transitive, and it is also not self-complementary. These findings can be verified, for example,
by the aid of the SageMath software \cite{SageMath}.
The existence of a self-complementary and strongly regular graph that is not vertex-transitive,
however, remains an open problem (see Item~\ref{item: open problem 1} in Remark~\ref{remark: regular self-complementary graphs}).
\end{remark}

\subsubsection{The Shannnon capacity of self-complementary vertex-transitive graphs}
\label{subsubsection: Shannnon capacity of self-complementary vertex-transitive graphs}

In some cases, the Shannon capacity of a graph can be calculated exactly, and
the Lov\'{a}sz $\vartheta$-function is a tight bound. This includes the following result.
\begin{theorem}[Theorem~12 in \cite{Lovasz79_IT}]
\label{thm: Lovasz - s.c.-v.t. capacity}
Let $\Gr{G}$ be an undirected and simple graph on $n$ vertices.
\begin{enumerate}
\item
If $\Gr{G}$ is a vertex-transitive graph on $n$ vertices, then
\begin{align}
\label{eq2: Lovasz - s.c.-v.t. capacity}
\indnum{\Gr{G} \boxtimes \Gr{G}} = \Theta(\Gr{G} \boxtimes \Gr{G}) = \vartheta(\Gr{G} \boxtimes \Gr{G}) = n.
\end{align}
\item
If $\Gr{G}$ is a self-complementary and vertex-transitive graph on $n$ vertices, then
\begin{align}
\label{eq1: Lovasz - s.c.-v.t. capacity}
& \Theta(\Gr{G}) = \sqrt{n} = \vartheta(\Gr{G}).
\end{align}
\end{enumerate}
\end{theorem}
\begin{remark}
Strengthened and refined versions of Theorem~\ref{thm: Lovasz - s.c.-v.t. capacity} are provided in
Section~\ref{subsection: New Results on Shannon Capacity of Graphs}.
\end{remark}

\subsubsection{On classes of regular and self-complementary graphs}
\label{subsubsection: classes of regular self-complementary graphs}

\begin{theorem}[On symmetry, regularity, and strong regularity]
\label{thm:symmetry and regular graphs are srg}
If $\Gr{G}$ is a regular graph, and $\Gr{G}$ and $\CGr{G}$ are both edge-transitive,
then $\Gr{G}$ is a strongly regular graph.
\end{theorem}

\begin{theorem}[\cite{Neumaier80}]
\label{thm: Neumaier '80}
A connected, strongly regular, and edge-transitive graph is vertex-transitive.
\end{theorem}

\begin{corollary}
\label{corollary: s.c.+srg+e.t. imply v.t.}
If $\Gr{G}$ is a self-complementary, regular, and edge-transitive graph,
then $\Gr{G}$ is strongly regular and vertex-transitive.
\end{corollary}
\begin{proof}
Corollary~\ref{corollary: s.c.+srg+e.t. imply v.t.} follows as a special case of the
combination of Theorems~\ref{thm:symmetry and regular graphs are srg} and~\ref{thm: Neumaier '80}.
\end{proof}

\begin{remark}
\label{remark: regular self-complementary graphs}
This remark refers to subclasses of self-complementary regular graphs.
\begin{enumerate}
\item Self-complementary vertex-transitive graphs and self-complementary strongly regular graphs are nonequivalent classes.
These two subclasses of self-complementary regular graphs have gained interest in the literature, as evidenced by their study
in various works, including \cite{Farrugia99,Harary69,Mathon88,Muzychuk99,Peisert01,Rao85,Rosenberg82,Ruiz81,SaliS99,Sason23,Zelinka79}.
\item \label{item: Harary} Self-complementary regular graphs that are edge-transitive are also strongly regular (Lemma~4.3 in \cite{Rao85})
and vertex-transitive (Theorem 4.12 in \cite{Harary69}, and Corollary~\ref{corollary: s.c.+srg+e.t. imply v.t.} here).
\item The converse of the statement in Item~\ref{item: Harary} is false. Even if a graph is self-complementary, strongly regular, and
vertex-transitive, it is not necessarily edge-transitive. This can be confirmed using the SageMath software, by
showing that there are pseudo Latin square graphs, which are squared skew-Hadamard matrix graphs, possessing the mentioned properties
but lacking edge-transitivity.
\item Self-complementary vertex-transitive graphs are not necessarily strongly regular, as evidenced by a specific construction.
This construction involves a self-complementary and vertex-transitive circulant graph on 13 vertices, which is not strongly regular (see
Theorems~4.1 and~4.5 in \cite{Rao85}, and \cite{Ruiz81}).
\item A conference graph is a strongly regular graph of the form $\, \srg{4k+1}{2k}{k-1}{k}$ with $k \in \naturals$.
A self-complementary and strongly regular graph is a conference graph, as it can be verified by Item~\ref{item: complement of SRG}
of Theorem~\ref{theorem: parameters of srg}.
\item \label{item: open problem 1} It is an open problem to determine if there exists a self-complementary and
strongly regular graph that is not vertex-transitive (see page~88 in \cite{Farrugia99}).
\end{enumerate}
\end{remark}

The following results hold for two subclasses of self-complementary regular graphs.
\begin{theorem}[Theorem 2 in \cite{Mathon88}]
\label{thm:s.c. and srg graphs}
There exists a self-complementary strongly regular graph on $n$ vertices only if
$n \equiv 1 \bmod 4$ is expressible as a sum of two squares of integers.
\end{theorem}
Notably, Theorem~\ref{thm:s.c. and srg graphs} omits the 'if' part, which posits the
presence of such graphs for all cases meeting the specified congruence and sum of
squares criteria. This aspect remains an unresolved issue.

\begin{theorem}[Muzychuk \cite{Muzychuk99}, Rao \cite{Rao85}]
\label{thm:s.c. and v.t. graphs}
There exists a self-complementary vertex-transitive graph on $n$ vertices if and only if
$n \equiv 1 \bmod 4$ is expressible as a sum of two squares of integers.
\end{theorem}
\begin{itemize}
\item The 'if' part of Theorem~\ref{thm:s.c. and v.t. graphs} is due to Theorem 4.6 in \cite{Rao85}.
\item The 'only if' part of Theorem~\ref{thm:s.c. and v.t. graphs} is due to \cite{Muzychuk99}.
\end{itemize}

In light of Theorems~\ref{thm:s.c. and srg graphs} and~\ref{thm:s.c. and v.t. graphs}, the following
result is presented, which is a well-established finding in number theory (see, e.g., Chapter~4 in \cite{AignerZ18}).
\begin{theorem}[Sum of two squares of integers]
\label{thm:sum of two squares of integers}
A natural number $n \in \naturals$ can be represented as a sum of two squares of integers
if and only if every prime factor of the form $p = 4m+3$ appears with an even exponent in the prime
decomposition of $n$.
\end{theorem}

\subsubsection{Latin square graphs}
\label{subsubsection: Latin Square Graphs}
The class of Latin square graphs consists of two types: strongly regular graphs and complete graphs.
To facilitate their presentation, two key notions are briefly introduced.

\begin{definition}[Transversal 2-designs]
\label{def: transversal 2-designs}
Let $m \geq 2$ and $n \geq 1$ be integers. A {\em transversal 2-design}
of order $n$ and block size $m$ is a triple $(\set{X}, \mathscr{G}, \mathscr{B})$,
satisfying the following conditions:
\begin{enumerate}
\item $\set{X}$ is a set of $mn$ points.
\item $\mathscr{G} = \{\set{G}_1, \ldots, \set{G}_m\}$ is a partition of $\set{X}$ into $m$ subsets
(called groups), each containing $n$ points.
\item $\mathscr{B}$ is a class of subsets of $\set{X}$ (called blocks) such that
\begin{enumerate}[(a)]
\item Every block $\set{B} \in \mathscr{B}$ contains exactly one point from each group $\set{G}_j$ with $j \in \OneTo{m}$.
\item Every two points, not contained in the same group, belong together to a single block $\set{B} \in \mathscr{B}$.
\end{enumerate}
\end{enumerate}
Such a transversal 2-design is denoted by $\mathrm{TD}(m,n)$.
\end{definition}

\begin{theorem}
\label{thm: on Transversal 2-designs}
Let $(\set{X}, \mathscr{G}, \mathscr{B})$ be a transversal 2-design $\mathrm{TD}(m,n)$. Then,
\begin{enumerate}
\item Every block $\set{B} \in \mathscr{B}$ contains $m$ points.
\item Every point $x \in \set{X}$ belongs to $n$ blocks.
\item Every two distinct blocks $\set{B}, \set{B}' \in \mathscr{B}$ intersect in one element or in no element.
\item $\card{\mathscr{B}} = n^2$.
\end{enumerate}
\end{theorem}

\begin{definition}[Latin squares]
\label{def: Latin squares}
A Latin square of order $n$ is an $n \times n$ array filled with $n$ different symbols, each occurring exactly once
in each row and in each column.
\end{definition}
In the continuation, a Latin square of order $n$ is filled with the elements of the set $\OneTo{n} \eqdef \{1, \ldots, n\}$.
\begin{theorem}
\label{theorem: number of Latin squares of a given order}
The number of Latin squares of order $n$, denoted by $L(n)$, satisfies
\begin{align}
\label{eq1: number of Latin squares}
\frac{n!^{2n}}{n^{n^2}} \leq L(n) \leq \prod_{k=1}^n k!^{\frac{n}{k}},
\end{align}
and, consequently, the following exact asymptotic result holds:
\begin{align}
\label{eq2: number of Latin squares}
\lim_{n \to \infty} \frac{L(n)^{\frac{1}{n^2}}}{n} = \frac{1}{\mathrm{e}^2}.
\end{align}
\end{theorem}
\begin{proof}
See, e.g., Chapter~37 (pages~266-267) of \cite{AignerZ18}.
\end{proof}

\begin{definition}[Orthogonal Latin squares]
\label{def: orthogonal Latin squares}
Two Latin squares $\mathbf{A} = (a_{i,j})$ and $\mathbf{B} = (b_{i,j})$ of order $n$ are said to be
{\em orthogonal Latin squares} if for all $x,y \in \OneTo{n}$, there exists a unique position
$(i,j) \in \OneTo{n} \times \OneTo{n}$ such that $a_{i,j} = x$ and $b_{i,j}=y$. A set of Latin squares
of the same order are said to be {\em mutually orthogonal} if every pair of squares within the set are orthogonal.
\end{definition}

\begin{theorem}[On mutually orthogonal Latin squares]
\label{thm:mutually orthognal Latin squares}
Let $n \in \naturals$, and denote by $N(n)$ the largest number of mutually orthogonal Latin squares
of order $n$.
\begin{enumerate}
\item If $n$ is a prime power, then $N(n) = n-1$. In general, $N(n) \leq n-1$.
\item The existence of $m-2$ mutually orthogonal Latin squares of order $n$ is equivalent to the
existence of a transversal 2-design of order $n$ and block size $m$, namely $\mathrm{TD}(m,n)$.
\end{enumerate}
\end{theorem}

\begin{definition}[Latin square graphs]
\label{def: Latin square graphs}
Let $(\set{X}, \mathscr{G}, \mathscr{B})$ be a $\mathrm{TD}(m,n)$ transversal 2-design with integers
$m \geq 2$ and $n \geq 1$. The associated {\em Latin square graph} $\Gr{G} = (\mathsf{V}, \mathsf{E})$
has the vertex set $\mathsf{V} = \mathscr{B}$, and each two distinct vertices $\set{B}, \set{B}' \in \mathscr{B}$
are adjacent if $\card{\set{B} \cap \set{B}'} = 1$.
This graph is said to be an $\LSG{m}{n}$-graph.
\end{definition}

The following result characterizes Latin square graphs as either strongly regular or complete graphs.
\begin{theorem}
\label{thm:Latin square graphs are srg}
Let $\Gr{G}$ be an $\LSG{m}{n}$-graph with integers $m$ and $n$ such that $2 \leq m \leq n+1$.
\begin{enumerate}
\item If $2 \leq m \leq n$, then $\Gr{G}$ is strongly regular
$\srg{n^2}{\, m(n-1)}{\, m^2-3m+n}{\, m(m-1)}$.
\item If $m=n+1$, then $\Gr{G}$ is isomorphic to the complete graph $\CoG{n^2}$.
\end{enumerate}
\end{theorem}

\begin{example}[Lattice graphs]
\label{example:Lattice graphs}
For $n \geq 2$, the lattice graph on $n^2$ vertices is a Latin square graph $\LSG{2}{n}$ with vertex set
$\OneTo{n} \times \OneTo{n}$, where distinct vertices $(a,b)$ and $(c,d)$ are adjacent if $a=c$ or $b=d$.
Two vertices in $\LSG{2}{n}$ are adjacent if they lie in
the same row or column in the Latin square of order $n$. By Theorem~\ref{thm:Latin square graphs are srg} with $m=2$,
it is a strongly regular graph with parameters $\srg{n^2}{2(n-1)}{n-2}{2}$.
\end{example}

\begin{example}[Latin square graphs with block size 3]
\label{example:Latin square graphs with block size 3}
The Latin square graph $\LSG{3}{n}$, with $n \geq 2$, is constructed as follows.
The vertices in the graph represent the $n^2$ cells of the Latin square of order $n$,
and two vertices are adjacent if they lie in the same row or column, or if the cells
of these vertices hold the same element in $\OneTo{n}$. It is a specialization of the
family of Latin square graphs $\LSG{m}{n}$ with $m=3$, referring to the row, column
and symbol in the cell as the conditions for adjacency of any two vertices.
By Theorem~\ref{thm:Latin square graphs are srg} with $m=3$, it is a strongly regular
graph with parameters $\srg{n^2}{3(n-1)}{n}{6}$.
\end{example}

\subsubsection{Symplectic polar graphs}
\label{subsubsection: Symplectic Polar Graphs}
The symplectic polar graphs are a parametric family of strongly regular graphs with parameters
$\srg{v}{k}{\lambda}{\mu}$ given by
\begin{align}
\label{eq1: symplectic}
& v = \frac{q^{2n}-1}{q-1}, \\[0.1cm]
\label{eq2: symplectic}
& k = \frac{q(q^{2n-2}-1)}{q-1}, \\[0.1cm]
\label{eq3: symplectic}
& \lambda = \frac{q^{2n-2}-1}{q-1} - 2, \\[0.1cm]
\label{eq4: symplectic}
& \mu = \frac{q^{2n-2}-1}{q-1},
\end{align}
for all $n \in \naturals$ and $q \geq 2$ that is a prime power, so
$\lambda = \mu-2$ and $\mu = \frac{k}{q}$. That symplectic polar graph
is denoted by $\Sp(2n,q)$ (see Section 2.5 in \cite{BrouwerM22}).
The clique number of $\Sp(2n,q)$ is equal to $\frac{q^n-1}{q-1}$.
By Theorem~\ref{theorem: eigenvalues of srg} and the parameters of the strongly regular graph
in \eqref{eq1: symplectic}--\eqref{eq4: symplectic}, the eigenvalues of the symplectic polar graph
$\Sp(2n,q)$ (with respect to its adjacency matrix) are given by
\begin{itemize}
\item $\lambda_1 = k = \dfrac{q(q^{2n-2}-1)}{q-1}$ (the largest eigenvalue)
with multiplicity 1;
\item $\lambda_2 = r$ (the second-largest eigenvalue) with multiplicity $f$;
\item $\lambda_{\min} = s$ (the smallest eigenvalue) with multiplicity $g$,
\end{itemize}
where
\begin{align}
\label{eq: r - symplectic}
& r = q^{n-1}-1, \qquad f = \frac12 \, \biggl( \frac{q^{2n}-q}{q-1} + q^n \biggr), \\[0.2cm]
\label{eq: s - symplectic}
& s = -q^{n-1}-1, \qquad g = \frac12 \, \biggl( \frac{q^{2n}-q}{q-1} - q^n \biggr).
\end{align}

\subsection{New Results on the Shannon Capacity of Graphs}
\label{subsection: New Results on Shannon Capacity of Graphs}

In this subsection, some new results are presented in regard to the Shannon capacity of graphs.
The next theorem strengthens and extends Theorem~\ref{thm: Lovasz - s.c.-v.t. capacity}
(see Theorem~12 in \cite{Lovasz79_IT}).
\begin{theorem}[On the Shannon capacity of graphs]
\label{thm:extension of Thm. 12 by Lovasz}
Let $\Gr{G}$ be an undirected and simple graph on $n$ vertices.
\begin{enumerate}
\item  \label{item 1: extension of Thm. 12 by Lovasz}
If $\Gr{G}$ is a vertex-transitive or strongly regular graph, then
\begin{align}
\indnum{\Gr{G} \boxtimes \CGr{G}} = \Theta(\Gr{G} \boxtimes \CGr{G}) = \vartheta(\Gr{G} \boxtimes \CGr{G}) = n.
\end{align}
\item  \label{item 2: extension of Thm. 12 by Lovasz}
If $\Gr{G}$ is a conference graph, then $\, \vartheta(\Gr{G}) = \sqrt{n}$.
\item  \label{item 3: extension of Thm. 12 by Lovasz}
If $\Gr{G}$ is a self-complementary graph with $\indnum{\Gr{G}} = k$, then
\begin{align}
\label{eq: capacity bounds for s.c. graphs}
\sqrt{n} \leq \Theta(\Gr{G}) \leq 16 \, n^{\frac{k-1}{k+1}}.
\end{align}
\item   \label{item 4: extension of Thm. 12 by Lovasz}
If $\Gr{G}$ is a self-complementary graph that is vertex-transitive or strongly regular, then
\begin{align}
& \Theta(\Gr{G}) = \sqrt{n} = \vartheta(\Gr{G}), \label{eq1:13.03.24} \\
& \sqrt{\indnum{\Gr{G} \boxtimes \Gr{G}}} = \Theta(\Gr{G}).  \label{eq2:13.03.24}
\end{align}
Hence, the minimum Shannon capacity among all self-complementary graphs of a fixed order $n$
is achieved by those that are vertex-transitive or strongly regular, and this minimum is equal to $\sqrt{n}$.
\end{enumerate}
\end{theorem}
\begin{proof}
See Section~\ref{subsubsection: proof of extended theorem by Lovasz}.
\end{proof}

\begin{remark}[Discussion on Theorem~\ref{thm:extension of Thm. 12 by Lovasz}]
\label{remark: discussion - self-complementary graphs}
The following are comments in regard to Theorem~\ref{thm:extension of Thm. 12 by Lovasz}.
\begin{enumerate}
\item Item~\ref{item 1: extension of Thm. 12 by Lovasz} of Theorem~\ref{thm:extension of Thm. 12 by Lovasz}
extends Theorem~\ref{thm: Lovasz - s.c.-v.t. capacity} (i.e., Theorem~12 in \cite{Lovasz79_IT})
since it holds for all strongly regular graphs, in addition to vertex-transitive graphs. In general,
strongly regular graphs are not necessarily vertex-transitive (see Remark~\ref{remark: srg are not necessarily v.t.}).
\item In regard to Item~\ref{item 1: extension of Thm. 12 by Lovasz} of Theorem~\ref{thm:extension of Thm. 12 by Lovasz},
it should be noted that the complement of a vertex-transitive graph is vertex-transitive, and the complement of a strongly
regular graph is strongly regular. The strong product of vertex-transitive graphs is vertex-transitive,
whereas the strong product of strongly regular graphs is not necessarily strongly regular,
but only regular; more explicitly, the strong product of $d_1$-regular and $d_2$-regular graphs
is $d$-regular with $d = (1+d_1)(1+d_2)-1$. If $\Gr{G}$ is a strongly regular graph that
is not vertex-transitive, then the regular graph $\Gr{G} \boxtimes \CGr{G}$
may not be vertex-transitive nor strongly regular. As a concrete example, let $\Gr{G}$ be the Gritsenko
graph \cite{Gritsenko01}.
The graph $\Gr{G}$, a conference graph on 65~vertices, is not vertex- or edge-transitive.
Utilizing the SageMath software, it can be verified that the strong product $\Gr{G} \boxtimes \CGr{G}$
is not a strongly regular graph, nor is it vertex- or edge-transitive.
\item Item~\ref{item 2: extension of Thm. 12 by Lovasz} of Theorem~\ref{thm:extension of Thm. 12 by Lovasz} states that
the Lov\'{a}sz $\vartheta$-function of a conference graph on $n$ vertices is equal to $\sqrt{n}$. For a conference graph
$\Gr{G}$ on $n$ vertices, $d = \tfrac12(n-1)$ and $\lambda_{\min}(\Gr{G}) = -\tfrac12 (1+\sqrt{n})$,
so, by Theorem 9 in \cite{Lovasz79_IT} (see \eqref{eq: Lovasz79 - Theorem 9}), $\vartheta(\Gr{G}) \leq \sqrt{n}$
with an equality if the regular graph $\Gr{G}$ is edge-transitive.
Item~\ref{item 2: extension of Thm. 12 by Lovasz} of Theorem~\ref{thm:extension of Thm. 12 by Lovasz}
shows that the rightmost equality in \eqref{eq1:13.03.24} holds for all conference graphs, regardless of the
edge-transitivity property of the graph. There are many conference graphs that are not vertex-transitive, nor edge-transitive
(for a concrete example, see Remark~\ref{remark: srg are not necessarily v.t.}).
\item In regard to Item~\ref{item 4: extension of Thm. 12 by Lovasz} of Theorem~\ref{thm:extension of Thm. 12 by Lovasz},
recall that self-complementary vertex-transitive graphs and self-complementary strongly regular graphs are not equivalent
classes (see Remark~\ref{remark: regular self-complementary graphs}).
\end{enumerate}
\end{remark}

The next two theorems provide the Shannon capacity of two infinite subclasses
of strongly regular graphs (see Sections~\ref{subsubsection: Latin Square Graphs}
and~\ref{subsubsection: Symplectic Polar Graphs} for the presentation of these
families of graphs).
\begin{theorem}[The Shannon capacity of complements of Latin square graphs]
\label{thm:Latin square graphs}
Let $\Gr{G} = \LSG{m}{n}$ be a Latin square graph with $m, n \in \naturals$ and
$2 \leq m < n$. Then,
\begin{align}
\label{eq:18.08.23}
\indnum{\CGr{G}} = \Theta(\CGr{G}) = \vartheta(\CGr{G}) = n.
\end{align}
\end{theorem}
\begin{proof}
See Section~\ref{subsubsection: proof of theorem on capacity of complements of Latin square graphs}.
\end{proof}

\begin{theorem}[Graph invariants of symplectic polar graphs and their complements]
\label{thm:on the symplectic graphs}
Let $\Gr{G}$ be the symplectic polar graph $\Sp(2n,q)$, where $n \in \naturals$ and
$q \geq 2$ is a power of a prime. Then, the following holds:
\begin{enumerate}
\item
The independence number, Shannon capacity, and Lov\'{a}sz $\vartheta$-function of the complement graph
$\CGr{G}$ coincide, and they are given by
\begin{align}
\label{eq1:02.08.23}
\indnum{\CGr{G}} = \Theta(\CGr{G}) = \vartheta(\CGr{G}) = \frac{q^n-1}{q-1}.
\end{align}
\item The Lov\'{a}sz $\vartheta$-function of $\Gr{G}$, the chromatic number of $\CGr{G}$, and the
fractional chromatic number of $\CGr{G}$ coincide, and they are given by
\begin{align}
\label{eq2:02.08.23}
\vartheta(\Gr{G}) = \fchrnum{\CGr{G}} = \chrnum{\CGr{G}} = q^n + 1.
\end{align}
\end{enumerate}
\end{theorem}
\begin{proof}
See Section~\ref{subsubsection: proof of theorem on capacity of complements of symplectic graphs}.
\end{proof}

\begin{remark}[The independence number of symplectic polar graphs]
\label{remark:independence number of symplectic polar graphs}
By Proposition 2.5.4 in \cite{BrouwerM22}, the
independence number of the symplectic polar graph $\Gr{G} = \Sp(2n,q)$
(with $n \in \naturals$ and $q \geq 2$ that is a prime power)
satisfies
\begin{itemize}
\item $\indnum{\Gr{G}} = q+1$ if $n=1$ (an edgeless graph on $q+1$ vertices);
\item $\indnum{\Gr{G}} = 2n+1$ if $q=2$ and $n \in \naturals$;
\item $\indnum{\Gr{G}} = 7$ if $q=3$ and $n=2$;
\item $\indnum{\Gr{G}} \leq 15 \cdot 2^{n-3} - 2$ if $q=3$ and $n \geq 3$;
\item If $q \geq 4$ is a prime power and $n \geq 3$, then
\begin{align}
\label{eq2:10.10.23}
\indnum{\Gr{G}} \leq & \, \frac{q(q-1)^{n-3}-2}{q-2}
+ \tfrac12 q(q-1)^{n-3} \, \bigl( \sqrt{5q^4 + 6q^3 + 7q^2 + 6q + 1} - q^2 - q - 1 \bigr).
\end{align}
\end{itemize}
Hence, unless $n=1$ (resulting in an edgeless graph on $q+1$ vertices),
in general $\indnum{\Gr{G}} < q^n+1 = \vartheta(\Gr{G})$ (the last equality
holds by \eqref{eq2:02.08.23}).
\end{remark}

A new result on the Shannon capacity of two types of joins of graphs
is later provided in Section~\ref{subsection: main results - cospectral and nonisomorphic graphs}.
The presentation of the Shannon capacity of these types of graphs is deferred to the next
section since additional preliminary material is needed for that purpose.

In the following, Section~\ref{subsection: proofs of theorems on the Shannon capacity of graphs} provides
proofs for the results in this section, and Section~\ref{subsection: Shannon capacity of graphs - examples}
concludes with numerical experiments illustrating these results.

\subsection{Proofs}
\label{subsection: proofs of theorems on the Shannon capacity of graphs}
This section proves the results in
Section~\ref{subsection: New Results on Shannon Capacity of Graphs}. \vspace*{-0.1cm}

\subsubsection{Proof of Theorem~\ref{thm:extension of Thm. 12 by Lovasz}}
\label{subsubsection: proof of extended theorem by Lovasz}
Items~\ref{item 1: extension of Thm. 12 by Lovasz}--\ref{item 4: extension of Thm. 12 by Lovasz} of
Theorem~\ref{thm:extension of Thm. 12 by Lovasz} are proved as follows.
\begin{enumerate}
\item \label{item 1: proof of extended theorem by Lovasz}
For every undirected and simple graph $\Gr{G}$ on $n$ vertices,
\begin{eqnarray}
\label{eq4:06.08.23}
\indnum{\Gr{G} \boxtimes \CGr{G}} \geq n,
\end{eqnarray}
which holds since $\{(1,1), \ldots (n,n)\}$ is an independent set in the strong product graph $\Gr{G} \boxtimes \CGr{G}$.
Indeed, if $i, j \in \OneTo{n}$, $i \neq j$, and $\{i,j\} \in \E{\Gr{G}}$, then $\{i,j\} \not\in \E{\CGr{G}}$, which
implies that $\{(i,i), (j,j)\} \not\in \E{\Gr{G} \boxtimes \CGr{G}}$.
If $\Gr{G}$ is vertex-transitive or strongly regular, then by \eqref{eq: factorization}, \eqref{eq3:02.08.23},
and~\eqref{eq4:06.08.23},
\begin{align}
n &\leq \indnum{\Gr{G} \boxtimes \CGr{G}} \\
&\leq \Theta(\Gr{G} \boxtimes \CGr{G}) \\
&\leq \vartheta(\Gr{G} \boxtimes \CGr{G}) \\
&= \vartheta(\Gr{G}) \; \vartheta(\CGr{G}) \\
\label{eq1:16.03.24}
&= n, \\
\label{eq1:19.10.2023}
\Rightarrow \, \; & \indnum{\Gr{G} \boxtimes \CGr{G}} = \Theta(\Gr{G} \boxtimes \CGr{G})
= \vartheta(\Gr{G} \boxtimes \CGr{G}) = n.
\end{align}

\item \label{item 2: proof of extended theorem by Lovasz}
Let $\Gr{G}$ be a conference graph on $n$ vertices. Then, $\Gr{G}$ and $\CGr{G}$ are strongly regular graphs
with the same set of parameters $\SRG(n, \frac{n-1}{2}, \frac{n-1}{5}, \frac{n-1}{4})$. By Theorem~\ref{thm:Lovasz function of srg},
the Lov\'{a}sz $\vartheta$-function of a strongly regular graph is determined by its four parameters, so it follows that
$\vartheta(\Gr{G}) = \vartheta(\CGr{G})$. By \eqref{eq3:02.08.23} and the rightmost equality in \eqref{eq1:19.10.2023}, it
follows that $\vartheta(\Gr{G}) = \sqrt{n}$.

\item \label{item 3: proof of extended theorem by Lovasz}
If $\Gr{G}$ is a self-complementary graph on $n$ vertices, then
\begin{align}
\label{eq13:11.10.23}
\Theta(\Gr{G}) & \geq \sqrt{\indnum{\Gr{G} \boxtimes \Gr{G}}} \\
\label{eq14:11.10.23}
& = \sqrt{\indnum{\Gr{G} \boxtimes \CGr{G}}} \\
\label{eq15:11.10.23}
& \geq \sqrt{n},
\end{align}
where inequality \eqref{eq13:11.10.23} holds by \eqref{eq1:graph capacity}; equality \eqref{eq14:11.10.23} holds by the assumption
that $\Gr{G}$ is self-complementary, so $\Gr{G} \boxtimes \Gr{G} \cong \Gr{G} \boxtimes \CGr{G}$;
finally, \eqref{eq15:11.10.23} holds by \eqref{eq4:06.08.23}. This proves the leftmost inequality in \eqref{eq: capacity bounds for s.c. graphs}.
The rightmost inequality in \eqref{eq: capacity bounds for s.c. graphs} holds by combining \eqref{eq: Lovasz 79 - theorem 1} and
\eqref{eq12:11.10.23} (it should be noted that inequality \eqref{eq12:11.10.23} holds by Theorem~11.18 in \cite{Lovasz19},
irrespectively of the self-complementary property of the graph $\Gr{G}$).

\item  \label{item 4: proof of extended theorem by Lovasz}
If a self-complementary graph $\Gr{G}$ on $n$ vertices is vertex-transitive or strongly regular, then
by \eqref{eq1:16.03.24}
\begin{align}
n = \vartheta(\Gr{G}) \; \vartheta(\CGr{G}) = \vartheta(\Gr{G})^2,
\end{align}
which gives $\vartheta(\Gr{G}) = \sqrt{n}$.
Combined with Item~\ref{item 3: proof of extended theorem by Lovasz} of this proof, it gives
\begin{align}
\label{eq2:19.10.2023}
& \sqrt{n} \leq \Theta(\Gr{G}) \leq \vartheta(\Gr{G}) = \sqrt{n},
\end{align}
so
\begin{align}
\label{eq3:19.10.2023}
\Theta(\Gr{G}) = \vartheta(\Gr{G}) = \sqrt{n}.
\end{align}
Finally, combining \eqref{eq1:19.10.2023} and \eqref{eq3:19.10.2023} gives
\begin{align}
\sqrt{\indnum{\Gr{G} \boxtimes \Gr{G}}} = \sqrt{n} = \Theta(\Gr{G}).
\end{align}
\end{enumerate}

\subsubsection{Proof of Theorem~\ref{thm:Latin square graphs}}
\label{subsubsection: proof of theorem on capacity of complements of Latin square graphs}

The Latin square graph $\LSG{m}{n}$, with $m, n \in \naturals$ and $2 \leq m < n$,
has the property that each of its edges is contained in a clique of size $n$, so
\begin{align}
\label{eq1:14.03.24}
\indnum{\CGr{G}} \geq n.
\end{align}
By Theorem~\ref{thm:Latin square graphs are srg}, the graph $\Gr{G} = \LSG{m}{n}$ is strongly regular with the parameters
\begin{align}
\label{eq15b:11.10.23}
& \card{\V{\Gr{G}}} = n^2, \quad d = m(n-1), \quad \lambda = m^2-3m+n, \quad \mu = m(m-1).
\end{align}
By Theorem~\ref{thm:Lovasz function of srg}, the Lov\'{a}sz $\vartheta$-function of $\Gr{G}$ is then equal to
\begin{align}
\label{eq16:11.10.23}
\vartheta(\Gr{G}) = \frac{\card{\V{\Gr{G}}} \; (t+\mu-\lambda)}{2d+t+\mu-\lambda}
\end{align}
with
\begin{align}
\label{eq17:11.10.23}
t & \eqdef \sqrt{(\mu-\lambda)^2 + 4(d-\mu)} \\
\label{eq18:11.10.23}
& = \sqrt{(2m-n)^2 + 4m(n-m)} \\
\label{eq19:11.10.23}
& = n,
\end{align}
where equality \eqref{eq17:11.10.23} holds by \eqref{eq: t parameter - srg}, equality \eqref{eq18:11.10.23}
is derived from \eqref{eq15b:11.10.23}, and equality \eqref{eq19:11.10.23} is obtained by straightforward algebra
(specifically through the cancellation of $m$). It therefore follows from
\eqref{eq15b:11.10.23}--\eqref{eq19:11.10.23} that
\begin{align}
\label{eq20:11.10.23}
\vartheta(\Gr{G}) &= \frac{\card{\V{\Gr{G}}} \; (t+\mu-\lambda)}{2d+t+\mu-\lambda} \\
\label{eq21:11.10.23}
&= \frac{n^2 \bigl[ n + m(m-1) - \bigl((m-1)(m-2)+n-2 \bigr) \bigr]}{2m(n-1)+n +m(m-1) - \bigl((m-1)(m-2)+n-2\bigr)} \\
\label{eq22:11.10.23}
&= \frac{2mn^2}{2mn} = n.
\end{align}
Since $\Gr{G}$ is a strongly regular graph on $n^2$ vertices, by
Corollary~\ref{corollary:identity for the Lovasz function of srg and its complement},
\begin{align}
\label{eq24:11.10.23}
\vartheta(\CGr{G}) = \frac{\card{\V{\Gr{G}}}}{\vartheta(\Gr{G})} = \frac{n^2}{n} = n.
\end{align}
Finally, since the Shannon capacity of a graph is bounded between its independence number
and its Lov\'{a}sz $\vartheta$-function, \eqref{eq:18.08.23} follows from
\eqref{eq1:14.03.24} and \eqref{eq24:11.10.23}; hence, \eqref{eq1:14.03.24}
holds with equality. Since the Lov\'{a}sz $\vartheta$-function of strongly regular graphs only
depends on their four parameters ($n, d, \lambda, \mu$), it is identical for all (nonisomorphic)
strongly regular graphs with the same set of four parameters. Hence, $\vartheta(\Gr{G})=n$ holds
for all the pseudo Latin square graphs $\PLG{m}{n}$, which form the
class of all the strongly regular graphs with four parameters that are identical to those of the
Latin square graph $\LSG{m}{n}$.

\subsubsection{Proof of Theorem~\ref{thm:on the symplectic graphs}}
\label{subsubsection: proof of theorem on capacity of complements of symplectic graphs}

Let $\Gr{G}$ be the symplectic polar graph $\Sp(2n,q)$, with $n \in \naturals$ and $q \geq 2$ that is a prime power.
By Section 2.5.4 in \cite{BrouwerM22},
\begin{align}
\label{eq25:11.10.23}
\indnum{\CGr{G}} = \clnum{\Gr{G}} = \frac{q^n-1}{q-1}.
\end{align}

The complement graph of a strongly regular graph, $\SRG(\nu,k,\lambda,\mu)$, is a strongly
regular graph with parameters $\SRG(v,v-k-1,v-2k+\mu-2,v-2k+\lambda)$ (see Item~\ref{item: complement of SRG}
of Theorem~\ref{theorem: parameters of srg}).
Substituting $v, k, \lambda, \mu$ in \eqref{eq1: symplectic}--\eqref{eq4: symplectic}
gives that $\CGr{G}$ is a strongly regular graph on $v$ vertices with the parameters
\begin{align}
\label{eq: H complement symplectic graph}
\vspace*{0.1cm}
\SRG \, \Biggl(\frac{q^{2n}-1}{q-1}, q^{2n-1}, q^{2n-2} (q-1), q^{2n-2} (q-1)\Biggr).
\end{align}
The complement graph $\CGr{G}$ has the 3 distinct eigenvalues $q^{2n-1}, \, -1-s = q^{n-1}$, and $-1-r = -q^{n-1}$ with
multiplicities~1, $g$, and $f$, respectively (see \eqref{eq: r - symplectic}--\eqref{eq: s - symplectic}).
Substituting the four parameters of the strongly regular graph $\Gr{G} = \SRG(v,k,\lambda,\mu)$, as given in
\eqref{eq1: symplectic}--\eqref{eq4: symplectic}, into \eqref{eq: t parameter - srg} gives
\begin{align}
\label{eq26:11.10.23}
t &= \sqrt{(\mu-\lambda)^2 + 4(k-\mu)} \\[0.1cm]
\label{eq27:11.10.23}
&= 2 \sqrt{k-\mu+1} \\[0.1cm]
\label{eq28:11.10.23}
&= 2 \sqrt{\frac{q(q^{2n-2}-1)}{q-1} - \frac{q^{2n-2}-1}{q-1} + 1} \\[0.1cm]
\label{eq29:11.10.23}
&= 2 q^{n-1},
\end{align}
where equality \eqref{eq27:11.10.23} holds, by \eqref{eq3: symplectic}--\eqref{eq4: symplectic}, since $\lambda = \mu-2$;
equality \eqref{eq28:11.10.23} further holds by \eqref{eq2: symplectic}.
Consequently, the Lov\'{a}sz $\vartheta$-function of the complement graph $\CGr{G}$ is equal to
\begin{align}
\label{eq30:11.10.23}
\vartheta(\CGr{G}) &= 1 + \dfrac{2k}{t+\mu-\lambda} \\[0.2cm]
\label{eq31:11.10.23}
&= 1 + \frac{2q (q^{2n-2}-1)}{(q-1) \, (2q^{n-1}+2)} \\[0.2cm]
\label{eq32:11.10.23}
&= \frac{q^n-1}{q-1},
\end{align}
where \eqref{eq30:11.10.23} holds by \eqref{eq: Lovasz-theta srg comp.}, and \eqref{eq31:11.10.23} follows from
the combination of \eqref{eq1: symplectic}--\eqref{eq4: symplectic} and \eqref{eq29:11.10.23}.
Since, by \eqref{eq25:11.10.23} and \eqref{eq32:11.10.23}$, \indnum{\CGr{G}} = \vartheta(\CGr{G}) = \frac{q^n-1}{q-1}$,
the Shannon capacity of the complement graph $\CGr{G}$ is also equal to that common value, i.e.,
$\Theta(\CGr{G}) = \frac{q^n-1}{q-1}$. This proves \eqref{eq1:02.08.23}.

By Theorem 2.4 in \cite{TangW06}, the chromatic number of the complement graph $\CGr{G}$ is given by
\begin{align}
\label{eq33:11.10.23}
\chrnum{\CGr{G}} = q^n+1,
\end{align}
and since $\Gr{G}$ and $\CGr{G}$ are strongly regular graphs on $v$ vertices, the combination of equalities
\eqref{eq3:02.08.23}, \eqref{eq1: symplectic}, and \eqref{eq32:11.10.23} give
\begin{align}
\vartheta(\Gr{G}) = \frac{v}{\vartheta(\CGr{G})}
= \dfrac{\dfrac{q^{2n}-1}{q-1}}{\dfrac{q^n-1}{q-1}}
\label{eq34:11.10.23}
= q^n+1.
\end{align}
Consequently, from \eqref{eq33:11.10.23} and \eqref{eq34:11.10.23},
\begin{align}
\label{eq35:11.10.23}
\vartheta(\Gr{G}) = \chrnum{\CGr{G}} = q^n+1.
\end{align}
This finally proves \eqref{eq2:02.08.23} by combining \eqref{eq1a: sandwich}, which gives
$\vartheta(\Gr{G}) \leq \fchrnum{\CGr{G}} \leq \chrnum{\CGr{G}}$, and \eqref{eq35:11.10.23}.

\subsection{Examples for the Results in Section~\ref{subsection: New Results on Shannon Capacity of Graphs}}
\label{subsection: Shannon capacity of graphs - examples}

\begin{example}[Paley graphs]
\label{example:Shannon capacity of Paley graphs}
Let $q$ be a prime power such that $q \equiv 1 \bmod 4$. That is,
$q$ is either an arbitrary power of a prime congruent to~1 modulo~4 or it is an even
power of a prime congruent to~3 modulo~4.
The vertex set of a Paley graph of order $q$ is $\{0, 1, \ldots, q-1\}$, and any two
such distinct vertices are adjacent if their difference is a quadratic residue modulo $q$.
Paley graphs are known to be a class of self-complementary, vertex-transitive, and strongly regular graphs
with parameters $\srg{q}{\tfrac12 (q-1)}{\tfrac14 (q-5)}{\tfrac14 (q-1)}$ (a Paley graph
is therefore a conference graph).
By Item~\ref{item 4: extension of Thm. 12 by Lovasz} of Theorem~\ref{thm:extension of Thm. 12 by Lovasz},
the Shannon capacity of a Paley graph $\Gr{G}_q$ of order $q$ is given by $\Theta(\Gr{G}_q) = \sqrt{q}$, and
\begin{align}
\label{eq: Paley graph}
\indnum{\Gr{G}_q \boxtimes \Gr{G}_q} = q = \Theta(\Gr{G}_q)^2.
\end{align}
If $q$ is an even power of a prime, then $\indnum{\Gr{G}_q} = \sqrt{q}$ (see \cite{BakerEHW96,Blokhuis84}),
so $\indnum{\Gr{G}_q \boxtimes \Gr{G}_q} = q = \indnum{\Gr{G}_q}^2$. If $q$ is
equal to a prime number, then the determination of the independence number of
such a Paley graph is an unsolved number-theoretic problem, and it is conjectured
that $\indnum{\Gr{G}_q} = O\bigl((\log q)^2\bigr)$ (see Example~11.14 in \cite{Lovasz19},
which refers to the case where $q \equiv 1 \bmod 4$ is a prime).
\end{example}

\begin{example}[Shannon capacity of the complements of Latin square graphs with block size~3]
\label{example:Shannon Capacity of the Complements of Latin Square Graphs with Block Size 3}
By Example~\ref{example:Latin square graphs with block size 3}, for any
integer $n \geq 3$, a Latin square graph with block size $m=3$ is a
strongly regular graph whose parameters are given by $\srg{n^2}{3(n-1)}{n}{6}$.
By Item~\ref{item: complement of SRG} of Theorem~\ref{theorem: parameters of srg} (see equalities
\eqref{eq: n-complement}--\eqref{eq: mu-complement}), the complement of
this graph is a strongly regular graph whose parameters are given by
$\srg{n^2}{(n-2)(n-1)}{(n-3)^2+1}{(n-3)(n-2)}$. In light of
Theorem~\ref{thm:Latin square graphs}, the Shannon capacity of these
graphs are given in Table~\ref{Table II: Latin square graphs} for all $3 \leq n \leq 16$.
\begin{table}[h!t!]
\caption{\label{Table II: Latin square graphs} \centering{The Shannon capacity of Latin square
graph complements $\CGr{\LSG{3}{n}}$ (Example~\ref{example:Shannon Capacity of the Complements of Latin Square Graphs with Block Size 3}).}}
\vspace*{0.1cm}
\renewcommand{\arraystretch}{1.5}
\centering
\begin{tabular}{||c|l|c||}
\hline
$n$ & $\hspace*{0.7cm} \overline{\LSG{3}{n}}$ & $\Theta(\overline{\LSG{3}{n}})$ \\[0.1cm] \hline
$3$ & $\srg{9}{2}{1}{0}$     & $3$ \\
$4$ & $\srg{16}{6}{2}{2}$    & $4$ \\
$5$ & $\srg{25}{12}{5}{6}$   & $5$ \\
$6$ & $\srg{36}{20}{10}{12}$ & $6$ \\
$7$ & $\srg{49}{30}{17}{20}$ & $7$ \\
$8$ & $\srg{64}{42}{26}{30}$ & $8$ \\
$9$ & $\srg{81}{56}{37}{42}$ & $9$ \\
$10$ & $\srg{100}{72}{50}{56}$ & $10$ \\
$11$ & $\srg{121}{90}{65}{72}$ & $11$ \\
$12$ & $\srg{144}{110}{82}{90}$ & $12$ \\
$13$ & $\srg{169}{132}{101}{110}$ & $13$ \\
$14$ & $\srg{196}{156}{122}{132}$ & $14$ \\
$15$ & $\srg{225}{182}{145}{156}$ & $15$ \\
$16$ & $\srg{256}{210}{170}{182}$ & $16$ \\[0.1cm] \hline
\end{tabular}
\end{table}
\end{example}

\begin{example}[Shannon capacity of the complements of the symplectic polar graphs]
\label{example: Shannon capacity of the complements of symplectic polar graphs}
\begin{table}[h!t!]
\caption{\label{Table: symplectic graphs II} \centering{The Shannon capacity of complements of
symplectic polar graphs (Example~\ref{example: Shannon capacity of the complements of symplectic polar graphs}).}}
\vspace*{0.1cm}
\renewcommand{\arraystretch}{1.5}
\centering
\begin{tabular}{||c|c|l|c||}
\hline
$n$ & $q$ & $\hspace*{0.7cm} \Gr{H} = \overline{\Sp(2n,q)}$ & $\Theta(\Gr{H})$ \\[0.1cm] \hline
$3$ & $2$ & $\srg{63}{32}{16}{16}$ & $7$ \\
$3$ & $3$ & $\srg{364}{243}{162}{162}$ & $13$ \\
$3$ & $4$ & $\srg{1365}{1024}{768}{768}$ & $21$ \\
$3$ & $5$ & $\srg{3906}{3125}{2500}{2500}$ & $31$ \\
$3$ & $7$ & $\srg{19608}{16807}{14406}{14406}$ & $57$ \\
$4$ & $2$ & $\srg{255}{128}{64}{64}$ & $15$ \\
$4$ & $3$ & $\srg{3280}{2187}{1458}{1458}$ & $40$ \\
$4$ & $4$ & $\srg{21845}{16384}{12288}{12288} $ & $85$
\\[0.1cm] \hline
\end{tabular}
\end{table}
The symplectic polar graphs form an infinite subfamily of the strongly regular graphs;
they are briefly addressed in Section~\ref{subsubsection: Symplectic Polar Graphs}
(see also Section 2.5 in \cite{BrouwerM22} for more details).
In light of Theorem~\ref{thm:on the symplectic graphs}, the independence number,
Shannon capacity, and the Lov\'{a}sz $\vartheta$-function of the complements of symplectic
polar graphs coincide, and their common value is given by \eqref{eq1:02.08.23}.
Table~\ref{Table: symplectic graphs II} provides
the Shannon capacity of the complements of the symplectic polar graphs $\Gr{H} = \overline{\Sp(2n,q)}$
where $q \geq 2$ is a power of a prime and $n \in \naturals$. The set of parameters of
these strongly regular graphs is given in \eqref{eq: H complement symplectic graph}.
\end{example}
\break

\section{On Cospectral and Nonisomorphic Graphs, and on Joins of Graphs}
\label{section: cospectral and nonisomorphic graphs}

The construction of cospectral and nonisomorphic graphs is a nontrivial problem in
spectral graph theory, which has gained interest in the literature (see
\cite{AbiadH12,BlazsikaCH15,ButlerG11,vanDamH03,vanDamH09,Dutta20,DuttaA20,GodsilM82,HaemersE04,Hamud23,HamudB24,LuMZ23}
and references therein).

The present section employs the Lov\'{a}sz $\vartheta$-function of graphs to unveil novel insights
into regular and irregular, cospectral and nonisomorphic graphs. Additionally, it explores
new properties of two types of joins of graphs, recently introduced in \cite{Hamud23,HamudB24},
within the context of irregular, cospectral and nonisomorphic graphs.

To facilitate comprehension, we introduce specific preliminaries essential for the present section in
Section~\ref{subsection: preliminaries - cospectral and nonisomorphic graphs}. Following this,
the new results are presented in Section~\ref{subsection: main results - cospectral and nonisomorphic graphs}.
Subsequently, we delve into the proofs of these results in Section~\ref{subsection: proofs - cospectral and nonisomorphic graphs}.
Finally, in Section~\ref{subsection: examples - cospectral and nonisomorphic graphs}, we engage
in discussions, offer illustrative examples, and conduct numerical experimentation to provide further
insights into the new results.

\subsection{Specialized Preliminaries for Section~\ref{section: cospectral and nonisomorphic graphs}}
\label{subsection: preliminaries - cospectral and nonisomorphic graphs}

The preliminary material in this subsection forms a continuation to Section~\ref{subsubsection: Matrices associated with graphs}.
\begin{definition}[Cospectral graphs]
\label{definition: cospectral graphs}
Let $\Gr{G}$ and $\Gr{H}$ be graphs. Then,
\begin{itemize}
\item
$\Gr{G}$ and $\Gr{H}$ are said to be $X$-cospectral, with $X \in \{A, L, Q, \mathcal{L}\}$,
if these graphs have an identical $X$-spectrum.
\item
$\Gr{G}$ and $\Gr{H}$ are said to be $X$-NICS if these graphs are nonisomorphic and $X$-cospectral.
\item
Let $\set{S} \subseteq \{A, L, Q, \mathcal{L}\}$. The graphs $\Gr{G}$ and $\Gr{H}$ are
said to be $\set{S}$-NICS if they are $X$-NICS for every $X \in \set{S}$.
\end{itemize}
\end{definition}

\begin{remark}[Relations between spectra of regular graphs]
\label{remark: Relations Between Spectra of Regular Graphs}
If $\Gr{G}$ is a $d$-regular graph on $n$ vertices, then by
\eqref{eq: laplacian}--\eqref{eq2: Normalized Laplacian matrix}
and \eqref{eq2:26.09.23}--\eqref{eq5:26.09.23}, for all $i \in \OneTo{n}$
\begin{align}
\label{eq6:26.09.23}
\mu_i(\Gr{G}) = d - \lambda_i(\Gr{G}), \qquad
\nu_i(\Gr{G}) = d + \lambda_i(\Gr{G}), \qquad
\delta_i(\Gr{G}) = 1 - \frac{\lambda_i(\Gr{G})}{d}.
\end{align}
If $\Gr{G}$ and $\Gr{H}$ are regular graphs that are $X$-cospectral for {\em some}
$X \in \{A, L, Q, \mathcal{L}\}$, then it follows from the equalities in \eqref{eq6:26.09.23}
that $\Gr{G}$ and $\Gr{H}$ are $X$-cospectral for {\em every} $X \in \{A, L, Q, \mathcal{L}\}$.
Therefore, for regular graphs, the property of being cospectral remains unchanged regardless
of the matrix chosen among $\A, \LM, \Q$, and ${\bf{\mathcal{L}}}$.
\end{remark}

\begin{definition}[Seidel switching]
\label{definition: Seidel switching}
Let $\Gr{G}$ be a graph with the vertex set $\V{\Gr{G}}$, and let $\set{X} \subseteq \V{\Gr{G}}$.
Constructing a new graph $\hat{\Gr{G}}$ by preserving all the edges in $\Gr{G}$ between
vertices in $\set{X}$ and all the edges in $\Gr{G}$ between vertices in $\V{\Gr{G}} \setminus \set{X}$,
while modifying adjacency and nonadjacency between vertices in $\set{X}$ and $\V{\Gr{G}} \setminus \set{X}$,
is referred to as Seidel switching.
\end{definition}

The next result shows the importance of Seidel switching for the construction of regular, connected,
and cospectral graphs.
\begin{theorem}[Seidel switching and cospectraility]
\label{theorem: Seidel switching}
Let $\Gr{G}$ be a connected and $d$-regular graph, and let $\hat{\Gr{G}}$ be obtained from $\Gr{G}$ by
Seidel switching. If $\hat{\Gr{G}}$ remains connected and $d$-regular, then $\Gr{G}$ and $\hat{\Gr{G}}$
are cospectral graphs.
\end{theorem}

\begin{definition}[Determination of a graph by its spectrum]
\label{definition: Determination of a Graph by its Spectrum}
A graph $G$ is determined by its $X$-spectrum if every graph $\Gr{H}$ that is
$X$-cospectral with $\Gr{G}$ is also isomorphic to $\Gr{G}$.
\end{definition}

\begin{theorem}[\hspace*{-0.1cm} \cite{vanDamH03}]
\label{theorem 1: van Dam and Haemers, 2003}
If $\Gr{G}$ is a regular graph, then the following statements are equivalent:
\begin{itemize}
\item $\Gr{G}$ is determined by its adjacency spectrum ($A$-spectrum);
\item $\Gr{G}$ is determined by its Laplacian spectrum ($L$-spectrum);
\item $\Gr{G}$ is determined by its signless Laplacian spectrum ($Q$-spectrum);
\item $\Gr{G}$ is determined by its normalized Laplacian spectrum ($\mathcal{L}$-spectrum).
\end{itemize}
\end{theorem}

\begin{theorem}[\hspace*{-0.12cm} \cite{vanDamH03}]
\label{theorem 2: van Dam and Haemers, 2003}
All regular graphs on less than 10 vertices are determined by their spectrum. Additionally,
there are four pairs of nonisomorphic and cospectral regular graphs on 10 vertices.
\end{theorem}

The reader is referred to \cite{AbdianBF20,BermanCCLZ18,vanDamH03,vanDamH09,Hamud23,LiuSD14,LiuLu15,LiuZG08,OmidiT07}
for analyses on graphs determined by their spectra with respect to their adjacency,
Laplacian, signless Laplacian, or normalized Laplacian matrices.

\begin{example}[Irregular NICS graphs]
\label{example: NICS graphs}
Figure~\ref{figure: NICS graphs} displays four pairs of irregular, nonisomorphic, and cospectral (NICS) graphs
with respect to a single matrix among the matrices $\A, \LM, \Q$, and $\mathcal{L}$.
Notably, each of these pairs of graphs are no longer cospectral with respect to the other three matrices \cite{ButlerG11}.
It can be verified (e.g., by the Sage Math software \cite{SageMath}) that the common characteristic polynomials
of the pairs of $A$-NICS, $L$-NICS, $Q$-NICS, and $\mathcal{L}$-NICS graphs in Figure~\ref{figure: NICS graphs}
are, respectively, given by
\begin{align}
\label{eq1:char-poly}
& f_{\A}(x) = x^5 - 4x^3, \\
\label{eq2:char-poly}
& f_{\LM}(x) = x^6 - 14x^5 + 73x^4 - 176x^3 + 192x^2-72x, \\
\label{eq3:char-poly}
& f_{\Q}(x) = x^4 - 6x^3 + 9x^2 - 4x, \\
\label{eq4:char-poly}
& f_{\mathcal{L}}(x) = x^4  - 4x^3 + 5x^2 - 2x.
\end{align}
\begin{figure}[h!t!]
\centering
\includegraphics[width=12cm]{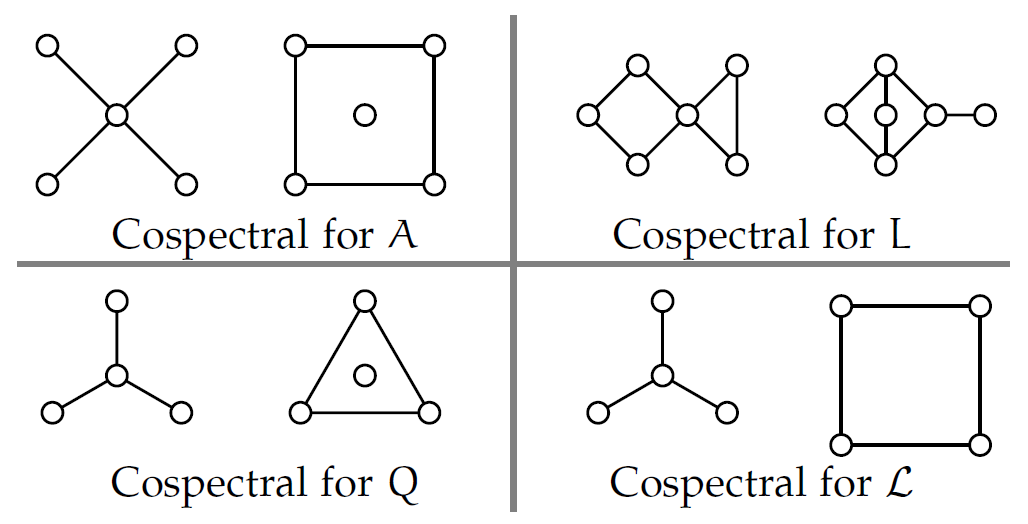}
\caption{\label{figure: NICS graphs} \centering{Pairs of $A$-NICS, $L$-NICS, $Q$-NICS, and $\mathcal{L}$-NICS graphs \cite{ButlerG11}.
Each pair of graphs is cospectral with respect to only one of the four matrices.}}
\end{figure}
\end{example}

By Theorem~\ref{theorem 2: van Dam and Haemers, 2003}, there is no pair of cospectral and
nonisomorphic regular graphs on less than 10 vertices. The next example shows a pair of
regular NICS graphs on 10 vertices.
\begin{example}[Regular NICS graphs \cite{vanDamH03}]
\label{example: regular NICS graphs with 10 vertices}
Let $\Gr{G}$ and $\Gr{H}$ be the regular graphs depicted on the left and right-hand sides of
Figure~\ref{fig:vanDamH03}, respectively. These are 4-regular graphs on 10 vertices, and they
are cospectral while being nonisomorphic \cite{vanDamH03}. Indeed, the identical $A$-characteristic
polynomials of $\Gr{G}$ and $\Gr{H}$ are given by
\begin{align}
\label{eq1:12.10.23}
f_{\A}(x) = x^{10} - 20 x^8 - 16 x^7 + 110 x^6 + 136 x^5 - 180 x^4 - 320 x^3 + 9 x^2 + 200 x + 80,
\end{align}
\begin{figure}[h!t!b!]
\centering
\includegraphics[width=11.5cm]{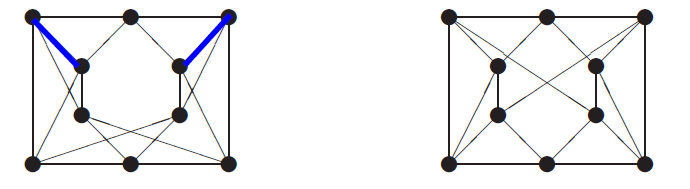}
\caption{\label{fig:vanDamH03} \centering{A pair of cospectral and nonisomorphic 4-regular graphs on 10 vertices \cite{vanDamH03}.
Denote the left and right plots by $\Gr{G}$ and $\Gr{H}$, respectively.}}
\end{figure}
so these regular graphs are cospectral. Alternatively, the cospectrality of the regular graphs $\Gr{G}$ and $\Gr{H}$ can be
inferred from Theorem~\ref{theorem: Seidel switching}, without needing to compute their characteristic polynomials or eigenvalues.
This inference is made by verifying the isomorphism of $\Gr{H}$ and the Seidel switching of $\Gr{G}$ with respect to its four corners,
while also ensuring that the latter graph and $\Gr{G}$ are both connected and 4-regular. However, the cospectral graphs $\Gr{G}$ and
$\Gr{H}$ are nonisomorphic because each of the two blue edges in $\Gr{G}$ belongs to three triangles, whereas no such edge exists in
$\Gr{H}$ (see Figure~\ref{fig:vanDamH03}).
\end{example}

It is of interest to explore constructions of irregular $\{A, L, Q, \mathcal{L}\}$-NICS graphs, as
they exhibit cospectrality concerning four matrices that include the adjacency, Laplacian, signless
Laplacian, and normalized Laplacian matrices. However, despite this cospectrality, the graphs
are nonisomorphic. To delve deeper into this topic, we refer to the recent study by Berman and
Hamud \cite{Hamud23,HamudB24}, which begins with essential definitions.

\begin{definition}[Join of graphs]
\label{definition: join of graphs}
Let $\Gr{G}$ and $\Gr{H}$ be two graphs with disjoint vertex sets.
The join of $\Gr{G}$ and $\Gr{H}$ is defined to be their disjoint union, together with all the edges that connect
the vertices in $\Gr{G}$ with the vertices in $\Gr{H}$. It is denoted by $\Gr{G} \vee \Gr{H}$.
\end{definition}

\begin{remark}[Join and disjoint union of graphs]
\label{remark: join and disjoint union of graphs}
Since the edge set of the disjoint union of graphs $\Gr{G}$ and $\Gr{H}$, denoted by $\Gr{G} + \Gr{H}$,
does not include any edge that connects a vertex in $\Gr{G}$ with a vertex in $\Gr{H}$ (see
Definition~\ref{def:disjoint_union_graphs}), the following equality holds:
\begin{align}
\label{eq1:28.09.23}
\CGr{\Gr{G} \vee \Gr{H}} = \CGr{G} + \CGr{H}.
\end{align}
\end{remark}

\begin{definition}[Neighbors splitting join of graphs \cite{LuMZ23}]
\label{definition: neighbors splitting join of graphs}
Let $\Gr{G}$ and $\Gr{H}$ be graphs with disjoint vertex sets, and let
$\V{\Gr{G}} = \{v_1, \ldots, v_n\}$. The {\em neighbors splitting (NS) join}
of $\Gr{G}$ and $\Gr{H}$ is obtained by adding vertices $v'_1, \ldots, v'_n$
to the vertex set of $\Gr{G} \vee \Gr{H}$ and connecting $v'_i$ to $v_j$ if
and only if $\{v_i, v_j\} \in \E{\Gr{G}}$.
The NS join of $\Gr{G}$ and $\Gr{H}$ is denoted by $\Gr{G} \NS \Gr{H}$.
\end{definition}

\begin{definition}[Nonneighbors splitting join of graphs \cite{Hamud23,HamudB24}]
\label{definition: nonneighbors splitting join of graphs}
Let $\Gr{G}$ and $\Gr{H}$ be graphs with disjoint vertex sets, and let
$\V{\Gr{G}} = \{v_1, \ldots, v_n\}$. The {\em nonneighbors splitting (NNS) join}
of $\Gr{G}$ and $\Gr{H}$ is obtained by adding vertices $v'_1, \ldots, v'_n$
to the vertex set of $\Gr{G} \vee \Gr{H}$ and connecting $v'_i$ to $v_j$, with
$i \neq j$, if and only if $\{v_i, v_j\} \not\in \E{\Gr{G}}$.
The NNS join of $\Gr{G}$ and $\Gr{H}$ is denoted by $\Gr{G} \NNS \Gr{H}$.
\end{definition}

\begin{remark}
In general, $\Gr{G} \NS \Gr{H} \not\cong \Gr{H} \NS \Gr{G}$ and
$\Gr{G} \NNS \Gr{H} \not\cong \Gr{H} \NNS \Gr{G}$ (unless $\Gr{G} \cong \Gr{H}$).
\end{remark}

\begin{example}[NS and NNS join of graphs \cite{HamudB24}]
\label{example: NS and NNS Join of Graphs}
Figure~\ref{fig:BermanH23-Fig. 2.2} shows the NS and NNS
joins of the path graphs $\PathG{4}$ and $\PathG{2}$, denoted by $\PathG{4} \NS \PathG{2}$
and $\PathG{4} \NNS \PathG{2}$, respectively.
\vspace*{-0.1cm}
\begin{figure}[h!t!b!]
\centering
\includegraphics[width=11cm]{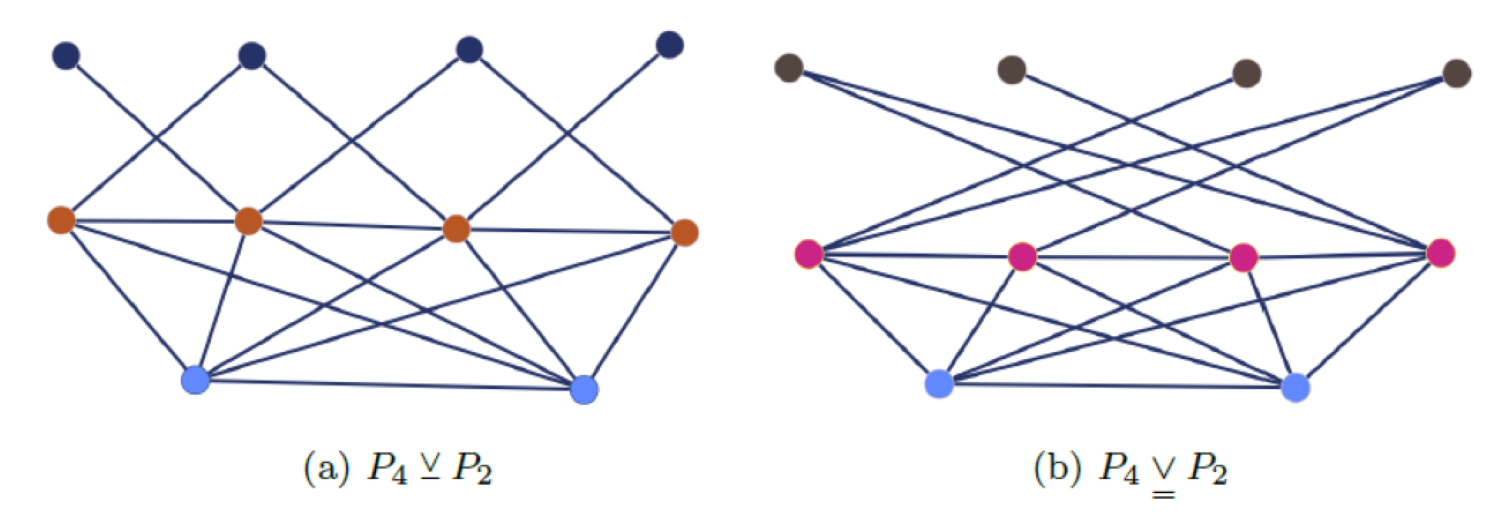}
\caption{\label{fig:BermanH23-Fig. 2.2} \centering{The neighbors splitting (NS)
and nonneighbors splitting (NNS) joins of the path graphs $\PathG{4}$ and $\PathG{2}$
are depicted, respectively, on the left and right-hand sides of this figure.
The NS and NNS joins of graphs are, respectively, denoted by $\PathG{4} \NS \PathG{2}$
and $\PathG{4} \NNS \PathG{2}$ \cite{HamudB24}.}}
\end{figure}
\end{example}

\begin{proposition}[Adjacency matrices of the NS and NNS joins of graphs \cite{Hamud23,HamudB24,LuMZ23}]
\label{proposition: Adjacency Matrices of NS and NNS Join Graphs}
Let $\Gr{G}_1$ and $\Gr{G}_2$ be, respectively, finite and simple graphs on $n_1$ and $n_2$ vertices.
Then, the adjacency matrices of the NS and NNS joins of graphs $\Gr{G}_1 \NS \Gr{G}_2$ and $\Gr{G}_1 \NNS \Gr{G}_2$
are, respectively, given by
\begin{align}
\label{eq: adjacency matrix - NS join}
\A(\Gr{G}_1 \NS \Gr{G}_2) =
\begin{pmatrix}
& \A(\Gr{G}_1) & \A(\Gr{G}_1) & {\mathbf{J}}_{n_1 \times n_2} \\[0.1cm]
& \A(\Gr{G}_1) & {\mathbf{0}}_{n_1 \times n_1} & {\mathbf{0}}_{n_1 \times n_2} \\[0.1cm]
& {\mathbf{J}}_{n_2 \times n_1} & {\mathbf{0}}_{n_2 \times n_1} & \A(\Gr{G}_2)
\end{pmatrix} , \\[0.2cm]
\label{eq: adjacency matrix - NNS join}
\A(\Gr{G}_1 \NNS \Gr{G}_2) =
\begin{pmatrix}
& \A(\Gr{G}_1) & \A(\CGr{G}_1) & {\mathbf{J}}_{n_1 \times n_2} \\[0.1cm]
& \A(\CGr{G}_1) & {\mathbf{0}}_{n_1 \times n_1} & {\mathbf{0}}_{n_1 \times n_2} \\[0.1cm]
& {\mathbf{J}}_{n_2 \times n_1} & {\mathbf{0}}_{n_2 \times n_1} & \A(\Gr{G}_2)
\end{pmatrix} ,
\end{align}
where ${\mathbf{J}}_{n \times m}$ and ${\mathbf{0}}_{n \times m}$ denote, respectively, the
$n \times m$ all-ones and all-zeros matrices. Denoting (for simplicity) $\J{n}$ and
${\mathbf{0}}_n$ to be, respectively, the $n \times n$ all-ones and all-zeros matrices, the equality
\begin{align}
\label{eq: adjaceny matrix of a graph and its complement}
\A(\CGr{G}) = \J{n} - \I{n} - \A(\Gr{G})
\end{align}
relates the adjacency matrices of a graph $\Gr{G}$ and its complement $\CGr{G}$.
\end{proposition}
\begin{proof}
Equality~\eqref{eq: adjaceny matrix of a graph and its complement} holds by
Definition~\ref{definition:complement and s.c. graphs} of a graph complement.
The adjacency matrices in \eqref{eq: adjacency matrix - NS join} and \eqref{eq: adjacency matrix - NNS join}
hold by Definitions~\ref{definition: neighbors splitting join of graphs} and~\ref{definition: nonneighbors splitting join of graphs},
respectively. In these matrices:
\begin{enumerate}[(a)]
\item The first $n_1$ rows and columns correspond to the vertices $v_1, \ldots, v_{n_1}$ in $\Gr{G}_1$.
\item The next $n_1$ rows and columns refer to the copied vertices
$v'_1, \ldots, v'_{n_1}$ that are added to the vertex set of $\Gr{G}_1 \vee \Gr{G}_2$ to form both vertex
sets of the graphs $\Gr{G}_1 \NS \Gr{G}_2$ and $\Gr{G}_1 \NNS \Gr{G}_2$.
\item The last $n_2$ rows and columns refer to the vertices in $\Gr{G}_2$.
\end{enumerate}
\end{proof}

Consider the NS and NNS joins of two graphs with disjoint vertex sets. The next result gives sufficient
conditions for these joins of graphs to be nonisomorphic and cospectral with respect to their adjacency,
Laplacian, signless Laplacian, and normalized Laplacian matrices. A partial result was first published in
Theorem~5.1 of \cite{LuMZ23}, followed by the next result presented in \cite{Hamud23,HamudB24}.
\begin{theorem}[Irregular $\{A, L, Q, \mathcal{L}\}$-NICS graphs]
\label{theorem: Berman and Hamud, 2023}
Let $\Gr{G}_1$ and $\Gr{H}_1$ be regular and cospectral graphs, and let $\Gr{G}_2$ and $\Gr{H}_2$
be regular, nonisomorphic, and cospectral (NICS) graphs. Then, the following statements hold:
\begin{enumerate}
\item  \label{Item 1: Berman and Hamud, 2023}
$\Gr{G}_1 \NS \Gr{G}_2$ and $\Gr{H}_1 \NS \Gr{H}_2$
are irregular $\{A, L, Q, \mathcal{L}\}$-NICS graphs.
\item  \label{Item 2: Berman and Hamud, 2023}
$\Gr{G}_1 \NNS \Gr{G}_2$ and $\Gr{H}_1 \NNS \Gr{H}_2$
are irregular $\{A, L, Q, \mathcal{L}\}$-NICS graphs.
\end{enumerate}
\end{theorem}

\subsection{New Results on Cospectral and Nonisomorphic Graphs, and on Joins of Graphs}
\label{subsection: main results - cospectral and nonisomorphic graphs}

In light of Theorems~\ref{thm: number of walks of a given length}--\ref{theorem: number of spanning trees},
it is established that every pair of $\{A, L, Q, \mathcal{L}\}$-NICS graphs share identical structural properties,
encompassing identical counts of vertices, edges, triangles, components, bipartite components, spanning trees,
and walks (or closed walks) of specific lengths. Moreover, these graphs exhibit either regularity or
irregularity, with matching girth, indicating shared characteristics despite being cospectral and nonisomorphic.
By Theorem~\ref{theorem: A-eigenvalues of a line graph}, the line graphs of these graphs are also $A$-cospectral,
thereby sharing identical counts of vertices, edges, triangles, and walks (or closed walks) of specified lengths.
In view of these findings, we introduce two questions as follows.
\begin{enumerate}
\item
Is the Lov\'{a}sz $\vartheta$-function identical for any pair of $\{A, L, Q, \mathcal{L}\}$-NICS graphs ?
\item
If not, are there subfamilies of graphs for which the Lov\'{a}sz $\vartheta$-function is unique
for any pair of $\{A, L, Q, \mathcal{L}\}$-NICS graphs within that particular subfamily ?
\end{enumerate}
These questions are addressed here with the following conclusions:
\begin{itemize}
\item
By Theorem~\ref{theorem 2: van Dam and Haemers, 2003}, regular graphs with fewer than
10~vertices can be uniquely determined by their spectrum. Therefore, the response to
the first question is affirmative for these small regular graphs.
\item
In general, the response to the first query is negative, as demonstrated by
Example~\ref{example: Regular NICS graphs cont.}, followed by
Theorems~\ref{theorem: existence of NICS graphs} and~\ref{theorem: on Graph Invariants of NS/NNS Joins of Graphs}.
\item
Some subfamilies of regular graphs exhibit a positive response to the second query, as shown
in Corollary~\ref{corollary: sufficient conditions for equalities},
Theorem~\ref{theorem: Shannon Capacity of NS and NNS Join Graphs}, and
Corollary~\ref{corollary: the Lovasz theta-function of regular NICS graphs}.
\end{itemize}
Our analysis relies in part on \cite{vanDamH03, Hamud23, HamudB24, Sason23},
with the background provided in Section~\ref{subsection: preliminaries - cospectral and nonisomorphic graphs}.

\begin{example}[Regular NICS graphs]
\label{example: Regular NICS graphs cont.}
The present example continues Example~\ref{example: regular NICS graphs with 10 vertices}. Let $\Gr{G}$
and $\Gr{H}$ be the graphs on the left and right plots of Figure~\ref{fig:vanDamH03}, respectively.
The Lov\'{a}sz $\vartheta$-functions of these graphs are computed numerically by solving the SDP problem
in \eqref{eq: SDP problem - Lovasz theta-function} for both $\Gr{G}$ and $\Gr{H}$. The resulting
values are given with a precision of 5~decimal points as follows:
\begin{align}
\label{eq3:30.09.23}
\vartheta(\Gr{G}) = 3.23607, \quad \vartheta(\Gr{H}) = 3.26880.
\end{align}
The complements $\CGr{G}$ and $\CGr{H}$ are 5-regular NICS graphs, whose Lov\'{a}sz
$\vartheta$-functions are equal to
\begin{align}
\label{eq4:30.09.23}
\vartheta(\CGr{G}) = 3.19656, \quad \vartheta(\CGr{H}) = 3.13198.
\end{align}
This results in two pairs of regular NICS graphs on 10 vertices, denoted by $\{\Gr{G}, \Gr{H}\}$ and $\{\CGr{G}, \CGr{H}\}$.
Each pair exhibits distinct values of the Lov\'{a}sz $\vartheta$-functions, as demonstrated by equations
\eqref{eq3:30.09.23} and \eqref{eq4:30.09.23}. It is noteworthy that the four graphs $\Gr{G}$, $\CGr{G}$, $\Gr{H}$,
and $\CGr{H}$ share identical independence numbers~(3), clique numbers~(3), and chromatic numbers~(4).
\end{example}

Following Example~\ref{example: Regular NICS graphs cont.}, the next result demonstrates the existence of
pairs of irregular $\{A, L, Q, \mathcal{L}\}$-NICS graphs on an arbitrarily large order, such that
these cospectral and nonisomorphic graphs have identical pairs of independence numbers, clique numbers, and
chromatic numbers, while exhibiting distinct values of their Lov\'{a}sz $\vartheta$-functions.
\begin{theorem}[On irregular NICS graphs]
\label{theorem: existence of NICS graphs}
For every even integer $n \geq 14$, there exist two connected, irregular $\{A, L, Q, \mathcal{L}\}$-NICS graphs
on $n$ vertices with identical independence, clique, and chromatic numbers, yet their Lov\'{a}sz $\vartheta$-functions
are distinct.
\end{theorem}

\begin{proof}
See Section~\ref{subsubsection: proof of a theorem on the existence of NICS graphs}.
\end{proof}

\begin{theorem}[Graph invariants of NS/NNS joins of graphs]
\label{theorem: on Graph Invariants of NS/NNS Joins of Graphs}
For every finite, undirected, and simple graphs $\Gr{G}_1$ and $\Gr{G}_2$, the following holds:
\begin{align}
\label{eq3:02.10.23}
\indnum{\Gr{G}_1 \NS \Gr{G}_2} &\geq \card{\V{\Gr{G}_1}} + \indnum{\Gr{G}_2}, \\
\label{eq3b:02.10.23}
\indnum{\Gr{G}_1 \NNS \Gr{G}_2} &\geq \card{\V{\Gr{G}_1}} + \indnum{\Gr{G}_2}, \\
\label{eq1:02.10.23}
\clnum{\Gr{G}_1 \NS \Gr{G}_2} &= \clnum{\Gr{G}_1} + \clnum{\Gr{G}_2} \\
\label{eq2:02.10.23}
&= \clnum{\Gr{G}_1 \NNS \Gr{G}_2}, \\
\label{eq8:02.10.23}
\chrnum{\Gr{G}_1 \NS \Gr{G}_2} &= \chrnum{\Gr{G}_1} + \chrnum{\Gr{G}_2} \\
\label{eq9:02.10.23}
&= \chrnum{\Gr{G}_1 \NNS \Gr{G}_2}, \\[0.1cm]
\label{eq1:30.09.23}
\vartheta(\Gr{G}_1 \NS \Gr{G}_2) &\geq \card{\V{\Gr{G}_1}} + \vartheta(\Gr{G}_2), \\[0.15cm]
\label{eq2:30.09.23}
\vartheta(\Gr{G}_1 \NNS \Gr{G}_2) &\geq \card{\V{\Gr{G}_1}} + \vartheta(\Gr{G}_2).
\end{align}
Moreover, the following sufficient conditions for equalities hold:
\begin{itemize}
\item
Inequality \eqref{eq1:30.09.23} holds with equality if there exists a permutation
$\pi$ of the vertex set $\V{\Gr{G}_1}$ such that $\{v_i, v_{\pi(i)}\} \in \E{\Gr{G}_1}$
for all $i \in \{1, \ldots, \card{\V{\Gr{G}_1}}\}$.
\item
Inequality \eqref{eq2:30.09.23} holds with equality if there exists a permutation
$\pi$ of the vertex set $\V{\Gr{G}_1}$ such that $\pi(i) \neq i$ and
$\{v_i, v_{\pi(i)}\} \not\in \E{\Gr{G}_1}$ for all $i \in \{1, \ldots, \card{\V{\Gr{G}_1}}\}$.
\item
Inequality \eqref{eq3:02.10.23} holds with equality if the sufficient condition
for equality in \eqref{eq1:30.09.23} is satisfied and $\indnum{\Gr{G}_2} = \vartheta(\Gr{G}_2)$.
\item
Inequality \eqref{eq3b:02.10.23} holds with equality if the sufficient condition
for equality in \eqref{eq2:30.09.23} is satisfied and $\indnum{\Gr{G}_2} = \vartheta(\Gr{G}_2)$.
\end{itemize}
\end{theorem}
\begin{proof}
See Section~\ref{subsubsection: proof of a theorem on graph invariants of NS/NNS joins of graphs}.
\end{proof}

\begin{remark}[Perfect matchings in graphs]
\label{remark: perfect matchings}
An equivalent formulation of the sufficient conditions in Theorem~\ref{theorem: on Graph Invariants of NS/NNS Joins of Graphs},
ensuring \eqref{eq1:30.09.23} and \eqref{eq2:30.09.23} to hold with equalities relies on perfect matchings in graphs.
Let $\Gr{G}$ be a graph on $n$ vertices, whose adjacency matrix is $\A \eqdef \A(\Gr{G})$.
Associate with the graph $\Gr{G}$ the bipartite graph $\mathrm{B}(\Gr{G}) = (\set{U} \cup \set{V}, \, \set{E})$
on $2n$ vertices. The vertex set $\set{U} \cup \set{V}$ consists of two disjoint sets $\set{U} = {u_1, \ldots, u_n}$
and $\set{V} = {v_1, \ldots, v_n}$, and an edge $\{u_i, v_j\} \in \set{E}$ if and only if $A_{i,j} = 1$
with $i, j \in \OneTo{n}$.

An equivalent formulation of the sufficient condition for \eqref{eq1:30.09.23} to hold with equality
is the existence of a perfect matching in the bipartite graph $\mathrm{B}(\Gr{G}_1)$.
Indeed, the number of different perfect matchings in
$\mathrm{B}(\Gr{G}_1)$ is equal to the permanent of the adjacency matrix $\A(\Gr{G}_1)$ (see, e.g.,
Chapter~37 in \cite{AignerZ18}), so that sufficient condition is given by $\mathrm{per}(\A(\Gr{G}_1)) > 0$.
Similarly, an equivalent formulation of the sufficient condition for \eqref{eq2:30.09.23} to hold
with equality is that the bipartite graph $\mathrm{B}(\CGr{G}_1)$ has a perfect matching, which can
be also represented by the condition $\mathrm{per}(\A(\CGr{G}_1)) > 0$.
\end{remark}

A well-established result is next presented, which hinges on the widely recognized van der Waerden conjecture.
Initially posed as a question, this conjecture has been elegantly proven and elucidated in Chapter~24 of
\cite{AignerZ18}, \cite{LauerntS10}, and Chapters~11--12 of \cite{LintW01}.
\begin{lemma}[Permanent of the adjacency matrix of a regular graph]
\label{lemma: positive permanent}
Let $\Gr{G}$ be a $d$-regular graph on $n$ vertices, and let $\A$ be its adjacency matrix. Then,
\begin{align}
\label{eq:van der Waerden}
\mathrm{per}(\A) &\geq n! \, \biggl(\frac{d}{n}\biggr)^n.
\end{align}
\end{lemma}
By Theorem~\ref{theorem: on Graph Invariants of NS/NNS Joins of Graphs} and Remark~\ref{remark: perfect matchings},
the positivity of the permanent of the adjacency matrix of the graph $\Gr{G}_1$ suffices for \eqref{eq1:30.09.23}
to hold with equality. Likewise, the positivity of the permanent of the adjacency matrix of the complement graph
$\CGr{G}_1$ suffices for \eqref{eq2:30.09.23} to hold with equality. Furthermore, by
Theorem~\ref{theorem: on Graph Invariants of NS/NNS Joins of Graphs}, the satisfiability of the supplementary condition
$\indnum{\Gr{G}_2} = \vartheta(\Gr{G}_2)$ implies that \eqref{eq3:02.10.23} and \eqref{eq3b:02.10.23}
hold with equalities, respectively, in addition to the equalities in \eqref{eq1:30.09.23} and \eqref{eq2:30.09.23}.
If $\Gr{G}_1$ is a noncomplete and nonempty regular graph, then Lemma~\ref{lemma: positive permanent}
yields the positivity of the permanent of each of the adjacency matrices $\A(\Gr{G}_1)$ and $\A(\CGr{G}_1)$ (since
if $\Gr{G}_1$ is a $d$-regular graph, then its complement $\CGr{G}_1$ is an $(n-d-1)$-regular graph). The next result follows.

\begin{corollary}
\label{corollary: sufficient conditions for equalities}
Let $\Gr{G}_1$ and $\Gr{G}_2$ be finite, noncomplete, nonempty, undirected, and simple graphs.
\begin{itemize}
\item If $\Gr{G}_1$ is a regular graph, then \eqref{eq1:30.09.23} and \eqref{eq2:30.09.23}
hold with equalities.
\item If $\Gr{G}_1$ is a regular graph and $\indnum{\Gr{G}_2} = \vartheta(\Gr{G}_2)$, then
\eqref{eq3:02.10.23} and \eqref{eq3b:02.10.23} hold with equalities.
\end{itemize}
\end{corollary}

In light of Theorem~\ref{theorem: on Graph Invariants of NS/NNS Joins of Graphs},
Remark~\ref{remark: perfect matchings}, and Corollary~\ref{corollary: sufficient conditions for equalities},
the next result follows.
\begin{theorem}[The independence number, Shannon capacity, and Lov\'{a}sz $\vartheta$-function of NS and NNS joins of graphs]
\label{theorem: Shannon Capacity of NS and NNS Join Graphs}
Let $\Gr{G}_1$ and $\Gr{G}_2$ be finite, undirected, and simple graphs.
\begin{enumerate}
\item \label{Item 1 - Shannon Capacity of NS and NNS Join Graphs}
If $\mathrm{per}(\A(\Gr{G}_1)) > 0$ and $\indnum{\Gr{G}_2} = \vartheta(\Gr{G}_2)$, then
\begin{align}
\label{eq14:02.10.23}
& \indnum{\Gr{G}_1 \NS \Gr{G}_2} = \Theta(\Gr{G}_1 \NS \Gr{G}_2)
= \vartheta(\Gr{G}_1 \NS \Gr{G}_2) = \card{\V{\Gr{G}_1}} + \vartheta(\Gr{G}_2).
\end{align}
The equalities in \eqref{eq14:02.10.23} hold, in particular, if $\Gr{G}_1$ is
a nonempty regular graph and $\indnum{\Gr{G}_2} = \vartheta(\Gr{G}_2)$.
\item \label{Item 2 - Shannon Capacity of NS and NNS Join Graphs}
If $\mathrm{per}(\A(\CGr{G}_1)) > 0$ and the equality
$\indnum{\Gr{G}_2} = \vartheta(\Gr{G}_2)$ holds, then
\begin{align}
\label{eq15:02.10.23}
& \indnum{\Gr{G}_1 \NNS \Gr{G}_2} = \Theta(\Gr{G}_1 \NNS \Gr{G}_2)
= \vartheta(\Gr{G}_1 \NNS \Gr{G}_2) = \card{\V{\Gr{G}_1}} + \vartheta(\Gr{G}_2).
\end{align}
Equalities in \eqref{eq15:02.10.23} hold, in particular, if $\Gr{G}_1$
is a noncomplete regular graph and $\indnum{\Gr{G}_2} = \vartheta(\Gr{G}_2)$.
\item \label{Item 3 - Shannon Capacity of NS and NNS Join Graphs}
If $\, \mathrm{per}(\A(\Gr{G}_1)) > 0$, $\, \mathrm{per}(\A(\CGr{G}_1)) > 0$,
and $\indnum{\Gr{G}_2} = \vartheta(\Gr{G}_2)$, then
the graphs $\Gr{G}_1 \NS \Gr{G}_2$ and $\Gr{G}_1 \NNS \Gr{G}_2$
share identical independence numbers, clique numbers, chromatic numbers,
Shannon capacities, and Lov\'{a}sz $\vartheta$-functions. This particulary holds
if $\Gr{G}_1$ is a nonempty and noncomplete regular graph and
$\indnum{\Gr{G}_2} = \vartheta(\Gr{G}_2)$.
\end{enumerate}
\end{theorem}
\begin{proof}
See Section~\ref{subsubsection: proof of a theorem on the Shannon Capacity of NS and NNS Join Graphs}.
\end{proof}

Discussions on the sufficient conditions in Theorem~\ref{theorem: on Graph Invariants of NS/NNS Joins of Graphs}
are further given in Remarks~\ref{remark: the independence numbers of NS/NNS Joins of graphs}
and~\ref{remark: the Lovasz theta-function of NS/NNS joins of graphs} (see
Section~\ref{subsection: examples - cospectral and nonisomorphic graphs}).
By the above results, a pair of regular or irregular $\{A, L, Q, \mathcal{L}\}$-NICS graphs may
have distinct Lov\'{a}sz $\vartheta$-functions. A restatement of Theorem~\ref{thm:bounds on the Lovasz function for regular graphs}
shows that, within some structured infinite subfamilies of regular graphs, the Lov\'{a}sz $\vartheta$-functions
of NICS graphs are identical.
\begin{corollary}[On NICS regular graphs]
\label{corollary: the Lovasz theta-function of regular NICS graphs}
Let $\Gr{G}$ and $\Gr{H}$ be noncomplete, $d$-regular NICS graphs, and let $\{\lambda_i\}_{i=1}^n$
be their common $A$-spectrum in nonincreasing order.
\begin{enumerate}
\item \label{corollary: item 1 - On NICS regular graphs}
If each of these graphs is either strongly regular or edge-transitive, then
\begin{align}
\label{eq1a:26.09.23}
\vartheta(\Gr{G}) = \vartheta(\Gr{H}) = -\frac{n \lambda_n}{d - \lambda_n}.
\end{align}
\item \label{corollary: item 2 - On NICS regular graphs}
If each of these graphs is either strongly regular or its complement is vertex-transitive and
edge-transitive, then
\begin{align}
\label{eq1b:26.09.23}
\vartheta(\Gr{G}) = \vartheta(\Gr{H}) = \frac{n-d+\lambda_2}{1+\lambda_2},
\end{align}
and, for the $(n-d-1)$-regular cospectral graphs $\CGr{G}$ and $\CGr{H}$,
\begin{align}
\label{eq1c:26.09.23}
\vartheta(\CGr{G}) = \vartheta(\CGr{H}) = 1-\frac{d}{\lambda_n}.
\end{align}
\end{enumerate}
\end{corollary}
\begin{proof}
See Section~\ref{subsubsection: proof of corollary: the Lovasz theta-function of regular NICS graphs}.
\end{proof}

\subsection{Proofs}
\label{subsection: proofs - cospectral and nonisomorphic graphs}
This section proves all the new results that are presented in
Section~\ref{subsection: main results - cospectral and nonisomorphic graphs}.

\subsubsection{Proof of Theorem~\ref{theorem: existence of NICS graphs}}
\label{subsubsection: proof of a theorem on the existence of NICS graphs}

In light of Theorem~\ref{theorem: Berman and Hamud, 2023}, for every $k \in \naturals$,
this proof suggests a construction of two connected, irregular $\{A, L, Q, \mathcal{L}\}$-NICS
graphs on $n_k \eqdef 2k + 12$ vertices. To that end, let
\begin{align}
\label{eq0:15.10.23}
\Gr{G}_1^{(k)} = \Gr{H}_1^{(k)} \eqdef
\begin{dcases}
\begin{array}{cl}
\CG{k+1}, \quad & \mbox{if $k \geq 2$,} \\[0.1cm]
\CoG{2},  \quad & \mbox{if $k=1$},
\end{array}
\end{dcases}
\end{align}
and let $\Gr{G}_2$ and $\Gr{H}_2$ be the 4-regular NICS graphs on 10~vertices, respectively
depicted on the left and right-hand sides of Figure~\ref{fig:vanDamH03}.
By Item~\ref{Item 1: Berman and Hamud, 2023} of Theorem~\ref{theorem: Berman and Hamud, 2023},
for all $k \in \naturals$, the graphs $\Gr{G}_1^{(k)} \NS \Gr{G}_2$ and
$\Gr{H}_1^{(k)} \NS \Gr{H}_2$ are irregular $\{A, L, Q, \mathcal{L}\}$-NICS
graphs. The orders of $\Gr{G}_1^{(k)}$, $\Gr{G}_2$, $\Gr{H}_1^{(k)}$, and
$\Gr{H}_2$ are given by
\begin{align}
\label{eq1:15.10.23}
\card{\V{\Gr{G}_1^{(k)}}} = \card{\V{\Gr{H}_1^{(k)}}} = k+1, \quad
\card{\V{\Gr{G}_2}} = \card{\V{\Gr{H}_2}} = 10.
\end{align}
By Definition~\ref{definition: neighbors splitting join of graphs} and \eqref{eq1:15.10.23},
it follows that
\begin{align}
\label{eq3:15.10.23}
\card{\V{\Gr{G}_1^{(k)} \NS \Gr{G}_2}} &=
2 \, \card{\V{\Gr{G}_1^{(k)}}} + \card{\V{\Gr{G}_2}} \\
\label{eq4:15.10.23}
&= 2k+12 \\
\label{eq5:15.10.23}
&= n_k,
\end{align}
and similarly, by Definition~\ref{definition: nonneighbors splitting join of graphs} and \eqref{eq1:15.10.23},
\begin{align}
\label{eq6:15.10.23}
\card{\V{\Gr{H}_1^{(k)} \NS \Gr{H}_2}} &=
2 \, \card{\V{\Gr{H}_1^{(k)}}} + \card{\V{\Gr{H}_2}} \\
\label{eq7:15.10.23}
&= n_k.
\end{align}
The graphs $\Gr{G}_2$ and $\Gr{H}_2$ can be verified to have identical independence numbers,
clique numbers, and chromatic numbers, which are equal to
\begin{align}
\label{eq1:29.09.23}
\indnum{\Gr{G}_2} = \indnum{\Gr{H}_2} = 3, \quad \clnum{\Gr{G}_2} = \clnum{\Gr{H}_2} = 3,
\quad \chrnum{\Gr{G}_2} = \chrnum{\Gr{H}_2} = 4.
\end{align}

A largest independent set in the graph $\Gr{G}_1^{(k)} \NS \Gr{G}_2$
is obtained by combining a largest
independent set in $\Gr{G}_2$ along with the vertices
$v'_1, \ldots, v'_{k+1}$ that are added to the vertex set
of $\Gr{G}_1^{(k)} \vee \Gr{G}_2$ to form both vertex sets
of $\Gr{G}_1^{(k)} \NS \Gr{G}_2$ and $\Gr{G}_1^{(k)} \NNS \Gr{G}_2$.
Indeed, this can be deduced as follows:
\begin{itemize}
\item
The vertices $v'_1, \ldots, v'_{k+1}$ are nonadjacent to each other, and
they are also nonadjacent to vertices in $\Gr{G}_2$ (by
Definition~\ref{definition: neighbors splitting join of graphs}),
so the disjoint union of a largest independent set in $\Gr{G}_2$ and $\{v'_1, \ldots, v'_{k+1}\}$
forms an independent set in $\Gr{G}_1^{(k)} \NS \Gr{G}_2$ whose size
is equal to $\card{\V{\Gr{G}_1}} + \indnum{\Gr{G}_2} = (k+1) + 3 = k+4$.
\item
Each of the vertices $v_1, \ldots, v_{k+1}$ in $\Gr{G}_1^{(k)}$ is adjacent to all
the vertices in $\Gr{G}_2$, so none of them can be added to get a larger independent
set in $\Gr{G}_1^{(k)} \NS \Gr{G}_2$.
\item
When forming an independent set in $\Gr{G}_1^{(k)} \NS \Gr{G}_2$, by taking a disjoint union
of a nonempty subset of vertices in $\Gr{G}_1^{(k)}$ with a subset of the vertices
$\{v'_1, \ldots, v'_{k+1}\}$, it is important to note that no vertex from $\Gr{G}_2$ can be
included in that independent set (since all vertices in $\Gr{G}_2$ are adjacent to each vertex
in $\Gr{G}_1^{(k)}$), and also the total cardinality of the two former subsets of vertices
cannot exceed $\card{\V{\Gr{G}_1^{(k)}}} = k+1$ (by Definition~\ref{definition: neighbors splitting join of graphs}
and the definition of the graph $\Gr{G}_1^{(k)}$ in \eqref{eq0:15.10.23}). This therefore imposes
a constraint on the size of that independent set, which is strictly less than $k+4$.
\end{itemize}
The same reasoning holds for the independence number of the NS join of graphs
$\Gr{H}_1^{(k)} \NS \Gr{H}_2$, so
\begin{align}
\label{eq2:29.09.23}
\indnum{\Gr{G}_1^{(k)} \NS \Gr{G}_2} = \indnum{\Gr{H}_1^{(k)} \NS \Gr{H}_2} = k+4,
\quad \forall \, k \in \naturals.
\end{align}

A largest clique in $\Gr{G}_1^{(k)} \NS \Gr{G}_2$ is obtained by taking the disjoint union of
largest cliques in $\Gr{G}_1^{(k)}$ and $\Gr{G}_2$. Indeed, this holds since each
vertex in $\Gr{G}_1^{(k)}$ is adjacent to each vertex in $\Gr{G}_2$, and the vertices
$v'_1, \ldots, v'_{k+1}$ (as defined above) are nonadjacent to each other.
Likewise, the same reasoning also holds for the clique number of $\Gr{H}_1^{(k)} \NS \Gr{H}_2$.
By \eqref{eq1:29.09.23} where $\clnum{\Gr{G}_2} = 3 = \clnum{\Gr{H}_2}$, and since
$\clnum{\Gr{G}_1^{(k)}} = 2 = \clnum{\Gr{H}_1^{(k)}}$, it follows that
\begin{align}
\label{eq3:29.09.23}
\clnum{\Gr{G}_1^{(k)} \NS \Gr{G}_2} = \clnum{\Gr{H}_1^{(k)} \NS \Gr{H}_2} = 5,
\quad \forall \, k \in \naturals.
\end{align}

The chromatic number of the NS join of graphs $\Gr{G}_1^{(k)} \NS \Gr{G}_2$ is next considered.
By \eqref{eq1:29.09.23}, $\chrnum{\Gr{G}_2} = 4$ colors are at least needed for properly coloring the vertices
in $\Gr{G}_2$ so that no two adjacent vertices in $\Gr{G}_2$ are assigned an identical color. All of the
vertices in the set $\V{\Gr{G}_1^{(k)}} = \{v_1, \ldots, v_k\}$ are adjacent to each vertex in $\Gr{G}_2$, so
$\chrnum{\Gr{G}_1^{(k)}}$ new colors are used for the vertices in $\Gr{G}_1^{(k)}$.
The chromatic number of $\Gr{G}_1^{(k)} = \CG{k+1}$ is equal to~2 or~3 if $k \geq 2$ is odd or even, respectively,
and the chromatic number of $\Gr{G}_1^{(1)} = \CoG{2}$ is equal to~2. Consequently, by \eqref{eq0:15.10.23}, for all $k \in \naturals$
\begin{align}
\label{eq0:16.10.23}
\chrnum{\Gr{G}_1^{(k)}} = \chrnum{\Gr{H}_1^{(k)}} =
\begin{dcases}
\begin{array}{cl}
2, \quad & \mbox{if $k$ is odd,} \\[0.1cm]
3, \quad & \mbox{if $k$ is even}.
\end{array}
\end{dcases}
\end{align}
Finally, since vertices in $\{v'_1, \ldots, v'_{k+1}\}$ are nonadjacent to one another and also they are nonadjacent to
any vertex in $\Gr{G}_2$,
one of the colors assigned to the vertices in $\Gr{G}_2$ can be also used for all the vertices $v'_1, \ldots, v'_{k+1}$.
Likewise, the same argument holds for $\Gr{H}_1^{(k)} \NS \Gr{H}_2$, so it follows from \eqref{eq1:29.09.23}
and \eqref{eq0:16.10.23} that
\begin{align}
\label{eq4:29.09.23}
\chrnum{\Gr{G}_1^{(k)} \NS \Gr{G}_2} &= \chrnum{\Gr{H}_1^{(k)} \NS \Gr{H}_2} \\
\label{eq1:16.10.23}
&= \chrnum{\Gr{G}_1^{(k)}} + \chrnum{\Gr{G}_2} \\
\label{eq2:16.10.23}
&=
\begin{dcases}
\begin{array}{cl}
6, \quad & \mbox{if $k$ is odd,} \\[0.1cm]
7, \quad & \mbox{if $k$ is even.}
\end{array}
\end{dcases}
\end{align}

Next, we demonstrate that the Lov\'{a}sz $\vartheta$-functions of the graphs $\Gr{G}_1^{(k)} \NS \Gr{G}_2$
and $\Gr{H}_1^{(k)} \NS \Gr{H}_2$ are distinct.
We compute the Lov\'{a}sz $\vartheta$-functions of $\Gr{G}_2$ and $\Gr{H}_2$
by solving the SDP problem in \eqref{eq: SDP problem - Lovasz theta-function}, which yields
\begin{align}
\label{eq5a:29.09.23}
& \vartheta(\Gr{G}_2) = 3.23607 \ldots, \\
\label{eq5b:29.09.23}
& \vartheta(\Gr{H}_2) = 3.26880 \ldots \; .
\end{align}
Consider a disjoint union of $k+1$ isolated vertices (i.e., the $(k+1)$-partite graph
$\mathrm{K}_{1, 1, \ldots, 1}$) with either $\Gr{G}_2$ or $\Gr{H}_2$, and denote these
disjoint unions by the graphs $\Gr{U}^{(k)}$ and $\Gr{W}^{(k)}$, respectively. In general, the
Lov\'{a}sz $\vartheta$-function of a disjoint union of graphs is equal to the sum of the
Lov\'{a}sz $\vartheta$-functions of the component graphs (see \eqref{eq: Lovasz function of a disjoint union of graphs}
with a proof in Section~18 of \cite{Knuth94}), so by \eqref{eq1:15.10.23}
\begin{align}
\label{eq6:29.09.23}
\vartheta(\Gr{U}^{(k)}) &= k+1 + \vartheta(\Gr{G}_2) = \card{\V{\Gr{G}_1^{(k)}}} + \vartheta(\Gr{G}_2), \\
\label{eq7:29.09.23}
\vartheta(\Gr{W}^{(k)}) &= k+1 + \vartheta(\Gr{H}_2) = \card{\V{\Gr{H}_1^{(k)}}} + \vartheta(\Gr{H}_2).
\end{align}
The graphs $\Gr{U}^{(k)}$ and $\Gr{W}^{(k)}$ are, respectively, induced subgraphs of $\Gr{G}_1^{(k)} \NS \Gr{G}_2$
and $\Gr{H}_1^{(k)} \NS \Gr{H}_2$. Indeed, $\Gr{U}^{(k)}$ and $\Gr{W}^{(k)}$ are obtained by deleting the vertices
in $\Gr{G}_1^{(k)}$ and $\Gr{H}_1^{(k)}$ from the graphs $\Gr{G}_1^{(k)} \NS \Gr{G}_2$ and $\Gr{H}_1^{(k)} \NS \Gr{H}_2$,
respectively (recall that $\Gr{G}_1^{(k)} = \Gr{H}_1^{(k)}$ by \eqref{eq0:15.10.23}).
The Lov\'{a}sz $\vartheta$-function of an induced subgraph is less than or equal to the Lov\'{a}sz $\vartheta$-function of
the original graph (see Item~\ref{item: subgraphs} in Section~\ref{subsection: Lovasz theta-function}),  so it follows
from \eqref{eq6:29.09.23} and \eqref{eq7:29.09.23} that
\begin{align}
\label{eq8:29.09.23}
& \vartheta(\Gr{G}_1^{(k)} \NS \Gr{G}_2) \geq \vartheta(\Gr{U}^{(k)}) = \vartheta(\Gr{G}_2) + k+1, \\
\label{eq9:29.09.23}
& \vartheta(\Gr{H}_1^{(k)} \NS \Gr{H}_2) \geq \vartheta(\Gr{W}^{(k)}) = \vartheta(\Gr{H}_2) + k+1.
\end{align}
We finally show that the leftmost inequalities in \eqref{eq8:29.09.23} and \eqref{eq9:29.09.23} hold with equality.
This is shown by Definition~\ref{definition: Lovasz theta function} of the Lov\'{a}sz $\vartheta$-function of a graph.
For $k \in \naturals$, let ${\bf{u}}^{(k)}$ be an orthonormal representation of the graph $\Gr{U}^{(k)}$,
$\{{\bf{u}}_i^{(k)}: i \in \V{\Gr{U}^{(k)}} \}$, and let the unit vector ${\bf{c}}^{(k)}$ be an optimal handle
of that orthonormal representation in the sense of attaining the minimum on the right-hand side of the
equality for the Lov\'{a}sz $\vartheta$-function of the graph $\Gr{U}^{(k)}$ (see \eqref{eq: Lovasz theta function}), i.e.,
\begin{align}
\label{eq10:29.09.23}
\vartheta(\Gr{U}^{(k)}) = \max_{i \in \V{\Gr{U}^{(k)}}} \,
\frac1{\left( \bigl({\bf{c}}^{(k)}\bigr)^{\mathrm{T}} {\bf{u}}_i^{(k)} \right)^2}.
\end{align}
Let $v'_1, \ldots, v'_{k+1}$ denote the vertices that are added to the vertex set of
the join of graphs $\Gr{G}_1^{(k)} \vee \Gr{G}_2$ in order to form the vertex set of
the NS join of graphs $\Gr{G}_1^{(k)} \NS \Gr{G}_2$.
These added vertices are associated with $k+1$ vectors, all of which are orthogonal.
This orthogonality arises because they correspond to vertices that are nonadjacent
in the graph $\Gr{G}_1^{(k)} \NS \Gr{G}_2$. Furthermore, these vectors are also orthogonal
to all the vectors assigned to the vertices in $\Gr{G}_2$, as none of the vertices
$v'_1, \ldots, v'_{k+1}$ are adjacent to any vertex in $\Gr{G}_2$.
By \eqref{eq0:15.10.23}, the vertices $v_i$ and $v_{i+1}$ are adjacent in $\Gr{G}_1^{(k)}$
for all $i \in \OneTo{k}$, and vertex $v_{k+1}$ is adjacent to $v_1$ in $\Gr{G}_1^{(k)}$.
The $k+1$ orthonormal vectors that are assigned to the vertices $v'_1, \ldots, v'_{k+1}$ can
be also used to represent the $k+1$ vertices $v_1, \ldots, v_{k+1}$ of $\Gr{G}_1^{(k)}$ in
an orthonormal representation of the graph $\Gr{G}_1^{(k)} \NS \Gr{G}_2$.
Indeed, that can be done by letting, for $i \in \OneTo{k}$, the vector assigned to $v'_i$ be
equal to the vector assigned to the vertex $v_{i+1}$, and the vector assigned to the vertex
$v'_{k+1}$ be equal to the vector assigned to vertex $v_1$. This implies that any two unit vectors that
are assigned to nonadjacent vertices from $\{v'_1, \ldots, v'_{k+1}\} \cup \{v_1, \ldots, v_{k+1}\}$
in the graph $\Gr{G}_1^{(k)} \NS \Gr{G}_2$ are orthonormal. Hence, by \eqref{eq: Lovasz theta function}
and \eqref{eq10:29.09.23},
\begin{align}
\label{eq10b:29.09.23}
\vartheta(\Gr{G}_1^{(k)} \NS \Gr{G}_2) \leq \vartheta(\Gr{U}^{(k)}),
\end{align}
and the combination of \eqref{eq8:29.09.23} and \eqref{eq10b:29.09.23} gives
\begin{align}
\label{eq11:29.09.23}
\vartheta(\Gr{G}_1^{(k)} \NS \Gr{G}_2) = \vartheta(\Gr{U}^{(k)}).
\end{align}
Similarly, $\vartheta(\Gr{H}_1^{(k)} \NS \Gr{H}_2) = \vartheta(\Gr{W}^{(k)})$,
so \eqref{eq8:29.09.23} and \eqref{eq9:29.09.23} hold with equality, and by
\eqref{eq5a:29.09.23}--\eqref{eq7:29.09.23}
\begin{align}
\vartheta(\Gr{G}_1^{(k)} \NS \Gr{G}_2) &= k+1 + \vartheta(\Gr{G}_2)
= k + 4.23607 \ldots \; , \label{eq13:29.09.23} \\
\vartheta(\Gr{H}_1^{(k)} \NS \Gr{H}_2) &= k+1 + \vartheta(\Gr{H}_2)
= k + 4.26880 \ldots \; . \label{eq14:29.09.23}
\end{align}
This demonstrates that every pair of the constructed irregular $\{A, L, Q, \mathcal{L}\}$-NICS graphs
exhibit identical independence numbers, clique numbers, and chromatic numbers, while also possessing
distinct Lov\'{a}sz $\vartheta$-functions.

\subsubsection{Proof of Theorem~\ref{theorem: on Graph Invariants of NS/NNS Joins of Graphs}}
\label{subsubsection: proof of a theorem on graph invariants of NS/NNS joins of graphs}

For any finite, undirected, and simple graphs $\Gr{G}_1$ and $\Gr{G}_2$, the arguments
presented in the proof of Theorem~\ref{theorem: existence of NICS graphs} apply directly to
the clique numbers $\clnum{\Gr{G}_1 \NS \Gr{G}_2}$ and $\clnum{\Gr{G}_1 \NNS \Gr{G}_2}$,
and the chromatic numbers $\chrnum{\Gr{G}_1 \NS \Gr{G}_2}$ and $\chrnum{\Gr{G}_1 \NNS \Gr{G}_2}$.
These considerations therefore validate equalities \eqref{eq1:02.10.23}--\eqref{eq9:02.10.23}.

The remainder of the proof of Theorem~\ref{theorem: existence of NICS graphs} does not extend,
however, to equalities in \eqref{eq3:02.10.23}, \eqref{eq3b:02.10.23}, \eqref{eq1:30.09.23}, and
\eqref{eq2:30.09.23}, regarding the independence numbers and Lov\'{a}sz $\vartheta$-functions
of the graphs $\Gr{G}_1 \NS \Gr{G}_2$ and $\Gr{G}_1 \NNS \Gr{G}_2$. Subsequently, we establish
inequalities \eqref{eq3:02.10.23}, \eqref{eq3b:02.10.23}, \eqref{eq1:30.09.23}, and \eqref{eq2:30.09.23},
and provide sufficient conditions to assert their validity with equality.
[Remarks~\ref{remark: the independence numbers of NS/NNS Joins of graphs}
and~\ref{remark: the Lovasz theta-function of NS/NNS joins of graphs} consider some cases
where \eqref{eq3:02.10.23} and \eqref{eq1:30.09.23} hold with strict inequalities].

Let $n_1 \eqdef \card{\V{\Gr{G}_1}}$, and let $v'_1, \ldots, v'_{n_1}$ be the
vertices added to the vertex set of $\Gr{G}_1 \vee \Gr{G}_2$ to form both vertex
sets of the graphs $\Gr{G}_1 \NS \Gr{G}_2$ and $\Gr{G}_1 \NNS \Gr{G}_2$. According
to Definitions~\ref{definition: neighbors splitting join of graphs}
and~\ref{definition: nonneighbors splitting join of graphs}, combining
an independent set in $\Gr{G}_2$ with the vertices $v'_1, \ldots, v'_{n_1}$ yields
an independent set in both $\Gr{G}_1 \NS \Gr{G}_2$ and $\Gr{G}_1 \NNS \Gr{G}_2$. By
selecting a largest independent set in $\Gr{G}_2$, the resulting independent set in
both $\Gr{G}_1 \NS \Gr{G}_2$ and $\Gr{G}_1 \NNS \Gr{G}_2$ has a cardinality equal to
$\card{\V{\Gr{G}_1}} + \indnum{\Gr{G}_2}$. This justifies inequalities \eqref{eq3:02.10.23}
and \eqref{eq3b:02.10.23}.

Inequalities \eqref{eq1:30.09.23} and \eqref{eq2:30.09.23} are next proved.
Consider the graph $\Gr{U}$, defined as the disjoint union of an empty graph
on $n_1$ vertices, denoted by $\EmG{n_1}$, and the graph $\Gr{G}_2$. By
\eqref{eq: Lovasz function of a disjoint union of graphs}, it follows that
\begin{align}
\label{eq4a:19.10.23}
\vartheta(\Gr{U}) &= \vartheta(\EmG{n_1}) + \vartheta(\Gr{G}_2) \\
\label{eq4b:19.10.23}
&= n_1 + \vartheta(\Gr{G}_2) \\
\label{eq4c:19.10.23}
&= \card{\V{\Gr{G}_1}} + \vartheta(\Gr{G}_2).
\end{align}
The graph $\Gr{U}$ is an induced subgraph of both graphs $\Gr{G}_1 \NS \Gr{G}_2$ and
$\Gr{G}_1 \NNS \Gr{G}_2$ since the former graph can be obtained from each of the latter
two graphs by removing the vertices in $\Gr{G}_1$ along with their incident edges. By
Item~\ref{item: subgraphs} of Section~\ref{subsection: Lovasz theta-function}, it follows that
\begin{align}
\label{eq4d:19.10.23}
& \vartheta(\Gr{G}_1 \NS \Gr{G}_2) \geq \vartheta(\Gr{U}),  \\
\label{eq4e:19.10.23}
& \vartheta(\Gr{G}_1 \NNS \Gr{G}_2) \geq \vartheta(\Gr{U}),
\end{align}
and then the combination of \eqref{eq4a:19.10.23}--\eqref{eq4e:19.10.23} results in
\eqref{eq1:30.09.23} and \eqref{eq2:30.09.23}.

Sufficient conditions for inequalities \eqref{eq1:30.09.23} and \eqref{eq2:30.09.23}
to hold as equalities are subsequently proved.
Let $\{{\bf{u}}_i: i \in \V{\Gr{U}} \}$ and ${\bf{c}}$ be, respectively, an orthonormal
representation of $\Gr{U}$ and a unit vector such that
\begin{align}
\label{eq5:19.10.23}
\vartheta(\Gr{U}) = \max_{i \in \V{\Gr{U}}} \,
\frac1{\bigl( {\bf{c}}^{\mathrm{T}} {\bf{u}}_i \bigr)^2}.
\end{align}
The vertex set of the graph $\Gr{U}$ consists of $n_1$ isolated vertices, denoted by $v'_1, \ldots, v'_{n_1}$,
and the vertex set of $\Gr{G}_2$. The vertex set of each of the graphs $\Gr{G}_1 \NS \Gr{G}_2$ and
$\Gr{G}_1 \NNS \Gr{G}_2$ can be regarded as the disjoint union of the vertex sets of $\Gr{U}$ and $\Gr{G}_1$,
where $\V{\Gr{U}} = \{v'_1, \ldots, v'_{n_1} \} \cup \V{\Gr{G}_2}$ and $\V{\Gr{G}_1} = \{v_1, \ldots, v_{n_1}\}$.
We next consider separately the Lov\'{a}sz $\vartheta$-functions of the two graphs $\Gr{G}_1 \NS \Gr{G}_2$
and $\Gr{G}_1 \NNS \Gr{G}_2$.
\begin{enumerate}
\item \label{item: proof - NS}
Consider the graph $\Gr{G}_1 \NS \Gr{G}_2$, where $n_1 \eqdef \card{\V{\Gr{G}_1}}$.
Also consider an orthonormal representation of the graph $\Gr{U}$, along with a unit vector
$\bf{c}$, attaining the maximum on the right-hand side of \eqref{eq5:19.10.23}.
Suppose that there exists a permutation $\pi \colon \OneTo{n_1} \to \OneTo{n_1}$
such that $\{v_i, v_{\pi(i)}\} \in \E{\Gr{G}_1}$ for all $i \in \OneTo{n_1}$.
Select the unit vector representing the vertex $v_{\pi(i)} \in \V{\Gr{G}_1}$ to be
equal to the vector representing the vertex $v'_i \in \V{\Gr{U}}$. This facilitates an orthonormal
representation of $\Gr{G}_1 \NS \Gr{G}_2$ using the same set of vectors as in the orthonormal
representation of the induced subgraph $\Gr{U}$. Indeed, the vertices $v'_1, \ldots, v'_{n_1}$
are nonadjacent to each other in $\Gr{U}$ or also in $\Gr{G}_1 \NS \Gr{G}_2$.
Consequently, due to the orthonormality of the set of vectors representing the $n_1$ vertices
$v'_1, \ldots, v'_{n_1}$, the vector representing any vertex $v'_i$ (with $i \in \OneTo{n_1}$)
is orthogonal to all vectors representing vertices in $\Gr{G}_1$, except for
the vertex $v_{\pi(i)}$. This exception exists because, by construction, the unit
vectors that represent $v'_i$ and $v_{\pi(i)}$ are identical. Specifically, for all
$i \in \OneTo{n_1}$, this setup ensures the required orthogonality between the
representing vector of $v'_i$ and the representing vectors of all vertices in $\Gr{G}_1$
that are not adjacent to $v'_i$.

It is worth noting that the considered orthonormal representation of $\Gr{G}_1 \NS \Gr{G}_2$
has two additional features:
\begin{itemize}
\item
Given that each vertex in $\Gr{G}_1$ is adjacent to every vertex in $\Gr{G}_2$, it is
unnecessary to check the orthogonality of pairs of vectors representing any vertex from
$\Gr{G}_1$ and any vertex from $\Gr{G}_2$. Despite this, these vector pairs are indeed orthogonal.
This orthogonality results from the specific construction of the graph $\Gr{U}$ and its
orthonormal representation, ensuring that the vectors representing vertices $v'_1, \ldots, v'_{n_1}$
are orthogonal to those representing the vertices in $\Gr{G}_2$. Moreover, there exists an
injective mapping (by assumption) between the vectors representing the vertices in $\Gr{G}_1$
and those representing the vertices $v'_1, \ldots, v'_{n_1}$, from which the orthogonality
between the representing vectors of any vertex pair from $\Gr{G}_1$ and $\Gr{G}_2$ follows.
\item
By assumption, there is an injective mapping from the orthonormal vectors representing the vertices
$v'_1, \ldots, v'_{n_1}$ to the vectors representing the vertices in $\Gr{G}_1$. Consequently, the
vectors representing the vertices in $\Gr{G}_1$ are also orthonormal. Therefore, any two of these
unit vectors that represent nonadjacent and distinct vertices in $\Gr{G}_1$ are, in particular, orthogonal.
\end{itemize}

Consequently, under the assumption on the existence of a permutation $\pi$ of $\OneTo{n_1}$ as described,
consider $\{{\bf{u}}_i: i \in \V{\Gr{G}_1 \NS \Gr{G}_2} \}$ as the proposed orthonormal representation of
the graph $\Gr{G}_1 \NS \Gr{G}_2$. It then follows that
\begin{align}
\label{eq6a:19.10.23}
\vartheta(\Gr{G}_1 \NS \Gr{G}_2) &\leq \max_{i \in \V{\Gr{G}_1 \NS \Gr{G}_2}} \,
\frac1{\bigl( {\bf{c}}^{\mathrm{T}} {\bf{u}}_i \bigr)^2} \\
\label{eq6b:19.10.23}
&= \max_{i \in \V{\Gr{U}}} \, \frac1{\bigl( {\bf{c}}^{\mathrm{T}} {\bf{u}}_i \bigr)^2} \\
\label{eq6c:19.10.23}
&= \vartheta(\Gr{U}),
\end{align}
where inequality \eqref{eq6a:19.10.23} arises from applying \eqref{eq: Lovasz theta function} to
the graph $\Gr{G}_1 \NS \Gr{G}_2$, based on the constructed orthonormal representation of that graph
and considering that $\mathbf{c}$ is a unit vector. Equality \eqref{eq6b:19.10.23} holds because
the considered orthonormal representation of $\Gr{G}_1 \NS \Gr{G}_2$ utilizes the same set of unit
vectors as the orthonormal representation of the induced subgraph $\Gr{U}$, also incorporating the
same unit vector $\mathbf{c}$. Furthermore, equality \eqref{eq6c:19.10.23} is due to
\eqref{eq5:19.10.23}. Combining \eqref{eq4d:19.10.23} with \eqref{eq6a:19.10.23}--\eqref{eq6c:19.10.23},
under the aforementioned assumption, it follows that
\begin{align}
\label{eq6d:19.10.23}
\vartheta(\Gr{G}_1 \NS \Gr{G}_2) = \vartheta(\Gr{U}).
\end{align}
Combining \eqref{eq4a:19.10.23}--\eqref{eq4c:19.10.23} and \eqref{eq6d:19.10.23}
gives that inequality \eqref{eq1:30.09.23} holds as an equality.
\item \label{item: proof - NNS}
Consider the graph $\Gr{G}_1 \NNS \Gr{G}_2$, where $n_1 \eqdef \card{\V{\Gr{G}_1}}$.
Suppose that there exists a permutation $\pi$ of $\OneTo{n_1}$ such that, for all
$i \in \OneTo{n_1}$, the conditions $i \neq \pi(i)$ and
$\{v_i, v_{\pi(i)}\} \notin \E{\Gr{G}_1}$ hold.
The proof that these conditions are sufficient for equality in inequality~\eqref{eq2:30.09.23}
is analogous to the proof for Item~\ref{item: proof - NS} regarding the satisfiability of
inequality~\eqref{eq1:30.09.23} with equality. The primary distinction involves a
minor modification in the new proof, reflecting differences between the definitions
of the NS and NNS joins of graphs, as delineated in
Definitions~\ref{definition: neighbors splitting join of graphs}
and~\ref{definition: nonneighbors splitting join of graphs}. This adjustment accounts for
the altered assumption compared to that used in Item~\ref{item: proof - NS}.
\end{enumerate}

The sufficient conditions for the equalities in inequalities \eqref{eq3:02.10.23} and
\eqref{eq3b:02.10.23} are finally demonstrated. Assume the above sufficient condition
for equality in \eqref{eq1:30.09.23} is met and $\indnum{\Gr{G}_2} = \vartheta(\Gr{G}_2)$. Then,
\begin{align}
\label{eq10:02.10.23}
\vartheta(\Gr{G}_1 \NS \Gr{G}_2) &= \card{\V{\Gr{G}_1}} + \vartheta(\Gr{G}_2) \\
\label{eq11:02.10.23}
&= \card{\V{\Gr{G}_1}} + \indnum{\Gr{G}_2} \\[0.1cm]
\label{eq12:02.10.23}
&\leq \indnum{\Gr{G}_1 \NS \Gr{G}_2} \\
\label{eq13:02.10.23}
&\leq \vartheta(\Gr{G}_1 \NS \Gr{G}_2),
\end{align}
where \eqref{eq10:02.10.23} is justified by the sufficient condition for equality
in \eqref{eq1:30.09.23}; \eqref{eq11:02.10.23} holds by assumption;
\eqref{eq12:02.10.23} holds by \eqref{eq3:02.10.23}, and \eqref{eq13:02.10.23} holds
by \eqref{eq1a: sandwich} (the Lov\'{a}sz $\vartheta$-function of a graph is an upper
bound on its independence number).
Hence, the two inequalities in \eqref{eq12:02.10.23} and \eqref{eq13:02.10.23} hold with equalities,
so \eqref{eq3:02.10.23} holds with equality.
Similarly, if the above sufficient condition for equality in \eqref{eq2:30.09.23} is met and
$\indnum{\Gr{G}_2} = \vartheta(\Gr{G}_2)$, then also \eqref{eq3b:02.10.23} holds with equality.
This completes the proof of Theorem~\ref{theorem: on Graph Invariants of NS/NNS Joins of Graphs}.

\subsubsection{Proof of Theorem~\ref{theorem: Shannon Capacity of NS and NNS Join Graphs}}
\label{subsubsection: proof of a theorem on the Shannon Capacity of NS and NNS Join Graphs}

Let $\Gr{G}_1$ and $\Gr{G}_2$ be finite, undirected, and simple graphs.
Consider the NS join of graphs $\Gr{G}_1 \NS \Gr{G}_2$. In light of Remark~\ref{remark: perfect matchings},
the assumptions in Item~\ref{Item 1 - Shannon Capacity of NS and NNS Join Graphs} of
Theorem~\ref{theorem: Shannon Capacity of NS and NNS Join Graphs} are equivalent to the satisfiability
of the sufficient condition for inequality \eqref{eq3:02.10.23} to hold with equality.
By Theorem~\ref{theorem: on Graph Invariants of NS/NNS Joins of Graphs}, the satisfiability
of the latter condition also suffices for inequality \eqref{eq1:30.09.23} to hold with equality.
Consequently, by the assumptions in Item~\ref{Item 1 - Shannon Capacity of NS and NNS Join Graphs}
of Theorem~\ref{theorem: Shannon Capacity of NS and NNS Join Graphs},
\begin{align}
\label{eq14b:02.10.23}
& \indnum{\Gr{G}_1 \NS \Gr{G}_2} = \vartheta(\Gr{G}_1 \NS \Gr{G}_2) = \card{\V{\Gr{G}_1}} + \vartheta(\Gr{G}_2).
\end{align}
Since the inequalities $\indnum{\Gr{G}} \leq \Theta(\Gr{G}) \leq \vartheta(\Gr{G})$ hold for every graph $\Gr{G}$,
it also follows from \eqref{eq14b:02.10.23} that
\begin{align}
\label{eq14c:02.10.23}
& \Theta(\Gr{G}_1 \NS \Gr{G}_2) = \card{\V{\Gr{G}_1}} + \vartheta(\Gr{G}_2).
\end{align}
Combining \eqref{eq14b:02.10.23} and \eqref{eq14c:02.10.23} gives \eqref{eq14:02.10.23}, which validates
Item~\ref{Item 1 - Shannon Capacity of NS and NNS Join Graphs} of Theorem~\ref{theorem: Shannon Capacity of NS and NNS Join Graphs}.

Consider the NNS join of graphs $\Gr{G}_1 \NNS \Gr{G}_2$.
Item~\ref{Item 2 - Shannon Capacity of NS and NNS Join Graphs} of
Theorem~\ref{theorem: Shannon Capacity of NS and NNS Join Graphs} is justified likewise
by the replacement of the condition $\mathrm{per}(\A(\Gr{G}_1)) > 0$ with the modified condition $\mathrm{per}(\A(\CGr{G}_1)) > 0$,
while maintaining the condition $\indnum{\Gr{G}_2} = \vartheta(\Gr{G}_2)$.
Under these assumptions, Item~\ref{Item 2 - Shannon Capacity of NS and NNS Join Graphs}
of Theorem~\ref{theorem: Shannon Capacity of NS and NNS Join Graphs}
is justified by combining Theorem~\ref{theorem: on Graph Invariants of NS/NNS Joins of Graphs}
and Remark~\ref{remark: perfect matchings}.

Item~\ref{Item 3 - Shannon Capacity of NS and NNS Join Graphs} of
Theorem~\ref{theorem: Shannon Capacity of NS and NNS Join Graphs} finally follows
by combining Items~\ref{Item 1 - Shannon Capacity of NS and NNS Join Graphs}
and~\ref{Item 2 - Shannon Capacity of NS and NNS Join Graphs} of that theorem.

\subsubsection{Proof of Corollary~\ref{corollary: the Lovasz theta-function of regular NICS graphs}}
\label{subsubsection: proof of corollary: the Lovasz theta-function of regular NICS graphs}
By the assumptions outlined in Items~\ref{corollary: item 1 - On NICS regular graphs}
and~\ref{corollary: item 2 - On NICS regular graphs} of
Corollary~\ref{corollary: the Lovasz theta-function of regular NICS graphs},
equalities \eqref{eq1a:26.09.23}--\eqref{eq1c:26.09.23} hold in light of the satisfiability
of the sufficient conditions for the leftmost and rightmost inequalities in \eqref{eq:21.10.22a1}
and \eqref{eq:21.10.22a2} to hold as equalities.
Additionally, if $\Gr{G}$ and $\Gr{H}$ are cospectral $d$-regular graphs on $n$ vertices, then their
complements $\CGr{G}$ and $\CGr{H}$ are cospectral $(n-d-1)$-regular graphs. This assertion is valid
since both complement graphs share the largest eigenvalue of their adjacency matrices, which is $n-d-1$,
and all other eigenvalues, accounting for multiplicities, are uniformly shared and arranged in nonincreasing
order as $-1 - \lambda_{n+1-i}$ for each $i$ from 1 to $n-1$.

\subsection{Discussions and Examples for the Results in Section~\ref{subsection: main results - cospectral and nonisomorphic graphs}}
\label{subsection: examples - cospectral and nonisomorphic graphs}

\begin{remark}[The independence numbers of NS/NNS joins of graphs]
\label{remark: the independence numbers of NS/NNS Joins of graphs}
In the construction used in the proof of Theorem~\ref{theorem: existence of NICS graphs},
we encounter a situation where \eqref{eq3:02.10.23} holds with equality.
Inequalities \eqref{eq3:02.10.23} and \eqref{eq3b:02.10.23} in
Theorem~\ref{theorem: on Graph Invariants of NS/NNS Joins of Graphs} may hold, however, with strict
inequalities. We next show two possible selections of $\Gr{G}_1$ in the NS join of graphs
$\Gr{G}_1 \NS \Gr{G}_2$, for which the difference between the left
and right sides of \eqref{eq3:02.10.23} can be made arbitrarily large.

\begin{itemize}
\item Let $\Gr{G}_1$ be a graph whose vertex set includes a subset of $k_1$ isolated vertices, and let $\Gr{G}_2$ be an arbitrary
finite, undirected, and simple graph. Consider the NS join of graphs $\Gr{G}_1 \NS \Gr{G}_2$.
In this construction, the $k_1$ isolated vertices in $\Gr{G}_1$, combined with the $\card{\V{\Gr{G}_1}}$
vertices added to the vertex set of $\Gr{G}_1 \vee \Gr{G}_2$ to form the complete vertex set
of $\Gr{G}_1 \NS \Gr{G}_2$, collectively form an independent set in $\Gr{G}_1 \NS \Gr{G}_2$.
Consequently, $\indnum{\Gr{G}_1 \NS \Gr{G}_2} \geq \card{\V{\Gr{G}_1}} + k_1$, so
\begin{align}
\label{eq4:02.10.23}
\indnum{\Gr{G}_1 \NS \Gr{G}_2} - \bigl( \card{\V{\Gr{G}_1}} + \indnum{\Gr{G}_2} \bigr)
\geq k_1 - \indnum{\Gr{G}_2}.
\end{align}
The right-hand side of \eqref{eq4:02.10.23} can be made arbitrarily large as we let $k_1$ tend to infinity.

\item Let $\Gr{G}_1$ be a star graph with $k_1$ leaves (i.e., vertices of degree~1). Let its
vertex set be given by $\V{\Gr{G}_1} = \{v_1, \ldots, v_{k_1+1}\}$,
where $v_1$ is the central vertex adjacent to all the leaves $v_2, \ldots, v_{k_1+1}$ in $\Gr{G}_1$.
Consider the NS join of graphs $\Gr{G}_1 \NS \Gr{G}_2$, with an arbitrary
finite, undirected, and simple graph $\Gr{G}_2$. Denote by $v'_1, \ldots, v'_{k_1+1}$ the $k_1+1$
vertices that are added to the join of graphs $\Gr{G}_1 \vee \Gr{G}_2$ to form the vertex set of
the NS join of graphs $\Gr{G}_1 \NS \Gr{G}_2$. In that graph, the vertex $v'_1$ is adjacent to the $k_1$ vertices
$v_2, \ldots, v_{k_1+1}$ in $\Gr{G}_1$, and each of the vertices $v'_2, \ldots, v'_{k_1+1}$
is adjacent only to $v_1$. Hence, the disjoint union
$\set{I} \eqdef \{v'_2, \ldots, v'_{k_1+1}\} \cup \{v_2, \ldots, v_{k_1+1}\}$
is an independent set in the graph $\Gr{G}_1 \NS \Gr{G}_2$, so
\begin{align}
\label{eq6:02.10.23}
& \indnum{\Gr{G}_1 \NS \Gr{G}_2} \geq \card{\set{I}} = 2k_1, \\
\label{eq7:02.10.23}
\Rightarrow \, & \indnum{\Gr{G}_1 \NS \Gr{G}_2} - \bigl( \card{\V{\Gr{G}_1}} + \indnum{\Gr{G}_2} \bigr)
\geq k_1 - \indnum{\Gr{G}_2} - 1.
\end{align}
Once again, the right-hand side of \eqref{eq7:02.10.23} can be made arbitrarily large by letting $k_1$ tend to infinity.
\end{itemize}
\end{remark}

\begin{remark}[The Lov\'{a}sz $\vartheta$-function of NS/NNS joins of graphs]
\label{remark: the Lovasz theta-function of NS/NNS joins of graphs}
Let $\Gr{G}_1 = \CoBG{1}{9}$ be the star graph on 10 vertices (with 9 leaves), and let $\Gr{G}_2$ be the graph
on the left side of Figure~\ref{fig:vanDamH03}. The star graph $\Gr{G}_1$ clearly does not satisfy the sufficient
condition in Theorem~\ref{theorem: on Graph Invariants of NS/NNS Joins of Graphs} for equality in \eqref{eq1:30.09.23},
but it satisfies the sufficient condition in Theorem~\ref{theorem: on Graph Invariants of NS/NNS Joins of Graphs}
for equality in \eqref{eq2:30.09.23}. The Lov\'{a}sz $\vartheta$-functions of the NS and NNS
joints of graphs $\Gr{G}_1 \NS \Gr{G}_2$ and $\Gr{G}_1 \NNS \Gr{G}_2$, respectively, are computed
by solving the SDP problem presented in \eqref{eq: SDP problem - Lovasz theta-function}, using the adjacency matrices
specified in \eqref{eq: adjacency matrix - NS join} and \eqref{eq: adjacency matrix - NNS join}. This gives
\begin{align}
\label{eq1a:18.10.2023}
\vartheta(\Gr{G}_1 \NS \Gr{G}_2) &= 18 \\
\label{eq1b:18.10.2023}
& > 13.23607 \ldots \\
\label{eq1c:18.10.2023}
& = \card{\V{\Gr{G}_1}} + \vartheta(\Gr{G}_2),
\end{align}
whereas,
\begin{align}
\label{eq2a:18.10.2023}
\vartheta(\Gr{G}_1 \NNS \Gr{G}_2) &= 13.23607 \ldots \\
\label{eq2b:18.10.2023}
& = \card{\V{\Gr{G}_1}} + \vartheta(\Gr{G}_2),
\end{align}
and the equalities \eqref{eq1c:18.10.2023} and \eqref{eq2b:18.10.2023} are justified by \eqref{eq5a:29.09.23}
and since $\card{\V{\Gr{G}_1}} = 10$. This shows that if the sufficient condition in
Theorem~\ref{theorem: on Graph Invariants of NS/NNS Joins of Graphs} for equality in \eqref{eq1:30.09.23}
is violated, then \eqref{eq1:30.09.23} may hold with strict inequality (see \eqref{eq1b:18.10.2023}).
Conversely, \eqref{eq2:30.09.23} holds here with equality (see \eqref{eq2a:18.10.2023}--\eqref{eq2b:18.10.2023}),
as it is expected in light of the explanation provided above.
\end{remark}

\begin{example}[NS and NNS joins of Kneser graphs]
\label{example: NS and NNS Joins of Kneser Graphs}
Let $\Gr{G}_1 = \KG{n_1}{r_1}$ and $\Gr{G}_2 = \KG{n_2}{r_2}$ be Kneser graphs with $n_1 \geq 2r_1$
and $n_2 \geq 2r_2$ (so, these are nonempty graphs). Then, for $i \in \{1,2\}$, $\Gr{G}_i$ is a
$d_i$-regular graph with $d_i = \binom{n_i-r_i}{r_i}$, and
\begin{align}
\label{eq: order KS graph}
& \card{\V{\Gr{G}_i}} = \binom{n_i}{r_i},  \\
\label{eq: chrnum KS graph}
& \chrnum{\Gr{G}_i} = n_i - 2r_i + 2, \\
\label{eq: independence number KS graph}
& \indnum{\Gr{G}_i} = \binom{n_i-1}{r_i-1} = \Theta(\Gr{G}_i) = \vartheta(\Gr{G}_i),
\end{align}
where \eqref{eq: order KS graph}, \eqref{eq: chrnum KS graph}, and the first equality in \eqref{eq: independence number KS graph}
are justified, e.g., in Chapter~43 of \cite{AignerZ18}; the second and third equalities in \eqref{eq: independence number KS graph}
are due to Theorem~13 in \cite{Lovasz79_IT}.
In light of Theorems~\ref{theorem: on Graph Invariants of NS/NNS Joins of Graphs} and~\ref{theorem: Shannon Capacity of NS and NNS Join Graphs},
it follows from \eqref{eq: order KS graph}--\eqref{eq: independence number KS graph} that
\begin{align}
\label{eq1:22.10.23}
& \indnum{\Gr{G}_1 \NS \Gr{G}_2} = \Theta(\Gr{G}_1 \NS \Gr{G}_2) = \vartheta(\Gr{G}_1 \NS \Gr{G}_2) = \binom{n_1}{r_1} + \binom{n_2-1}{r_2-1}, \\[0.2cm]
\label{eq2:22.10.23}
& \indnum{\Gr{G}_1 \NNS \Gr{G}_2} = \Theta(\Gr{G}_1 \NNS \Gr{G}_2) = \vartheta(\Gr{G}_1 \NNS \Gr{G}_2) = \binom{n_1}{r_1} + \binom{n_2-1}{r_2-1}, \\[0.2cm]
\label{eq3:22.10.23}
& \chrnum{\Gr{G}_1 \NS \Gr{G}_2} = n_1 + n_2 - 2(r_1+r_2) + 4 = \chrnum{\Gr{G}_1 \NNS \Gr{G}_1}, \\[0.2cm]
\label{eq4:22.10.23}
& \clnum{\CGr{G}_1 \NS \CGr{G}_2} = \binom{n_1-1}{r_1-1} + \binom{n_2-1}{r_2-1} = \clnum{\CGr{G}_1 \NNS \CGr{G}_2}.
\end{align}
Specifically, let $\Gr{G}_1 = \Gr{G}_2 = \Gr{P}$, where $\Gr{P}$ denotes the Petersen graph; $\Gr{P}$ is
a graph on~10 vertices, and $\Gr{P} \cong \KG{5}{2}$, so \eqref{eq1:22.10.23} and \eqref{eq2:22.10.23}
with $n_1=n_2=5$ and $r_1=r_2=2$ give
\begin{align}
\label{eq: NS and NNS Joins of Petersen Graphs}
\vartheta(\Gr{P} \NS \Gr{P}) = 14 = \vartheta(\Gr{P} \NNS \Gr{P}).
\end{align}
Equality~\eqref{eq: NS and NNS Joins of Petersen Graphs} was verified numerically using a combination of the CVX and SageMath software tools
\cite{GrantB20,SageMath}. This was done by first constructing the NS and NNS join of graphs $\Gr{P} \NS \Gr{P}$
and $\Gr{P} \NNS \Gr{P}$, respectively, based on their adjacency matrices in \eqref{eq: adjacency matrix - NS join} and
\eqref{eq: adjacency matrix - NNS join}. Subsequently, the SDP problem presented in \eqref{eq: SDP problem - Lovasz theta-function} was solved
for computing the Lov\'{a}sz $\vartheta$-functions of the two constructed graphs.
\end{example}

\begin{example}[NS/NNS joins of a regular graph and the Tietze graph]
\label{example: The Shannon capacity of NS/NNS joins of a regular graph and Tietze graph}
Let $\Gr{G}_1$ be a nonempty and noncomplete regular graph, and let $\Gr{G}_2$ be the Tietze graph
(see Figure~\ref{fig.: Tietze graph}). The graph $\Gr{G}_2$ is a 3-regular
graph on 12 vertices, not strongly regular, nor vertex- or edge-transitive
(these properties can be verified by the SageMath software \cite{SageMath}).
\begin{figure}[h!t!]
\centering
\includegraphics[width=9cm]{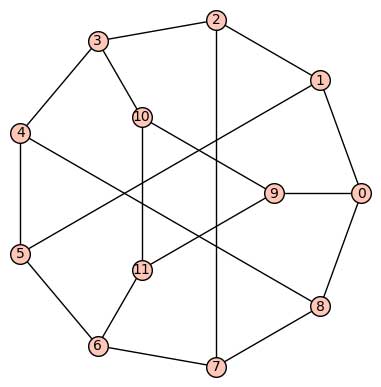}
\caption{\label{fig.: Tietze graph} \centering{The Tietze graph $\Gr{G}_2$
(see Example~\ref{example: The Shannon capacity of NS/NNS joins of a regular graph and Tietze graph}).}}
\end{figure}

It is next shown that
\begin{align}
\label{eq1:Tietze}
\indnum{\Gr{G}_2} = 5 = \vartheta(\Gr{G}_2).
\end{align}
Indeed, a maximum independent set in $\Gr{G}_2$ is given by $\{0, 3, 5, 7, 11\}$, so $\indnum{\Gr{G}_2} = 5$.
Additionally, the equality $\vartheta(\Gr{G}_2) = 5$ is verified by solving the SDP problem in \eqref{eq: SDP problem - Lovasz theta-function}:
\begin{equation}
\label{eq2:Tietze}
\mbox{\fbox{$
\begin{array}{l}
\text{maximize} \; \; \mathrm{Tr}({\bf{B}} \, {\bf{J}}_{12}) \\
\text{subject to} \\
\begin{cases}
{\bf{B}} \succeq 0, \; \; \mathrm{Tr}({\bf{B}}) = 1, \\
A_{i,j} = 1  \; \Rightarrow \;  B_{i,j} = 0, \quad i,j \in \{1, \ldots, 12\}.
\end{cases}
\end{array}$}}
\end{equation}
The maximizing matrix $\bf{B}$ in \eqref{eq2:Tietze}, obtained numerically by the CVX
software \cite{GrantB20}, is given by
\[
\label{eq3:Tietze}
{\bf{B}} = \frac{1}{15} \,
\begin{pmatrix}
2 & 0 & 1 & 1 & 1 & 1 & 1 & 1 & 0 & 0 & 1 & 1 \\[0.1cm]
0 & 1 & 0 & 1 & 0 & 0 & 1 & 0 & 1 & 1 & 0 & 0 \\[0.1cm]
1 & 0 & 1 & 0 & 1 & 0 & 1 & 0 & 0 & 0 & 1 & 0 \\[0.1cm]
1 & 1 & 0 & 2 & 0 & 1 & 1 & 1 & 1 & 1 & 0 & 1 \\[0.1cm]
1 & 0 & 1 & 0 & 1 & 0 & 1 & 0 & 0 & 0 & 1 & 0 \\[0.1cm]
1 & 0 & 0 & 1 & 0 & 1 & 0 & 1 & 0 & 0 & 0 & 1 \\[0.1cm]
1 & 1 & 1 & 1 & 1 & 0 & 2 & 0 & 1 & 1 & 1 & 0 \\[0.1cm]
1 & 0 & 0 & 1 & 0 & 1 & 0 & 1 & 0 & 0 & 0 & 1 \\[0.1cm]
0 & 1 & 0 & 1 & 0 & 0 & 1 & 0 & 1 & 1 & 0 & 0 \\[0.1cm]
0 & 1 & 0 & 1 & 0 & 0 & 1 & 0 & 1 & 1 & 0 & 0 \\[0.1cm]
1 & 0 & 1 & 0 & 1 & 0 & 1 & 0 & 0 & 0 & 1 & 0 \\[0.1cm]
1 & 0 & 0 & 1 & 0 & 1 & 0 & 1 & 0 & 0 & 0 & 1
\end{pmatrix},
\]
which gives
\begin{align}
\label{eq4:Tietze}
\vartheta(\Gr{G}_2) = \mathrm{Tr}({\bf{B}} \, {\bf{J}}_{12}) = 5.
\end{align}
For comparison with \eqref{eq4:Tietze}, since the graph $\Gr{G}_2$ is
a 3-regular graph on 12 vertices, inequality \eqref{eq: Lovasz79 - Theorem 9}
(Theorem~9 in \cite{Lovasz79_IT}) with $d=3$ and $\lambda_{\min}(\Gr{G}_2) = -2.30278 \ldots$ gives
\begin{align}
\label{eq6:Tietze}
\vartheta(\Gr{G}_2) \leq - \frac{n \lambda_{\min}(\Gr{G})}{d - \lambda_{\min}(\Gr{G})} = 5.21110 \ldots,
\end{align}
so the upper bound \eqref{eq6:Tietze} is not tight. It should be noted that $\Gr{G}_2$ is neither edge-transitive
nor strongly regular. If it were, then Theorem~\ref{thm:bounds on the Lovasz function for regular graphs} would have
transformed inequality \eqref{eq6:Tietze} into an equality. This transformation would have occurred provided the
sufficient conditions for equality, specified in regard to the rightmost inequality of \eqref{eq:21.10.22a1}, were met.

By \eqref{eq1:Tietze}, it follows that the Shannon capacity of the Tietze graph is equal to~5
\begin{align}
\label{eq5:Tietze}
\Theta(\Gr{G}_2) = 5.
\end{align}
The vertices in $\Gr{G}_2$ can be partitioned into at least three independent
sets, e.g., $\{0, 2, 5, 11\}$, $\{1, 4, 7, 10\}$, and $\{3, 6, 8, 9\}$ (see Figure~\ref{fig.: Tietze graph}), so
the chromatic number of $\Gr{G}_2$ is equal to~3
\begin{align}
\label{eq7:Tietze}
\chrnum{\Gr{G}_2} = 3.
\end{align}
The maximum clique in $\Gr{G}_2$ (see Figure~\ref{fig.: Tietze graph}) is the triangle $\{9, 10, 11\}$, so
\begin{align}
\label{eq8:Tietze}
\clnum{\Gr{G}_2} = 3.
\end{align}
In light of Theorem~\ref{theorem: on Graph Invariants of NS/NNS Joins of Graphs} and Corollary~\ref{corollary: sufficient conditions for equalities},
it follows from \eqref{eq1:Tietze}, \eqref{eq7:Tietze}, and \eqref{eq8:Tietze} that
\begin{align}
\label{eq9:Tietze}
& \indnum{\Gr{G}_1 \NS \Gr{G}_2} = \Theta(\Gr{G}_1 \NS \Gr{G}_2) = \vartheta(\Gr{G}_1 \NS \Gr{G}_2) = \card{\V{\Gr{G}_1}} + 5, \\[0.1cm]
\label{eq10:Tietze}
& \indnum{\Gr{G}_1 \NNS \Gr{G}_2} = \Theta(\Gr{G}_1 \NNS \Gr{G}_2) = \vartheta(\Gr{G}_1 \NNS \Gr{G}_2) = \card{\V{\Gr{G}_1}} + 5, \\[0.1cm]
\label{eq11:Tietze}
& \chrnum{\Gr{G}_1 \NS \Gr{G}_2} = \chrnum{\Gr{G}_1} + 3 = \chrnum{\Gr{G}_1 \NNS \Gr{G}_2}, \\[0.1cm]
\label{eq12:Tietze}
& \clnum{\Gr{G}_1 \NS \Gr{G}_2} = \clnum{\Gr{G}_1} + 3 = \clnum{\Gr{G}_1 \NNS \Gr{G}_2}.
\end{align}
\end{example}

\begin{example}[Chang graphs]
\label{example: Chang Graphs}
Chang graphs are three cospectral and strongly regular graphs with the identical
parameters $\srg{28}{12}{6}{4}$. The line graph of the complete graph $\CoG{8}$
is also a strongly regular graph with the same set of parameters.
By Theorem~\ref{theorem: eigenvalues of srg}, these graphs have an
identical spectrum with $\lambda_1 = 12$, $\lambda_2 = \ldots = \lambda_8 = 4$
(multiplicity~7), and $\lambda_9 = \ldots = \lambda_{28} = -2$ (multiplicity~20).
It can be verified, using the SageMath software \cite{SageMath}, that all pairs
among these cospectral graphs are nonisomorphic and that the three Chang graphs
and their complements are neither vertex-transitive nor edge-transitive.
In light of Corollary~\ref{corollary: the Lovasz theta-function of regular NICS graphs},
those four strongly regular graphs have an identical Lov\'{a}sz $\vartheta$-function
valued at~4, as demonstrated by either \eqref{eq1a:26.09.23} or \eqref{eq1b:26.09.23}.
Furthermore, the graph complements, which are strongly regular, cospectral, and nonisomorphic,
share a common value of~7 for their Lov\'{a}sz $\vartheta$-functions, as demonstrated by
\eqref{eq1c:26.09.23}.
\end{example}

\section{Bounds on Graph Invariants via the Lov\'{a}sz $\vartheta$-Function}
\label{section: bounds on graph invariants}
Building on preliminaries in Section~\ref{section: preliminaries} and introducing additional
background in Section~\ref{subsection: preliminaries - bounds on graph invariants}, this section utilizes
the Lov\'{a}sz $\vartheta$-function and its unique properties to derive bounds on various graph invariants.
These bounds are presented in Section~\ref{subsection: main results - bounds on graph invariants},
and their proofs are provided in Section~\ref{subsection: proofs - bounds on graph invariants}.
The tightness of these bounds is examined in Section~\ref{subsection: examples - bounds on graph invariants},
including comparisons with established bounds in the literature.

\subsection{Specialized Preliminaries for Section~\ref{section: bounds on graph invariants}}
\label{subsection: preliminaries - bounds on graph invariants}

\subsubsection{Perfect graphs}
\label{subsubsection: perfect graphs}
For all graphs, the chromatic number is greater than or equal to the clique number, but they can be far apart.
It is specifically possible to construct finite, connected, and simple graphs whose clique numbers are equal
to~2, while having an arbitrarily large chromatic number (see, e.g., page~314 in \cite{AignerZ18}).
Perfect graphs are characterized by the property that not only their chromatic and clique numbers
coincide, but also the same coincidence applies to every induced subgraph. These graphs received a lot of attention (see,
e.g., the books \cite{AR01,BeinekeGW21,Golumbic04}), dating back to the paper by Claude Shannon in 1956
\cite{Shannon56} and the work by Claude Berge in the early sixties \cite{Berge63}. Perfect graphs include many
important families of graphs such as bipartite graphs, line graphs of bipartite graphs, chordal graphs,
comparability graphs, and the complements of all these graphs \cite{Lovasz83}.
Unlike general graphs, many graph invariants of all perfect graphs are computationally feasible \cite{GrotschelLS84,Lovasz83}.

\begin{definition}[Perfect graph]
\label{definition: perfect graphs}
A graph $\Gr{G}$ is {\em perfect} if $\chrnum{\Gr{H}} = \clnum{\Gr{H}}$ for every induced subgraph~$\Gr{H}$~of~$\Gr{G}$.
\end{definition}

\begin{theorem}[Weak perfect graph theorem \cite{Lovasz72}]
\label{theorem: Weak Perfect Graph Theorem}
The complement of a perfect graph retains this perfection. Equivalently, a graph is perfect if and only if its complement is perfect.
\end{theorem}

A cycle graph of length $2k+1$, with $k \geq 2$, satisfies $3 = \chrnum{\CG{2k+1}} > \clnum{\CG{2k+1}} = 2$, and its complement
satisfies $k+1 = \chrnum{\CGr{\CG{2k+1}}} > \clnum{\CGr{\CG{2k+1}}} = k$. This shows that cycle graphs of odd length greater than~3,
as well as their complements, are imperfect. By definition, the family of perfect graphs is closed under the operation of taking induced subgraphs.
Therefore, these graphs can be characterized by the exclusion of a specific family of induced subgraphs, denoted $\mathscr{G}$.
The strong perfect graph theorem, presented as follows, states that these two examples are the only elements within $\mathscr{G}$.

\begin{theorem}[Strong perfect graph theorem \cite{ChudnovskyRST06}]
\label{theorem: Strong Perfect Graph Theorem}
A graph is perfect if and only if neither the graph nor its complement contains an
induced cycle of odd length greater than 3.
\end{theorem}
The reader is referred to \cite{Trotignon15} as a survey paper on perfect graphs, which is mostly focused on the strong perfect graph theorem.
This theorem was conjectured by Berge \cite{Berge63}, and it was proved four decades later (in 2002) by Chudnovsky, Robertson, Seymour, and Thomas \cite{ChudnovskyRST06}.

\subsubsection{Triangle-free graphs}
\label{subsubsection: triangle-free graphs}

A triangle-free graph is a graph that does not contain any triangles, which are cycles of length three, as subgraphs.
Triangle-free graphs have garnered attention, as evidenced by their coverage in theoretical graph theory,
game theory, and advanced coding techniques for data storage
\cite{Alon94,Balogh23,Csikvari22,BermanDSMW16,BermanSM22,BernshteynBCK23,PirotS21,GlockJS23,CameronDR16,BargZ22,BargSY23,DengHWX23,HuangX23,Shearer83,Shearer91,Griggs83}.
This interest underscores the importance of investigating graph invariants of triangle-free graphs.

Lower bounds on the independence number of triangle-free graphs were derived in \cite{Shearer83,Shearer91}, as
presented in the next result.
\begin{theorem}[Lower bound on the independence number of triangle-free graphs \cite{Shearer91}]
\label{thm: Shearer, 1991}
Let $\Gr{G}$ be a triangle-free graph on $n$ vertices with a degree sequence
$\{d_1, \ldots, d_n\}$. Let $f \colon \{0, 1, 2, \ldots\} \to \Reals$ be defined
by the recursion
\begin{align}
\label{eq:recursion}
f(k) = \frac{1 + (k^2-k) \, f(k-1)}{k^2+1}, \quad k \in \naturals,
\end{align}
with $f(0) = 1$. Then,
\begin{align}
\label{eq:LB-Sheaer91}
\indnum{\Gr{G}} \geq \sum_{\ell=1}^n f(d_\ell).
\end{align}
\end{theorem}
\begin{corollary}[Upper bound on the fractional chromatic number]
\label{corollary: triangle-free and vertex-transitive graphs}
Let $\Gr{G}$ be a triangle-free and vertex-transitive graph on $n$ vertices with degree $d$.
Then,
\begin{align}
\label{eq1:13.08.23}
\fchrnum{\Gr{G}} \leq \frac{n}{\lceil n f(d) \rceil} \leq \frac{1}{f(d)}.
\end{align}
\end{corollary}
\begin{proof}
The triangle-free graph $\Gr{G}$ is $d$-regular, which implies by
\eqref{eq:LB-Sheaer91} that $\indnum{\Gr{G}} \geq n \, f(d)$.
Also, since $\Gr{G}$ is vertex-transitive on $n$ vertices,
$\fchrnum{\Gr{G}} = \dfrac{n}{\indnum{\Gr{G}}} \leq \dfrac{n}{\lceil n f(d) \rceil} \leq \dfrac{1}{f(d)}$.
\end{proof}

A recent work \cite{PirotS21} provides upper bounds on
the fractional chromatic number of a triangle-free graph, expressed as a function of the maximal degree
of its vertices.
\begin{theorem}[Upper bound on the fractional chromatic number of a triangle-free graph \cite{PirotS21}]
\label{thm: Pirot and Sereni, 2021}
For every triangle-free graph $\Gr{G}$ of maximum degree $\Delta = \Delta(\Gr{G})$,
\begin{align}
\label{eq:UB1-PS01}
\fchrnum{\Gr{G}} \leq 1 + \min_{k \in \naturals} \, \inf_{\lambda > 0} \;
\frac{(1+\lambda)^k + \lambda (1 + \lambda) \Delta}{\lambda (1+k \lambda)} \, .
\end{align}
\end{theorem}
The following result gives an upper bound on the fractional chromatic number of
graphs with a large girth (so they are, in particular, triangle-free graphs).
\begin{theorem}[Upper bound on the chromatic number of graphs with a large girth, \cite{PirotS21}]
\label{theorem: large girth}
\label{eq:UB2-PS01}
If $\Gr{G}$ is a graph of girth of at least 7, then
\begin{align}
\fchrnum{\Gr{G}} \leq 1 + \min_{k \in \naturals} \; \frac{2 \Delta(\Gr{G}) + 2^{k-3}}{k}.
\end{align}
\end{theorem}
Numerical results that rely on Theorem~\ref{theorem: large girth} are presented in Table~1.4 of \cite{PirotS21}.

\subsubsection{Lower bounds on the fractional chromatic number of a graph}
\label{subsubsection: lower bounds on the fractional chromatic number of a graph - background}
The fractional chromatic number of a graph $\Gr{G}$ is lower-bounded by the Lov\'{a}sz
$\vartheta$-function of the complement $\CGr{G}$
\begin{align}
\label{eq1:04.10.23}
\fchrnum{\Gr{G}} \geq \vartheta(\CGr{G})
\end{align}
(see Exercise~11.20 in \cite{Lovasz19}, and \eqref{eq1a: sandwich} here).
For a self-contained presentation in the continuation of this section, we next provide a proof of \eqref{eq1:04.10.23}.

\begin{proof}
By the formulation of the fractional chromatic number of a graph as an LP problem (see
\eqref{eq: fractional chromatic number}), let $\fchrnum{\CGr{G}} \eqdef t$ be the smallest
value for which there exists a family $\{\set{B}_j\}_{j=1}^p$ of cliques in $\Gr{G}$,
along with nonnegative weights $\{\tau_j\}_{j=1}^p$ such that
\begin{align}
\label{eq1:LP}
& \hspace*{0.2cm} \sum_{j=1}^p \tau_j = t, \\
\label{eq2:LP}
& \sum_{j: \, v \in \set{B}_j} \tau_j \geq 1, \quad \forall \, v \in \V{\Gr{G}}.
\end{align}
For every $v \in \V{\Gr{G}}$, let the vertex $v$ be assigned the vector
\begin{align}
\label{eq:u_vec}
{\bf{u}}(v) = \sum_{j: \, v \in \set{B}_j} \alpha_j(v) \, {\bf{e}}_j,
\end{align}
where, for $j \in \OneTo{p}$, ${\bf{e}}_j$ is the $p$-dimensional unit column vector characterized by having $\tt{1}$ in its
$j$-th component and zeros in all other components, and
\begin{align}
\label{eq:alpha_j}
\alpha_j(v) \eqdef \sqrt{\frac{\tau_j}{\underset{\ell: \, v \in \set{B}_\ell}{\sum} \tau_\ell}},
\quad \forall \, j: \, v \in \set{B}_j.
\end{align}
This mapping gives an orthonormal representation of $\Gr{G}$ since these are unit vectors
in $\Reals^p$, and if $\{v_1, v_2\} \notin \E{\Gr{G}}$, then there exists no $j \in \OneTo{p}$
such that $v_1, v_2 \in \set{B}_j$ (since the set $\set{B}_j$ forms a clique in
$\Gr{G}$), so ${\bf{u}}(v_1)^{\mathrm{T}} \, {\bf{u}}(v_2) = 0$.
Let the $p$-dimensional handle ${\bf{c}}$ of that orthonormal representation be given by
\begin{align}
\label{eq0:24.10.23}
{\bf{c}} \eqdef \frac1{\sqrt{t}} \; (\sqrt{\tau}_1, \ldots, \sqrt{\tau_p})^{\mathrm{T}},
\end{align}
so it follows from \eqref{eq1:LP} that $\|{\bf{c}}\|=1$. For all $v \in \V{\Gr{G}}$,
\begin{align}
\label{eq1:24.10.23}
{\bf{u}}(v)^{\mathrm{T}} \, {\bf{c}}
&= \frac1{\sqrt{t}} \sum_{j: \, v \in \set{B}_j} \alpha_j(v) \, \sqrt{\tau_j} \\
\label{eq2:24.10.23}
&= \frac1{\sqrt{t}} \, \sqrt{\sum_{j: \, v \in \set{B}_j} \tau_j} \\
\label{eq3:24.10.23}
&\geq \frac1{\sqrt{t}},
\end{align}
where equality \eqref{eq1:24.10.23} holds by \eqref{eq:u_vec} and \eqref{eq0:24.10.23};
\eqref{eq2:24.10.23} holds by \eqref{eq:alpha_j}, and inequality \eqref{eq3:24.10.23}
holds by \eqref{eq2:LP}. Consequently, it follows from \eqref{eq: Lovasz theta function}
and \eqref{eq3:24.10.23} that
\begin{align}
\label{eq4:24.10.23}
\vartheta(\Gr{G}) \leq \max_{v \in \V{\Gr{G}}} \frac1{\bigl({\bf{u}}(v)^{\mathrm{T}}
\, {\bf{c}}\bigr)^2} \leq t = \fchrnum{\CGr{G}}.
\end{align}
\end{proof}

Two additional spectral lower bounds on the fractional chromatic number of a graph are presented as follows.
These are compared with the lower bound on the right-hand side of equation \eqref{eq1:04.10.23}.
\begin{enumerate}
\item For a graph $\Gr{G}$ on $n$ vertices, let $\lambda_1(\Gr{G})$ and $\lambda_n(\Gr{G})$ be, respectively, the
largest and least eigenvalues of the adjacency matrix of $\Gr{G}$ (see \eqref{eq2:26.09.23}). Then, independently by \cite{Bilu06,Galtman00},
\begin{align}
\label{eq2:04.10.23}
\vchrnum{\Gr{G}} \geq 1 - \frac{\lambda_1(\Gr{G})}{\lambda_n(\Gr{G})}.
\end{align}
Combining \eqref{eq: vchrnum vs. Lovasz theta-function}, \eqref{eq1:04.10.23}, and \eqref{eq2:04.10.23} gives
\begin{align}
\label{eq11:8.11.23}
\fchrnum{\Gr{G}} \geq \vartheta(\CGr{G}) \geq \vchrnum{\Gr{G}} \geq 1 - \frac{\lambda_1(\Gr{G})}{\lambda_n(\Gr{G})}.
\end{align}
This demonstrates that the bound expressed on the right-hand side of \eqref{eq2:04.10.23} gives
a looser lower bound on the fractional chromatic number in comparison to the bound on the right-hand
side of \eqref{eq1:04.10.23}. Alternatively, this conclusion can also be reached by applying
\eqref{eq3:optimization} to the complement graph $\CGr{G}$, which gives
\begin{align}
\label{eq5:optimization}
\vartheta(\CGr{G}) = 1 + \max_{{\bf{T}}} \frac{\lambda_{\max}({\bf{T}})}{\bigl|\lambda_{\min}({\bf{T}})\bigr|},
\end{align}
where ${\bf{T}} = (T_{i,j})$ ranges over all the symmetric nonzero $n \times n$ matrices with $T_{i,j}=0$
for all $\{i,j\} \not\in \E{\Gr{G}}$, and also for $i=j$. Selecting a potentially suboptimal choice for
${\bf{T}}$ on the right-hand side of \eqref{eq5:optimization}, namely, the adjacency matrix of $\Gr{G}$, yields
\begin{align}
\label{eq10a:24.10.23}
\vartheta(\CGr{G}) \geq 1 - \frac{\lambda_1(\Gr{G})}{\lambda_n(\Gr{G})}.
\end{align}
\item A recent lower bound by Guo and Sapiro (Theorem 1.1 in \cite{GuoS22}) states that
\begin{align}
\label{eq3:04.10.23}
\fchrnum{\Gr{G}} \geq 1 + \max \Biggl\{ \frac{s^{+}(\Gr{G})}{s^{-}(\Gr{G})},
\, \frac{s^{-}(\Gr{G})}{s^{+}(\Gr{G})} \Biggr\},
\end{align}
where, on the right-hand side of \eqref{eq3:04.10.23}, $s^{+}(\Gr{G})$ and $s^{-}(\Gr{G})$ respectively
denote the sum of squares of the positive and negative eigenvalues of the adjacency matrix $\A$ of $\Gr{G}$;
these sums are referred to as the positive and negative square energies of the graph $\Gr{G}$.
The interested reader is directed to recent studies in \cite{AbiadLDHM23,ElphickL23}, which explore
the symmetry or asymmetry between the positive and negative square energies of structured
and random graphs. These studies also discuss open questions related to this topic.

The lower bound on the fractional chromatic number presented on the right-hand side of
\eqref{eq3:04.10.23} strengthens Theorem~2.3 from \cite{AndoL15}, which provides an identical
lower bound on the chromatic number of the graph $\Gr{G}$. It is worth noting that
inequality \eqref{eq3:04.10.23} was first conjectured in \cite{WocjanE13} with a proof of
a loosened lower bound that was smaller by~1 than the right-hand side of \eqref{eq3:04.10.23}.
In continuation to that line of research, Coutino and Spier \cite{CoutinhoS23} recently proved
that the right-hand side of \eqref{eq3:04.10.23} is even a lower bound on the vector
 chromatic number of the graph $\Gr{G}$, i.e.,
\begin{align}
\label{eq12:8.11.23}
\fchrnum{\Gr{G}} \geq \vartheta(\CGr{G}) \geq \vchrnum{\Gr{G}} \geq
1 + \max \Biggl\{ \frac{s^{+}(\Gr{G})}{s^{-}(\Gr{G})}, \, \frac{s^{-}(\Gr{G})}{s^{+}(\Gr{G})} \Biggr\}.
\end{align}
The chain of inequalities in \eqref{eq12:8.11.23} show that the right-hand side of \eqref{eq3:04.10.23} is
a looser lower bound on the fractional chromatic number in comparison to the right-hand side of \eqref{eq1:04.10.23}.
\end{enumerate}
The spectral lower bounds on the fractional chromatic number of a graph $\Gr{G}$, as presented
on the right-hand sides of \eqref{eq2:04.10.23} and \eqref{eq3:04.10.23}, are looser than the bound on the
right-hand side of \eqref{eq1:04.10.23}. The latter is equal to the Lov\'{a}sz $\vartheta$-function of the
graph complement $\CGr{G}$.

We next refer to the so called {\em inertial} lower bounds on the chromatic and fractional chromatic numbers of a graph.
Let $\Gr{G}$ be a graph on $n$ vertices, and let $\A$ be its adjacency matrix. The {\em inertia} of $\A$ is defined as
the ordered triple $(n^+, n^0, n^-)$ where $n^+$, $n^0$, and $n^-$, respectively, denote the numbers (counting
multiplicities) of the positive, zero, and negative eigenvalues of $\A$ (so, $n = n^+ + n^0 + n^-$). Note that
$\mathrm{rank}(\A) = n^{+} + n^{-}$ and $\mathrm{nullity}(\A) = n^0$. Theorem~1 of \cite{ElphickW17} states that
\begin{align}
\label{eq: ElphickW17 - Theorem1}
\chrnum{\Gr{G}} \geq 1 + \max \Biggl\{ \frac{n^{+}}{n^{-}}, \, \frac{n^{-}}{n^{+}} \Biggr\}.
\end{align}
A graph $\Gr{G}$ is called {\em nonsingular} if its adjacency matrix $\A$ is nonsingular, which holds if
and only if $n^0 = 0$. It is further proved in Section~3 of \cite{ElphickW17} that if $\Gr{G}$ is a nonsingular
graph, then the right-hand side of \eqref{eq: ElphickW17 - Theorem1} is even a lower bound on its fractional
chromatic number, i.e.,
\begin{align}
\label{eq: ElphickW17 - Section3}
\fchrnum{\Gr{G}} \geq 1 + \max \Biggl\{ \frac{n^{+}}{n^{-}}, \, \frac{n^{-}}{n^{+}} \Biggr\}.
\end{align}
Additionally, \eqref{eq: ElphickW17 - Section3} holds with equality for all Kneser graphs $\KG{r}{k}$
with $r \geq 2k$ (i.e., for all nonempty Kneser graphs) and for all cycle graphs.

In light of Theorem~\ref{theorem: eigenvalues of srg}, the lower bound in \eqref{eq: ElphickW17 - Section3}
holds in particular for all strongly regular graphs (since their adjacency matrices have no zero eigenvalues).
Then, for strongly regular graphs, \eqref{eq: ElphickW17 - Section3} gets the form
\begin{align}
\label{eq2: ElphickW17 - Section3}
\fchrnum{\Gr{G}} \geq 1 + \max \Biggl\{ \frac{m_1+1}{m_2}, \, \frac{m_2}{m_1+1} \Biggr\},
\end{align}
where $m_1$ and $m_2$ on the right-hand side of \eqref{eq2: ElphickW17 - Section3}, respectively, denote
the multiplicities of the second-largest and least eigenvalues of the adjacency matrix of a strongly
regular graph $\Gr{G}$. The expressions for $m_{1,2}$ are given in \eqref{eig-multiplicities-SRG}.
It is further commented in \cite{ElphickW17} that the lower bound \eqref{eq: ElphickW17 - Section3} on the
fractional chromatic number of a nonsingular graph $\Gr{G}$ cannot serve as a lower bound on the vector chromatic
number of $\Gr{G}$.

\subsubsection{Lower bounds on the independence and clique numbers of a graph}
\label{subsubsection: lower bounds on the independence and clique numbers of a graph - background}
Lower bounds on the independence and clique numbers of a graph were established using the vertex degrees by Wei \cite{Wei81}.
Let $\Gr{G}$ be a graph on $n$ vertices, and let $\{d_i\}_{i=1}^n$ represent the degree sequence of its vertices.
Then, the following (nonspectral) bounds on the independence and clique numbers of $\Gr{G}$, respectively, hold:
\begin{align}
\label{eq: Wei1}
& \indnum{\Gr{G}} \geq \sum_{i=1}^n \frac1{1+d_i}, \\
\label{eq: Wei2}
& \clnum{\Gr{G}} \geq \sum_{i=1}^n \frac1{n-d_i}.
\end{align}
The equivalent bounds \eqref{eq: Wei1} and \eqref{eq: Wei2} are proved by using a probabilistic approach
(see, e.g., page~287 in \cite{AignerZ18}).
The reader is referred to \cite{Griggs83} for a refinement of Wei's bounds in \eqref{eq: Wei1} and
\eqref{eq: Wei2} for connected triangle-free graphs.

Spectral lower bounds on the independence and clique numbers of a graph, originally conjectured in
\cite{EdwardsE83} and proved in \cite{Nikiforov02}, are given by
\begin{align}
\label{eq: Nikiforov1a-2002}
& \indnum{\Gr{G}} \geq \frac{n(n-1)-2m}{n(n-1)-2m-\lambda_1(\CGr{G})^2} \, , \\
\label{eq: Nikiforov1b-2002}
& \clnum{\Gr{G}} \geq \frac{2m}{2m-\lambda_1(\Gr{G})^2} \, ,
\end{align}
where $m \eqdef \card{\E{\Gr{G}}}$ denotes the size of the graph.
The equivalent spectral bounds in \eqref{eq: Nikiforov1a-2002} and \eqref{eq: Nikiforov1b-2002} were derived
in \cite{Nikiforov02} through a combination of the Motzkin-Straus result (Theorem~1 in \cite{MotzkinS65}), an equality
for the largest eigenvalue of the adjacency matrix of the graph $\Gr{G}$, and the
Cauchy-Schwarz inequality. The Motzkin-Starus result states that if $\Gr{G}$ is a graph on $n$ vertices, then
\begin{align}
\label{eq1: Motzkin-Straus}
\max \Biggl\{ \; \sum_{\{i,j\} \in \E{\Gr{G}}} x_i x_j : \sum_{i=1}^n x_i = 1, \; x_i \geq 0, \;
\forall i \in \OneTo{n} \; \Biggr\} = \frac12 \, \Biggl(1 - \frac{1}{\clnum{\Gr{G}}} \Biggr),
\end{align}
where every edge $\{i,j\}$ and $\{j,i\}$ is counted only once on the left-hand side of \eqref{eq1: Motzkin-Straus},
the maximization is attained by selecting a maximum clique $\set{C}$ in $\Gr{G}$, and setting
$x_i \eqdef \frac1{\clnum{\Gr{G}}}$ for all $i \in \OneTo{n}$ such that $v_i \in \set{C}$, and $x_i \eqdef 0$ otherwise.
It is noteworthy that \cite{Wilf86} appears to be the first paper to identify a relationship between the Motzkin-Straus
result and spectral bounds for graphs.

Bounds \eqref{eq: Wei1}--\eqref{eq: Wei2} and
\eqref{eq: Nikiforov1a-2002}--\eqref{eq: Nikiforov1b-2002}, respectively, coincide if $\Gr{G}$
is a regular graph. Indeed, if $\Gr{G}$ is a $d$-regular graph on $n$ vertices, then
$m = \tfrac12 \, nd$, $\lambda_1(\Gr{G})=d$, and $\lambda_1(\CGr{G}) = n-d-1$,
so the right-hand sides of \eqref{eq: Wei1} and \eqref{eq: Nikiforov1a-2002} are equal to
$\frac{n}{1+d}$, and the right-hand sides of \eqref{eq: Wei2} and \eqref{eq: Nikiforov1b-2002}
are equal to $\frac{n}{n-d}$. In light of \eqref{eq:18.04.2024} and the Brooks' theorem, which provides an upper
bound on the chromatic number of a connected graph (see, e.g., Theorem~6.14.2 in \cite{GodsilR}),
a lower bound on the independence number of a connected $d$-regular graph on $n$ vertices can be improved to $\frac{n}{d}$,
unless $\Gr{G}$ forms a complete graph or an odd-cycle graph. Further lower bounds on the clique and independence numbers
of a general simple graph $\Gr{G}$, expressed in terms of the average degree of the vertices in $\Gr{G}$ and the
smallest eigenvalue of the adjacency matrix of $\Gr{G}$ or its complement $\CGr{G}$, were derived in \cite{Nikiforov09}.
The interested reader is also referred to Section~2 in \cite{LuLT07} for the derivation of lower bounds on the
independence and clique numbers of a graph, based on the spectrum of its Laplacian matrix.

\subsection{New Bounds and Exact Results on Graph Invariants via the Lov\'{a}sz $\vartheta$-Function}
\label{subsection: main results - bounds on graph invariants}

The present subsection provides new bounds and exact results on various graph invariants, derived
from the properties of the Lov\'{a}sz $\vartheta$-function. This subsection is organized into four parts:
\begin{itemize}
\item
Section~\ref{subsubsection: bounds on graph invariants based on the sandwich theorem}
discusses strongly regular graphs, offering explicit bounds on several invariants of these graphs;
\item
Section~\ref{subsubsection: exact results for the vchrnum and svchrnum} derives spectral upper and
lower bounds on the vector and strict vector chromatic numbers of regular graphs, including
sufficient conditions that facilitate achieving these bounds.
\item
Section~\ref{subsubsection: lower bounds on the fractional chromatic number of a graph}
uses the Lov\'{a}sz $\vartheta$-function of a graph's complement to establish lower bounds on the
fractional chromatic number of a graph, with a focus on perfect and triangle-free graphs.
\item
Section~\ref{subsubsection: lower bounds on the independence and clique numbers of a graph} provides
lower bounds on the independence and clique numbers of a graph, based on the Lov\'{a}sz $\vartheta$-function
of the graph or its complement.
\end{itemize}

\subsubsection{Bounds on graph invariants based on the sandwich theorem for the Lov\'{a}sz $\vartheta$-function}
\label{subsubsection: bounds on graph invariants based on the sandwich theorem}

The sandwich theorem for the Lov\'{a}sz $\vartheta$-function (see \eqref{eq1a: sandwich}
and \eqref{eq1b: sandwich}) enables the calculation of feasible bounds on the fractional
and integral independence, clique, and chromatic numbers of graphs.
Although the exact calculation of these graph invariants is NP-hard, determining their bounds,
based on the numerical solution of the SDP problem outlined in \eqref{eq: SDP problem - Lovasz theta-function},
is feasible with a polynomial complexity in the number of vertices (see Theorem~11.11 in \cite{Lovasz19}).

The following result illustrates the approach by deriving bounds on the considered graph invariants
of strongly regular graphs. The upper bounds on the independence and clique numbers are equivalent
to those derived in Theorem~2.1.5 of \cite{Haemers79_thesis}. However, the derivation method here
is different, relying on the Lov\'{a}sz $\vartheta$-function of strongly regular graphs.
Through this alternative approach, new lower bounds on the fractional graph invariants are also derived.

\begin{corollary}
\label{corollary:bounds on parameters of SRGs}
Let $\Gr{G}$ be a strongly regular graph with parameters $\SRG(n, d, \lambda, \mu)$.
Then,
\begin{align}
\label{eq:01.09.23-srg1}
& \indnum{\Gr{G}} \leq \bigg\lfloor \dfrac{n \, (t+\mu-\lambda)}{2d+t+\mu-\lambda} \bigg\rfloor,
\quad \findnum{\Gr{G}} \geq \dfrac{n \, (t+\mu-\lambda)}{2d+t+\mu-\lambda}, \\[0.1cm]
\label{eq:01.09.23-srg2}
& \clnum{\Gr{G}} \leq 1 + \bigg\lfloor \dfrac{2d}{t+\mu-\lambda} \bigg\rfloor,  \quad \hspace*{0.2cm}
\fclnum{\Gr{G}} \geq 1 + \dfrac{2d}{t+\mu-\lambda}, \\[0.1cm]
\label{eq:01.09.23-srg3}
& \chrnum{\Gr{G}} \geq 1 + \bigg\lceil \dfrac{2d}{t+\mu-\lambda} \bigg\rceil, \quad \hspace*{0.2cm}
\fchrnum{\Gr{G}} \geq 1 + \dfrac{2d}{t+\mu-\lambda},\\[0.1cm]
\label{eq:01.09.23-srg4}
& \chrnum{\CGr{G}} \geq \bigg\lceil \dfrac{n \, (t+\mu-\lambda)}{2d+t+\mu-\lambda} \bigg\rceil, \quad
\fchrnum{\CGr{G}} \geq \dfrac{n \, (t+\mu-\lambda)}{2d+t+\mu-\lambda},
\end{align}
with
\begin{align}
\label{eq:01.09.23-srg5}
t \eqdef \sqrt{(\mu-\lambda)^2 + 4(d-\mu)}.
\end{align}
\end{corollary}
\begin{proof}
Inequalities \eqref{eq:01.09.23-srg1}--\eqref{eq:01.09.23-srg4} follow from \eqref{eq1a: sandwich},
\eqref{eq1b: sandwich}, and Theorem~\ref{thm:Lovasz function of srg}.
\end{proof}

Numerical examinations of the tightness of bounds on graph invariants, as given in Corollary~\ref{corollary:bounds on parameters of SRGs},
are provided in Example~\ref{example: bounds on graph invariants of srgs} for several strongly regular graphs.
The upper bound on the independence number, presented in \eqref{eq:01.09.23-srg1}, is a formulation of the Delsarte
and Hoffman bound, which was originally developed for regular graphs (see Section~3.3 in \cite{Delsarte73} and the
survey in \cite{Haemers21}). This formulation has been specialized in \eqref{eq:01.09.23-srg1} to apply to
strongly regular graphs.
It is noteworthy that this bound, also known as Hoffman's ratio bound, was generalized to simple graphs that are not
necessarily regular (Theorem~2.1.3 of \cite{Haemers79_thesis}), and to graphs that may have self-loops (Corollaries~3.4
and~3.6 of \cite{GodsilN08}). Additionally, other upper bounds on the independence number of a connected simple graph
$\Gr{G}$ are provided in \cite{LiZ14}, expressed in terms of the largest and smallest eigenvalues of the adjacency
matrix of $\Gr{G}$, along with its maximum and minimum degrees.

\subsubsection{Exact results and bounds for the vector and strict vector chromatic numbers of graphs}
\label{subsubsection: exact results for the vchrnum and svchrnum}
Upper and lower bounds on the vector and strict vector chromatic numbers are next derived for regular graphs,
along with sufficient conditions for the attainability of each of these bounds. For strongly regular graphs
or for graphs that are both vertex- and edge-transitive, it leads to exact closed-form expressions of their
vector and strict vector chromatic numbers. These results demonstrate that these two types of chromatic numbers
coincide for such regular graphs.
The subsequent theorem forms a strengthening of Theorem~\ref{thm:bounds on the Lovasz function for regular graphs}
(see Remark~\ref{remark: comparison between two theorems}).
\begin{theorem}
\label{theorem: bounds on vchrnum and svchrnum for regular graphs}
Let $\Gr{G}$ be a graph on $n$ vertices.
\begin{enumerate}
\item \label{item 1: bounds on vchrnum and svchrnum for regular graphs}
The following inequality holds:
\begin{align}
\label{eq9:7.11.23}
\svchrnum{\Gr{G}} \, \svchrnum{\CGr{G}} \geq n,
\end{align}
with equality if $\Gr{G}$ is a vertex-transitive or a strongly regular graph.
\item \label{item 2: bounds on vchrnum and svchrnum for regular graphs}
If $\Gr{G}$ is a $d$-regular graph, then
\begin{align}
\label{eq10:7.11.23}
& 1 - \dfrac{d}{\Eigval{n}{\Gr{G}}} \leq \vchrnum{\Gr{G}} \leq \svchrnum{\Gr{G}}
\leq \dfrac{n \bigl(1+\Eigval{2}{\Gr{G}}\bigr)}{n-d+\Eigval{2}{\Gr{G}}}, \\[0.2cm]
\label{eq11:7.11.23}
& \dfrac{n-d+\Eigval{2}{\Gr{G}}}{1+\Eigval{2}{\Gr{G}}} \leq \vchrnum{\CGr{G}} \leq \svchrnum{\CGr{G}}
\leq -\dfrac{n \Eigval{n}{\Gr{G}}}{d - \Eigval{n}{\Gr{G}}}.
\end{align}
Furthermore,
\begin{itemize}
\item Equality holds in the leftmost inequality of \eqref{eq10:7.11.23}
if $\Gr{G}$ is both vertex-transitive and edge-transitive, or if $\Gr{G}$ is
a strongly regular graph;
\item Equality holds in the rightmost inequality of \eqref{eq10:7.11.23}
if $\CGr{G}$ is edge-transitive, or if $\Gr{G}$ is a strongly regular graph;
\item Equality holds in the leftmost inequality of \eqref{eq11:7.11.23} if $\CGr{G}$
is both vertex-transitive and edge-transitive, or if $\Gr{G}$ is a strongly regular graph;
\item Equality holds in the rightmost inequality of \eqref{eq11:7.11.23} if $\Gr{G}$
is edge-transitive, or if $\Gr{G}$ is a strongly regular graph.
\end{itemize}
\end{enumerate}
\end{theorem}
\begin{proof}
See Section~\ref{subsubsection: proof of bounds on vchrnum and svchrnum for regular graphs}.
\end{proof}
\begin{remark}[Comparison of Item~\ref{item 2: bounds on vchrnum and svchrnum for regular graphs} in
Theorem~\ref{theorem: bounds on vchrnum and svchrnum for regular graphs}
with Theorem~\ref{thm:bounds on the Lovasz function for regular graphs}]
\label{remark: comparison between two theorems}
By equality \eqref{eq: svchrnum vs. Lovasz theta-function}, where $\svchrnum{\Gr{G}} = \vartheta(\CGr{G})$
holds for every graph $\Gr{G}$, the rightmost inequalities in \eqref{eq10:7.11.23} and \eqref{eq11:7.11.23} are, respectively,
equivalent to those in \eqref{eq:21.10.22a2} and \eqref{eq:21.10.22a1}. Furthermore, in light of inequality
\eqref{eq: vchrnum vs. Lovasz theta-function}, which states that $\vchrnum{\Gr{G}} \leq \vartheta(\CGr{G})$ for all $\Gr{G}$,
the leftmost inequalities in \eqref{eq10:7.11.23} and \eqref{eq11:7.11.23} are, respectively, stronger than those
in \eqref{eq:21.10.22a2} and \eqref{eq:21.10.22a1} of Theorem~\ref{thm:bounds on the Lovasz function for regular graphs}.
These observations affirm that Item~\ref{item 2: bounds on vchrnum and svchrnum for regular graphs}
in Theorem~\ref{theorem: bounds on vchrnum and svchrnum for regular graphs} not only corroborates,
but also strengthens the assertions made in Theorem~\ref{thm:bounds on the Lovasz function for regular graphs}.
\end{remark}

\begin{remark}[Relation to spectral lower bounds in \cite{WocjanEA23}]
\label{remark: connection to the paper by Wocjan et al. 2023}
The spectral lower bound on the vector chromatic number of a regular graph $\Gr{G}$, as presented in the leftmost
term of \eqref{eq10:7.11.23}, is due to Galtman \cite{Galtman00}; it is also referred to as Hoffman's bound for the vector chromatic number
\cite{Bilu06,WocjanEA23}. Spectral lower bounds on the vector chromatic number of any connected graph $\Gr{G}$ are introduced
in Theorems~1 and~6 of \cite{WocjanEA23}. The lower bound in Theorem~1 of \cite{WocjanEA23} is given by
\begin{align}
\label{eq1:WocjanEA23}
\vchrnum{\Gr{G}} \geq 1 + \frac{2m}{2m - n \nu_n(\Gr{G})},
\end{align}
where $m \eqdef \card{\E{\Gr{G}}}$ is the number of edges in $\Gr{G}$, and $\nu_n(\Gr{G})$ denotes the least eigenvalue of
the signless Laplacian matrix ${\bf{Q}}$ of the graph $\Gr{G}$ (see \eqref{eq4:26.09.23}). Additionally, the lower bound
in Theorem~6 of \cite{WocjanEA23} is given by
\begin{align}
\label{eq2:WocjanEA23}
\vchrnum{\Gr{G}} \geq 1 + \frac{\lambda_1(\Gr{G})}{\lambda_1(\Gr{G}) - \nu_1(\Gr{G}) + \mu_n(\Gr{G})},
\end{align}
where $\lambda_1(\Gr{G})$, $\mu_n(\Gr{G})$, and $\nu_1(\Gr{G})$ denote, respectively, the largest eigenvalues of the
adjacency matrix $\A$, the Laplacian matrix ${\bf{L}}$, and the signless Laplacian matrix ${\bf{Q}}$ of the graph $\Gr{G}$
(see \eqref{eq2:26.09.23}--\eqref{eq4:26.09.23}).
For regular graphs, it is mentioned in Section~6 of \cite{WocjanEA23} that these lower bounds coincide. Indeed, this
can be verified as follows: for a $d$-regular graph $\Gr{G}$ on $n$ vertices, the equalities
\begin{align}
\label{eq3:WocjanEA23}
2m = nd, \quad \lambda_1(\Gr{G}) = d, \quad \mu_n(\Gr{G}) = d - \lambda_n(\Gr{G}), \quad
\nu_1(\Gr{G}) = 2d, \quad \nu_n(\Gr{G}) = d + \lambda_n(\Gr{G})
\end{align}
hold, so the equalities in \eqref{eq3:WocjanEA23} confirm that the lower bounds on the
left-hand sides of \eqref{eq1:WocjanEA23} and \eqref{eq2:WocjanEA23} coincide.
Furthermore, these lower bounds are identical to the lower bound on $\vchrnum{\Gr{G}}$
as presented in the leftmost term of \eqref{eq10:7.11.23}.
\end{remark}

Based on Theorem~\ref{theorem: bounds on vchrnum and svchrnum for regular graphs}, exact
closed-form expressions for the vector and strict vector chromatic numbers are next provided
for two important subclasses of regular graphs.

\begin{theorem}[Vector and strict vector chromatic numbers of strongly regular graphs]
\label{theorem: vchrnum and svchrnum of srg}
Let $\Gr{G}$ be a strongly regular graph with parameters $\SRG(n, d, \lambda, \mu)$.
Then, the vector and strict vector chromatic numbers of $\Gr{G}$ and its complement
$\CGr{G}$ satisfy
\begin{align}
\label{eq1: vchrnum and svchrnum of srg}
& \vchrnum{\Gr{G}} = 1 + \dfrac{2d}{t+\mu-\lambda} = \svchrnum{\Gr{G}},  \\[0.15cm]
\label{eq2: vchrnum and svchrnum of srg}
& \vchrnum{\CGr{G}} = \dfrac{n \, (t+\mu-\lambda)}{2d+t+\mu-\lambda} = \svchrnum{\CGr{G}},
\end{align}
where $t$ is given in \eqref{eq:01.09.23-srg5}. The vector and strict vector
chromatic numbers of strongly regular graphs are therefore identical. Additionally, the product of
the vector (or strict vector) chromatic numbers of a strongly regular graph and its complement
equals the order of the graph, i.e.,
\begin{align}
\label{eq3: vchrnum and svchrnum of srg}
& \vchrnum{\CGr{G}} = \dfrac{n}{\vchrnum{\Gr{G}}}, \\[0.15cm]
\label{eq4: vchrnum and svchrnum of srg}
&\svchrnum{\CGr{G}} = \dfrac{n}{\svchrnum{\Gr{G}}}.
\end{align}
\end{theorem}
\begin{proof}
See Section~\ref{subsubsection: proof of chrnum and schrnum of srg}.
\end{proof}

\begin{corollary}[Vector and strict vector chromatic numbers of conference graphs]
\label{corollary: vchrnum and svchrnum of conference graphs}
If $\Gr{G}$ is a conference graph on $n$ vertices, then
\begin{align}
\label{eq: vchrnum and svchrnum of conference graphs}
\vchrnum{\CGr{G}} = \vchrnum{\Gr{G}} = \sqrt{n} = \svchrnum{\Gr{G}} = \svchrnum{\CGr{G}}.
\end{align}
\end{corollary}
\begin{proof}
The equalities in \eqref{eq: vchrnum and svchrnum of conference graphs} stem from applying
Theorem~\ref{theorem: vchrnum and svchrnum of srg} with the parameters of a conference
graph. Specifically, a conference graph,
when considered as a strongly regular graph, is characterized by the parameters $d = \frac{1}{2}(n-1)$,
$\lambda = \frac{1}{4}(n-5)$, and $\mu = \frac{1}{4}(n-1)$ (see
Definition \ref{definition: conference graphs}). The extreme equalities in
\eqref{eq: vchrnum and svchrnum of conference graphs} also hold
since $\Gr{G}$ and $\CGr{G}$ are both conference graphs on $n$ vertices.
\end{proof}

\begin{theorem}[Vector/ strict vector chromatic numbers of vertex- and edge-transitive graphs]
\label{theorem: vchrmum and svchrnum of vt+et graphs}
Let $\Gr{G}$ be a graph that is both vertex-transitive and edge-transitive. Then, its vector
and strict vector chromatic numbers coincide, and
\begin{align}
\label{eq1: 07.11.23}
\vchrnum{\Gr{G}} = 1 - \frac{\lambda_1(\Gr{G})}{\lambda_n(\Gr{G})} = \svchrnum{\Gr{G}}.
\end{align}
\end{theorem}
\begin{proof}
See Section~\ref{subsubsection: proof of chrnum and schrnum of vt+et graphs}.
\end{proof}

\begin{remark}[On the subclasses of regular graphs in Theorems~\ref{theorem: vchrnum and svchrnum of srg}
and~\ref{theorem: vchrmum and svchrnum of vt+et graphs}]
Theorems~\ref{theorem: vchrnum and svchrnum of srg} and~\ref{theorem: vchrmum and svchrnum of vt+et graphs}
consider, respectively, the class of strongly regular graphs and the class of graphs that are vertex-transitive
and edge-transitive. These are nonequivalent subclasses of the regular graphs, where also none of them
is contained in the other. Specifically,
the complement of the Cameron graph (see Section~10.54 in \cite{BrouwerM22}) is a strongly regular
graph, $\SRG(231, 200, 172, 180)$, which is vertex-transitive but not edge-transitive. On the other hand,
the Foster graph is 3-regular on 90 vertices (see page 305 of \cite{BrouwerM22}), which is vertex-transitive
and edge-transitive, but it is not strongly regular. Another such example is Holt's graph that is 4-regular on
27~vertices, vertex- and edge-transitive, but not strongly regular \cite{Holt81}. The family of Kneser graphs
$\KG{n}{k}$, where $n \geq 2k$, forms a family of vertex- and edge-transitive graphs that are not necessarily
strongly regular graphs; if, $k=2$ and $n \geq 4$, then it is a strongly regular graph with parameters
$\srg{a_n}{b_n}{c_n}{d_n}$ with $a_n = \binom{n}{2}$, $b_n = \binom{n-2}{2}$, $c_n = \binom{n-4}{2}$,
and $d_n = \binom{n-3}{2}$. It shows that none of these two subclasses of regular graphs in
Theorems~\ref{theorem: vchrnum and svchrnum of srg} and~\ref{theorem: vchrmum and svchrnum of vt+et graphs}
contains the other. To clarify why a graph that is both vertex- and edge-transitive is not necessarily
strongly regular, suppose that $\Gr{G}$ is a vertex- and edge-transitive graph. Then,
\begin{itemize}
\item
($\Gr{G}$ is vertex-transitive) $\Rightarrow$ $\Gr{G}$ is regular;
\item
($\Gr{G}$ is edge-transitive) $\Rightarrow$
(every edge in $\Gr{G}$ is contained in the same number of
triangles) $\Leftrightarrow$ (every pair of adjacent vertices
in $\Gr{G}$ has the same number of common neighbors).
\end{itemize}
Hence, the condition that every pair of nonadjacent vertices has
the same number of common neighbors, which is a requirement for
strongly regular graphs, is not necessarily fulfilled here. To
that end, suppose that also $\CGr{G}$ is edge-transitive. Then,
\begin{itemize}
\item
($\CGr{G}$ is edge transitive) $\Rightarrow$
(for every edge $\{u,v\} \in \E{\CGr{G}}$, the same number of vertices
are not adjacent in $\CGr{G}$ to either $u$ or $v$) $\Leftrightarrow$
(every pair of nonadjacent vertices in $\Gr{G}$ has the
same number of common neighbors).
\end{itemize}
It is important to note that while the closeness property holds for vertex-transitive
graphs under the operation of graph complementation, this property does not extend
to edge-transitive graphs.
\end{remark}

\begin{remark}[The vector and strict vector chromatic numbers of vertex-transitive graphs may be distinct]
\label{remark: vchrmum and svchrnum of vt graphs}
As opposed to graphs that are both vertex- and edge-transitive, whose
vector and strict vector chromatic numbers coincide (by
Theorem~\ref{theorem: vchrmum and svchrnum of vt+et graphs}),
the strict vector chromatic number of a graph that is only vertex-transitive
(but not edge-transitive) may be strictly greater than its vector chromatic
number (see Example~\ref{example: vchrnum and svchrnum, regular graphs - part 2}).
\end{remark}

\subsubsection{Lower bounds on the fractional chromatic number of a graph}
\label{subsubsection: lower bounds on the fractional chromatic number of a graph}
The fractional chromatic number of a graph $\Gr{G}$ is lower-bounded by the Lov\'{a}sz $\vartheta$-function
of its complement $\CGr{G}$, as demonstrated by inequality \eqref{eq1:04.10.23}. Following the derivation
of a closed-form lower bound on the fractional chromatic number of any strongly regular graph (see \eqref{eq:01.09.23-srg3}),
two additional classes of graphs are considered, while linking these results with relevant literature.

{\em On perfect graphs}:
The strong perfect graph theorem provides an exact characterization of perfect graphs (see
Theorem~\ref{theorem: Strong Perfect Graph Theorem}). This characterization, alongside the
sandwich theorem for the Lov\'{a}sz $\vartheta$-function (see \eqref{eq1a: sandwich} and \eqref{eq1b: sandwich}) and
the weak perfect graph theorem (see Theorem~\ref{theorem: Weak Perfect Graph Theorem}),
establishes that for perfect graphs, the computations of both integral and fractional
independence numbers, clique numbers, and chromatic numbers, as well as the Shannon capacity,
vector chromatic number, and strict vector chromatic number coincide with the Lov\'{a}sz $\vartheta$-function
of either the graph itself or its complement. Remarkably, the computational complexity of these graph
invariants, typically NP-hard for general graphs, can be efficiently computed in polynomial time
for perfect graphs by solving the SDP problem outlined in \eqref{eq: SDP problem - Lovasz theta-function} \cite{GrotschelLS84}.

\begin{theorem}[Graph invariants of perfect graphs]
\label{theorem: graph invariants of perfect graphs}
If $\Gr{G}$ is a perfect graph, then
\begin{align}
\label{eq1: perfect graph}
& \clnum{\Gr{G}} = \fclnum{\Gr{G}} = \vchrnum{\Gr{G}} = \svchrnum{\Gr{G}} = \chrnum{\Gr{G}}
= \fchrnum{\Gr{G}} = \Theta(\CGr{G}) = \vartheta(\CGr{G}) = \vartheta'(\CGr{G}), \\
\label{eq2: perfect graph}
& \indnum{\Gr{G}} = \findnum{\Gr{G}} = \vchrnum{\CGr{G}} = \svchrnum{\CGr{G}} = \chrnum{\CGr{G}}
= \fchrnum{\CGr{G}} = \Theta(\Gr{G}) = \vartheta(\Gr{G}) = \vartheta'(\Gr{G}).
\end{align}
\end{theorem}
\begin{proof}
If $\Gr{G}$ is a perfect graph, so is its complement $\CGr{G}$.
Hence, $\clnum{\Gr{G}} = \chrnum{\Gr{G}}$ and $\indnum{\Gr{G}} = \chrnum{\CGr{G}}$. The other
equalities in \eqref{eq1: perfect graph} and \eqref{eq2: perfect graph} hold by the sandwich
theorem in \eqref{eq1a: sandwich} and \eqref{eq1b: sandwich}.
\end{proof}
If $\Gr{G}$ is a perfect graph, then so is every induced subgraph of $\Gr{G}$ or $\CGr{G}$.
It therefore follows that, if $\Gr{G}$ is a perfect graph, then the chains of equalities in \eqref{eq1: perfect graph}
and \eqref{eq2: perfect graph} hold for every induced subgraph of $\Gr{G}$ or $\CGr{G}$.

Theorem~\ref{theorem: graph invariants of perfect graphs} is illustrated in Examples~\ref{example: Hanoi Tower Graphs}
and~\ref{example: windmill graphs}, with respect to three-tower Hanoi graphs and windmill graphs
(these are all perfect graphs), and with comparisons to lower bounds on the fractional chromatic number.
By Theorem~\ref{theorem: graph invariants of perfect graphs}, the lower bound on the right-hand side of
\eqref{eq1:04.10.23} is tight for perfect graphs, whereas the other lower bounds in
\eqref{eq2:04.10.23} and \eqref{eq3:04.10.23} are not necessarily tight for these graphs.

\vspace*{0.2cm}
{\em On triangle-free graphs}: The next result gives a lower bound on the fractional chromatic number
of triangle-free graphs (i.e., graphs that have no cliques of three vertices).
\begin{theorem}
\label{theorem: LB for triangle-free graphs}
{\em Let $\Gr{G}$ be a nonempty triangle-free graph on $n$ vertices.
Then,
\begin{align}
\label{eq: LB for triangle-free graphs}
\fchrnum{\CGr{G}} \geq \vartheta(\Gr{G}) \geq \tfrac{1}{16} \, n^{\frac{2}{3}}.
\end{align}}
\end{theorem}
\begin{proof}
See Section~\ref{subsubsection: proof of LB for triangle-free graphs}.
\end{proof}

\begin{remark}[Relation to \cite{Alon94}]
\label{remark: Relation to N. Alon, 1994}
In Theorem 2.1 of \cite{Alon94}, Alon proved that for every integer represented as $n = 8^k$, where $k \in \naturals$
is not divisible by~3, there exists a triangle-free and $d_n$-regular graph $\Gr{G}_n$ on $n$ vertices, with
$d_n = 2^{k-1} (2^{k-1}-1)$, and
\begin{align}
\label{eq1:1.11.23}
\vartheta(\Gr{G}_n) \leq (36 + o(1)) n^{\frac{2}{3}},
\end{align}
as $o(1) \underset{n \to \infty}{\longrightarrow} 0$. The result in Theorem~\ref{theorem: LB for triangle-free graphs}
shows that, up to a constant factor, $\vartheta(\Gr{G}_n)$ in \eqref{eq1:1.11.23} has the smallest scaling in $n$. A result
of a similar focus, examining the minimum chromatic number of a triangle-free graph as a function of its number
of vertices or edges, is analyzed in \cite{Nilli00}.
\end{remark}

\subsubsection{Lower bounds on the independence and clique numbers of a graph}
\label{subsubsection: lower bounds on the independence and clique numbers of a graph}
Section~\ref{subsubsection: lower bounds on the independence and clique numbers of a graph - background}
introduces known lower bounds on the independence and clique numbers of a graph.
The following bounds are expressed in terms of the Lov\'{a}sz $\vartheta$-function of the graph or its complement.
\begin{theorem}[Lower bounds on the independence and clique numbers of a graph]
\label{thm: lower bound on the independence number}
Let $\Gr{G}$ be a graph on $n$ vertices, excluding the complete graph or its complement. Then,
\begin{align}
\label{eq1:17.10.23}
\indnum{\Gr{G}} \geq \max \left\{ \frac{1}{10 \sqrt{\ln n}} \cdot n^{\, \dfrac{3}{\vartheta(\CGr{G})+1}},
\; \; \frac{2 \ln n}{\ln \Biggl( \dfrac{16n}{\vartheta(\Gr{G})} \Biggr)} - 1 \right\}, \\[0.1cm]
\label{eq1b:17.10.23}
\clnum{\Gr{G}} \geq \max \left\{ \frac{1}{10 \sqrt{\ln n}} \cdot n^{\, \dfrac{3}{\vartheta(\Gr{G})+1}},
\; \; \frac{2 \ln n}{\ln \Biggl( \dfrac{16n}{\vartheta(\CGr{G})} \Biggr)} - 1 \right\}.
\end{align}
\end{theorem}

\begin{proof}
See Section~\ref{subsubsection: proof of lower bounds on the independence number}.
\end{proof}

\subsection{Proofs}
\label{subsection: proofs - bounds on graph invariants}
This section offers proofs of results outlined in
Section~\ref{subsection: main results - bounds on graph invariants}.

\subsubsection{Proof of Theorem~\ref{theorem: bounds on vchrnum and svchrnum for regular graphs}}
\label{subsubsection: proof of bounds on vchrnum and svchrnum for regular graphs}
Let $\Gr{G}$ be a graph on $n$ vertices. Inequality \eqref{eq9:7.11.23} is derived by combining
\eqref{eq: svchrnum vs. Lovasz theta-function} (Theorem~8.2 in \cite{KargerMS98})
and \eqref{eq10:11.10.23} (Corollary~2 in \cite{Lovasz79_IT}). Additionally, if $\Gr{G}$ is
a strongly regular or vertex-transitive graph, then \eqref{eq9:7.11.23} is
assured to hold with equality by combining \eqref{eq: svchrnum vs. Lovasz theta-function}
and Corollary~\ref{corollary:identity for the Lovasz function of srg and its complement}.
This completes the proof of Item~\ref{item 1: bounds on vchrnum and svchrnum for regular graphs}
in Theorem~\ref{theorem: bounds on vchrnum and svchrnum for regular graphs}.

We now turn to prove Item~\ref{item 2: bounds on vchrnum and svchrnum for regular graphs}
in Theorem~\ref{theorem: bounds on vchrnum and svchrnum for regular graphs}.
Let $\Gr{G}$ be a $d$-regular graph on $n$ vertices.
Referring to Remark~\ref{remark: comparison between two theorems} and \eqref{eq: vchrnum vs. Lovasz theta-function},
it suffices to prove the leftmost inequalities in \eqref{eq10:7.11.23} and \eqref{eq11:7.11.23}. This is sufficient
because the middle inequalities in \eqref{eq10:7.11.23} and \eqref{eq11:7.11.23} are \eqref{eq1:05.11.23}, and the rightmost
inequalities in \eqref{eq10:7.11.23} and \eqref{eq11:7.11.23} are, respectively, equivalent to those
in \eqref{eq:21.10.22a2} and \eqref{eq:21.10.22a1} (due to equality \eqref{eq: svchrnum vs. Lovasz theta-function}).
By Galtman's result \cite{Galtman00},
\begin{align}
\label{eq: Galtman}
\vchrnum{\Gr{G}} \geq 1 - \frac{\lambda_1(\Gr{G})}{\lambda_n(\Gr{G})},
\end{align}
which is the leftmost inequality in \eqref{eq10:7.11.23} with $\lambda_1(\Gr{G}) = d$. The leftmost inequality
in \eqref{eq11:7.11.23} is obtained as follows:
\begin{align}
\label{eq2: Galtman}
\vchrnum{\CGr{G}} & \geq 1 - \dfrac{\lambda_1(\CGr{G})}{\lambda_n(\CGr{G})} \\
\label{eq13:7.11.23}
& = 1 + \dfrac{n-d-1}{1+\lambda_2(\Gr{G})} \\[0.1cm]
\label{eq14:7.11.23}
& = \dfrac{n-d+\lambda_2(\Gr{G})}{1+\lambda_2(\Gr{G})},
\end{align}
where \eqref{eq2: Galtman} holds by replacing $\Gr{G}$ with $\CGr{G}$ on both sides of \eqref{eq: Galtman};
\eqref{eq13:7.11.23} is obtained from \eqref{eq: adjaceny matrix of a graph and its complement},
where the latter reveals the following relation for a $d$-regular graph $\Gr{G}$ on $n$ vertices (see Section~1.3.2 in \cite{BrouwerH12}):
\begin{eqnarray}
\label{eq: 21.11.2022a1}
&& \Eigval{1}{\CGr{G}} = n-d-1 = n-1-\Eigval{1}{\Gr{G}}, \\
\label{eq: 21.11.2022a2}
&& \Eigval{\ell}{\CGr{G}} = -1 - \Eigval{n+2-\ell}{\Gr{G}}, \quad \ell = 2, \ldots, n.
\end{eqnarray}
Specifically, by setting $\ell=n$ in \eqref{eq: 21.11.2022a2},
\begin{eqnarray}
\label{eq: 21.11.2022a3}
\Eigval{n}{\CGr{G}} = -1 - \Eigval{2}{\Gr{G}}.
\end{eqnarray}
The leftmost inequalities in \eqref{eq10:7.11.23} and \eqref{eq11:7.11.23} hold by
\eqref{eq: Galtman}--\eqref{eq14:7.11.23}. Subsequently, by \eqref{eq: svchrnum vs. Lovasz theta-function},
the sufficient conditions for equality in the rightmost inequalities of \eqref{eq:21.10.22a2} and
\eqref{eq:21.10.22a1}, respectively, coincide with those required for equality in the rightmost inequalities
of \eqref{eq10:7.11.23} and \eqref{eq11:7.11.23}. Furthermore, by \eqref{eq: svchrnum vs. Lovasz theta-function}
alongside the leftmost and middle inequalities in \eqref{eq10:7.11.23} and \eqref{eq11:7.11.23},
the sufficient conditions for equality in the leftmost inequalities of \eqref{eq:21.10.22a2} and
\eqref{eq:21.10.22a1}, respectively, also ensure equality in the leftmost inequalities
of \eqref{eq10:7.11.23} and \eqref{eq11:7.11.23}.

\subsubsection{Proof of Theorem~\ref{theorem: vchrnum and svchrnum of srg}}
\label{subsubsection: proof of chrnum and schrnum of srg}
Let $\Gr{G}$ be a strongly regular graph with parameters $\srg{n}{d}{\lambda}{\mu}$.
Then, by the satisfiability of the sufficient conditions in
Theorem~\ref{theorem: bounds on vchrnum and svchrnum for regular graphs}
for equality in the leftmost and rightmost inequalities
of \eqref{eq10:7.11.23} and \eqref{eq11:7.11.23}, it follows that
\begin{align}
\label{eq1:8.11.23}
& 1 - \dfrac{d}{\Eigval{n}{\Gr{G}}} = \vchrnum{\Gr{G}} \leq \svchrnum{\Gr{G}}
= \dfrac{n \bigl(1+\Eigval{2}{\Gr{G}}\bigr)}{n-d+\Eigval{2}{\Gr{G}}}, \\[0.2cm]
\label{eq2:8.11.23}
& \dfrac{n-d+\Eigval{2}{\Gr{G}}}{1+\Eigval{2}{\Gr{G}}} = \vchrnum{\CGr{G}} \leq \svchrnum{\CGr{G}}
= -\dfrac{n \Eigval{n}{\Gr{G}}}{d - \Eigval{n}{\Gr{G}}}.
\end{align}
We proceed to prove that equality is also attained in the middle inequalities of \eqref{eq1:8.11.23} and \eqref{eq2:8.11.23}.
By Theorem~\ref{theorem: eigenvalues of srg} and the parameter $t$ as defined in \eqref{eq: t parameter - srg},
\begin{align}
\label{eq3:8.11.23}
\Eigval{2}{\Gr{G}} = \tfrac12 (\lambda-\mu+t), \quad \Eigval{n}{\Gr{G}} = \tfrac12 (\lambda-\mu-t),
\end{align}
which implies that the leftmost and rightmost terms in \eqref{eq1:8.11.23} coincide, and
\begin{align}
\label{eq4:8.11.23}
1 - \dfrac{d}{\Eigval{n}{\Gr{G}}} & = 1 + \dfrac{2d}{\mu+t-\lambda} \\[0.1cm]
\label{eq5:8.11.23}
&= \dfrac{n \bigl(1+\Eigval{2}{\Gr{G}}\bigr)}{n-d+\Eigval{2}{\Gr{G}}}.
\end{align}
Likewise, by taking reciprocals in \eqref{eq4:8.11.23}--\eqref{eq5:8.11.23} and multiplying by $n$, the leftmost
and rightmost terms in \eqref{eq2:8.11.23} also coincide, and
\begin{align}
\label{eq6:8.11.23}
\dfrac{n-d+\Eigval{2}{\Gr{G}}}{1+\Eigval{2}{\Gr{G}}} & = \dfrac{n(t+\mu-\lambda)}{2d+t+\mu-\lambda} \\[0.1cm]
\label{eq7:8.11.23}
&= -\dfrac{n \Eigval{n}{\Gr{G}}}{d - \Eigval{n}{\Gr{G}}}.
\end{align}
In light of \eqref{eq4:8.11.23}--\eqref{eq7:8.11.23}, the middle inequalities
in \eqref{eq10:7.11.23} and \eqref{eq11:7.11.23} hold with equalities,
establishing all the equalities in \eqref{eq1: vchrnum and svchrnum of srg}--\eqref{eq4: vchrnum and svchrnum of srg}.

\subsubsection{Proof of Theorem~\ref{theorem: vchrmum and svchrnum of vt+et graphs}}
\label{subsubsection: proof of chrnum and schrnum of vt+et graphs}
Let $\Gr{G}$ be a graph on $n$ vertices that is both vertex-transitive and edge-transitive.
Due to its vertex-transitivity, $\Gr{G}$ is regular, with $d$ denoting the degree of its vertices.
By Item~\ref{item 2: bounds on vchrnum and svchrnum for regular graphs} in
Theorem~\ref{theorem: bounds on vchrnum and svchrnum for regular graphs},
equality holds in the leftmost inequality of \eqref{eq10:7.11.23}
and the rightmost inequality of \eqref{eq11:7.11.23}, which gives
\begin{align}
\label{eq8:8.11.23}
& \vchrnum{\Gr{G}} = 1 - \dfrac{d}{\Eigval{n}{\Gr{G}}},  \\[0.1cm]
\label{eq9:8.11.23}
& \svchrnum{\CGr{G}} = -\dfrac{n \Eigval{n}{\Gr{G}}}{d - \Eigval{n}{\Gr{G}}}.
\end{align}
By Item~\ref{item 1: bounds on vchrnum and svchrnum for regular graphs} of
Theorem~\ref{theorem: bounds on vchrnum and svchrnum for regular graphs},
the vertex-transitivity of $\Gr{G}$ yields equality in \eqref{eq9:7.11.23}, so
\begin{align}
\label{eq10:8.11.23}
\svchrnum{\Gr{G}} = \frac{n}{\svchrnum{\CGr{G}}} = 1 - \dfrac{d}{\Eigval{n}{\Gr{G}}}.
\end{align}
Combining \eqref{eq8:8.11.23}, \eqref{eq10:8.11.23}, and the equality
$\Eigval{1}{\Gr{G}}=d$ (for a $d$-regular graph $\Gr{G}$) gives \eqref{eq1: 07.11.23}.

\subsubsection{Proof of Theorem~\ref{theorem: LB for triangle-free graphs}}
\label{subsubsection: proof of LB for triangle-free graphs}
Let $\Gr{G}$ be a nonempty triangle-free graph on $n$ vertices, which implies that $\indnum{\CGr{G}} = \clnum{\Gr{G}} = 2$.
By \eqref{eq12:11.10.23} (Theorem~11.18 in \cite{Lovasz19}), with $\Gr{G}$ replaced by $\CGr{G}$ and setting $k=2$,
it follows that
\begin{align}
\label{eq10:01.10.23}
\vartheta(\CGr{G}) \leq 16 n^{\frac{1}{3}}.
\end{align}
Since $\vartheta(\Gr{G}) \, \vartheta(\CGr{G}) \geq n$ holds for every graph $\Gr{G}$ on $n$ vertices
(see \eqref{eq10:11.10.23}), it follows from \eqref{eq10:01.10.23} that
\begin{align}
\label{eq11:01.10.23}
\vartheta(\Gr{G}) \geq \tfrac{1}{16} \, n^{\frac{2}{3}}.
\end{align}
The result in \eqref{eq: LB for triangle-free graphs} is deduced by combining \eqref{eq11:01.10.23}
with the inequality $\fchrnum{\CGr{G}} \geq \vartheta(\Gr{G})$ (see \eqref{eq1a: sandwich}).

\subsubsection{Proof of Theorem~\ref{thm: lower bound on the independence number}}
\label{subsubsection: proof of lower bounds on the independence number}
Inequalities \eqref{eq1:17.10.23} and \eqref{eq1b:17.10.23} are equivalent since
$\indnum{\Gr{G}} = \clnum{\CGr{G}}$, so the right-hand side of \eqref{eq1b:17.10.23}
is obtained from the right-hand side of \eqref{eq1:17.10.23} by replacing $\Gr{G}$ with $\CGr{G}$.
For a graph $\Gr{G}$ on $n$ vertices, inequality \eqref{eq1:17.10.23} is composed of two
lower bounds on its independence number. The first term on the right-hand side of \eqref{eq1:17.10.23}
is the lower bound on the right-hand of \eqref{eq0:17.10.23} with $t \eqdef \vartheta(\CGr{G})$
(by assumption, since neither $\Gr{G}$ nor its complement $\CGr{G}$ are complete graphs, it follows that
$\vartheta(\Gr{G}) \geq \indnum{\Gr{G}} \geq 2$ and $\vartheta(\CGr{G}) \geq \clnum{\Gr{G}} \geq 2$).
The second term on the right-hand side of \eqref{eq1:17.10.23} is an equivalent form of
inequality \eqref{eq12:11.10.23}, where the independence number $\indnum{\Gr{G}}$ is lower-bounded
by a function of $\vartheta(\Gr{G})$.

\subsection{Examples for the Results in Section~\ref{subsection: main results - bounds on graph invariants}}
\label{subsection: examples - bounds on graph invariants}

The present subsection provides the utility of the results in Section~\ref{subsection: main results - bounds on graph invariants},
and it also obtains new results by examples and counterexamples.

\begin{example}[Bounds on graph invariants for strongly regular graphs]
\label{example: bounds on graph invariants of srgs}
The present example aims to illustrate the tightness of the bounds in Corollary~\ref{corollary:bounds on parameters of SRGs}
for four selected strongly regular graphs:
\begin{itemize}
\item Petersen graph that is a strongly regular graph with parameters $\srg{10}{3}{0}{1}$;
\item Schl\"{a}fli graph that is a strongly regular graph with parameters $\srg{27}{16}{10}{8}$;
\item Shrikhande graph that is a strongly regular graph with parameters $\srg{16}{6}{2}{2}$;
\item Hall-Janko graph that is a strongly regular graph with parameters $\srg{100}{36}{14}{12}$.
\end{itemize}
\begin{enumerate}
\item Let $\Gr{G}_1$ be the Petersen graph. Then, $\vartheta(\Gr{G}_1) = 4$, and $\vartheta(\CGr{G}_1) = \tfrac{5}{2}$.
The bounds on the independence, clique, and chromatic numbers are tight:
\begin{align}
\label{eq2:25.10.23}
& \indnum{\Gr{G}_1} = 4 = \vartheta(\Gr{G}_1), \\
\label{eq2a:25.10.23}
&\clnum{\Gr{G}_1} = 2 = \lfloor \vartheta(\CGr{G}_1) \rfloor, \\
\label{eq2b:25.10.23}
& \fchrnum{\Gr{G}_1} = \tfrac{5}{2} = \vartheta(\CGr{G}_1), \\
\label{eq2c:25.10.23}
& \chrnum{\Gr{G}_1} = 3 = \lceil \vartheta(\CGr{G}_1) \rceil.
\end{align}
\item Let $\Gr{G}_2, \Gr{G}_3, \Gr{G}_4$ be the Schl\"{a}fli, Shrikhande, and Hall-Janko graphs, respectively.
Then,
\begin{align}
\label{eq3:25.10.23}
& \vartheta(\Gr{G}_2) = 3, \; \; \vartheta(\CGr{G}_2) = 9, \\
\label{eq3a:25.10.23}
& \vartheta(\Gr{G}_3) = \vartheta(\CGr{G}_3) = 4, \\
\label{eq3b:25.10.23}
& \vartheta(\Gr{G}_4) = \vartheta(\CGr{G}_4) = 10.
\end{align}
\item
The lower bounds on the chromatic and fractional chromatic numbers of these graphs are tight:
\begin{align}
\label{eq4:25.10.23}
& \fchrnum{\Gr{G}_2} = \chrnum{\Gr{G}_2} = 9, \\
\label{eq4a:25.10.23}
& \fchrnum{\Gr{G}_3} = \chrnum{\Gr{G}_3} = 4, \\
\label{eq4b:25.10.23}
& \fchrnum{\Gr{G}_4} = \chrnum{\Gr{G}_4} = 10.
\end{align}
\item
The upper bounds on their independence numbers are also tight:
\begin{align}
\label{eq5:25.10.23}
& \indnum{\Gr{G}_2} = 3 = \vartheta(\Gr{G}_2), \\
\label{eq5a:25.10.23}
& \indnum{\Gr{G}_3} = 4 = \vartheta(\Gr{G}_3), \\
\label{eq5b:25.10.23}
& \indnum{\Gr{G}_4} = 10 = \vartheta(\Gr{G}_4).
\end{align}
\item
The upper bounds on their clique numbers are, however, not tight since
\begin{align}
\label{eq6:25.10.23}
& \clnum{\Gr{G}_2} = 6 < \vartheta(\CGr{G}_2), \\
\label{eq6a:25.10.23}
& \clnum{\Gr{G}_3} = 3 < \vartheta(\CGr{G}_3), \\
\label{eq6b:25.10.23}
& \clnum{\Gr{G}_4} = 4 < \vartheta(\CGr{G}_4).
\end{align}
\item
For comparison, the inertial lower bound \eqref{eq2: ElphickW17 - Section3} gives
\begin{align}
\label{eq1:16.11.23}
\fchrnum{\Gr{G}_1} \geq 2 \tfrac{1}{2}, \quad \fchrnum{\Gr{G}_2} \geq 3 \tfrac{6}{7},
\quad \fchrnum{\Gr{G}_3} \geq 2 \tfrac{2}{7}, \quad \fchrnum{\Gr{G}_4} \geq 2 \tfrac{26}{37}.
\end{align}
With the exception of the Petersen graph ($\Gr{G}_1$), it is observed that the lower bounds in \eqref{eq1:16.11.23} are not tight.
In contrast, the bound provided in Corollary~\ref{corollary:bounds on parameters of SRGs} proves to be tight for the fractional
chromatic numbers of all four graphs.
\end{enumerate}
\end{example}

\begin{table}[h!b!]
\caption{\label{table: vchrmum and svchrnum of regular graphs} \centering{Regular graphs with identical vector and strict vector chromatic numbers
by Theorems~\ref{theorem: vchrnum and svchrnum of srg} and~\ref{theorem: vchrmum and svchrnum of vt+et graphs} (see
Example~\ref{example: vchrnum and svchrnum, regular graphs}).}}
\renewcommand{\arraystretch}{1.5}
\centering
\vspace*{0.1cm}
\begin{tabular}{||c|c|c||}
\hline
\text{Named Graph} $\Gr{G}$ & \text{Type of Graph} & $\Chromatic_{\mathrm{v,sv}}(\Gr{G})$ \\[0.05cm] \hline
Pentagon $(\CG{5})$ & $\srg{5}{2}{0}{1}$ & $\sqrt{5}$ \\
Petersen & $\srg{10}{3}{0}{1}$ & $2 \tfrac{1}{2}$ \\
Clebsch & $\srg{16}{5}{0}{2}$ & $2 \tfrac{1}{3}$ \\
Shrikhande & $\srg{16}{6}{2}{2}$  & $4$ \\
Schl\"{a}fli  & $\srg{27}{16}{10}{8}$ & $9$ \\
Holt (a.k.a. Doyle) & 27 vertices, vertex- and edge-transitive, not srg & $2.53747 \ldots$ \\
Hoffman-Singleton & $\srg{50}{7}{0}{1}$ & $3 \tfrac{1}{3}$ \\
Sims-Gewirtz & $\srg{56}{10}{0}{2}$ & $3 \tfrac{1}{2}$ \\
Gritsenko & $\srg{65}{32}{15}{16}$  & $\sqrt{65}$ \\
Mesner ($M_{22}$) & $\srg{77}{16}{0}{4}$ & $3 \tfrac{2}{3}$ \\
Brouwer-Haemers & $ \srg{81}{20}{1}{6}$ & $3 \tfrac{6}{7}$ \\
Foster & 90 vertices, vertex- and edge-transitive, not srg & $2$  \\
Higman-Sims & $\srg{100}{22}{0}{6}$ & $3 \tfrac{3}{4}$ \\
Hall-Janko & $\srg{100}{36}{14}{12}$ & $10$ \\
Dejter & 112 vertices, vertex- and edge-transitive, bipartite & 2 \\
Cameron & $\srg{231}{30}{9}{3}$  & $11$  \\
Mathon-Rosa & $\srg{280}{117}{44}{52}$ & $10$ \\
Janko-Kharaghani-Tonchev & $\srg{324}{153}{72}{72}$ & $18$ \\
Even-cycle graph $\CG{n}, n \geq 4$ & vertex- and edge-transitive, not srg if $n \neq 4$ & 2 \\
Odd-cycle graph $\CG{n}, n \geq 3$ & vertex- and edge-transitive, not srg if $n \neq 5$ & $1 + \sec\bigl(\frac{\pi}{n}\bigr)$ \\
Kneser graph $\KG{n}{k}, n \geq 2k$ & vertex- and edge-transitive, not srg in general & $\dfrac{n}{k}$ \\
Paley graph of order $n$ & conference graph, vertex- and edge-transitive & $\sqrt{n}$ \\[0.05cm]
\hline
\end{tabular}
\end{table}

\begin{example}[Identical vector and strict vector chromatic numbers in
Theorems~\ref{theorem: vchrnum and svchrnum of srg} and~\ref{theorem: vchrmum and svchrnum of vt+et graphs}]
\label{example: vchrnum and svchrnum, regular graphs}
By Theorems~\ref{theorem: vchrnum and svchrnum of srg} and~\ref{theorem: vchrmum and svchrnum of vt+et graphs},
the vector and strict vector chromatic numbers coincide for every graph that is either strongly regular
or both vertex- and edge-transitive. Theorems~\ref{theorem: vchrnum and svchrnum of srg}
and~\ref{theorem: vchrmum and svchrnum of vt+et graphs} are confirmed numerically by solving the SDP problems in
\eqref{eq: dual SDP problem - vector chromatic number} and \eqref{eq: dual SDP problem - strict vector chromatic number}
for such regular graphs, while getting a match with \eqref{eq1: vchrnum and svchrnum of srg}, \eqref{eq2: vchrnum and svchrnum of srg},
and \eqref{eq1: 07.11.23} (see Table~\ref{table: vchrmum and svchrnum of regular graphs}).
Numerical results in Table~\ref{table: vchrmum and svchrnum of regular graphs} also exemplify
that the vector and strict vector chromatic numbers may coincide for graphs that do not exhibit
either strong regularity or both vertex- and edge-transitivity. That happens to be the case, e.g., for the Frucht graph,
whose vector and strict vector chromatic numbers are identical and equal to~3. The Frucht graph is
a 3-regular graph on 12~vertices that is neither vertex- nor edge-transitive, and is also
not a strongly regular graph.
\end{example}

\begin{example}[The vector and strict vector chromatic numbers of vertex-transitive graphs may be distinct]
\label{example: vchrnum and svchrnum, regular graphs - part 2}
In \cite{Schrijver79} (see page~429), there is an example of a regular graph whose strict vector
chromatic number is strictly greater than its vector chromatic number. Since only final results are
given in \cite{Schrijver79}, we first elaborate on that example from \cite{Schrijver79}, and we then
generalize it to examine the vector and strict vector chromatic numbers of these generalized graphs.

\begin{table}[h!t!]
\caption{\label{table: generalized Schrijver's example} \centering{The vector and strict vector chromatic numbers of the vertex-transitive
graph $\Gr{H} \eqdef \Gr{H}_{\ell, d_{\mathrm{L}}, d_{\mathrm{H}}}$ in Example~\ref{example: vchrnum and svchrnum, regular graphs - part 2},
for different selections of integers $\ell, d_{\mathrm{L}}, d_{\mathrm{H}} \in \naturals$ such that $1 \leq d_{\mathrm{L}} \leq d_{\mathrm{H}} \leq \ell$.
The values of these two types of chromatic numbers are bolded if $\vchrnum{\Gr{H}} < \svchrnum{\Gr{H}}$. In light of
Theorem~\ref{theorem: vchrmum and svchrnum of vt+et graphs}, it is specified if $\Gr{H}$ is also edge-transitive.}}
\renewcommand{\arraystretch}{1.5}
\vspace*{0.1cm}
\centering
\begin{tabular}{||c|c|c||c|c||c||}
\hline
$\ell$ & $d_{\mathrm{L}}$ & $d_{\mathrm{H}}$ & $\vchrnum{\Gr{H}}$ & $\svchrnum{\Gr{H}}$ & \text{Is} $\Gr{H}$ \text{edge-} \\
& & & & & \text{transitive ?} \\[0.05cm] \hline
$4$ & $2$ & $2$ & $4$ & $4$ & yes\\
$4$ & $2$ & $3$ & $6$ & $6$ & yes\\
$4$ & $2$ & $4$ & $8$ & $8$ & no \\
$4$ & $3$ & $4$ & $2 \tfrac{2}{3}$ & $2 \tfrac{2}{3}$ & yes\\
$5$ & $2$ & $5$ & $16$ & $16$ & no\\
$5$ & $3$ & $4$ & $4$ & $4$ & no \\
$5$ & $3$ & $5$ & ${\bf{4}}$ & ${\bf{5 \tfrac{1}{3}}}$ & no \\
$5$ & $4$ & $5$ & $2 \tfrac{2}{3}$ & $2 \tfrac{2}{3}$ & no \\
$6$ & $2$ & $5$ & $26 \tfrac{2}{3}$ & $26 \tfrac{2}{3}$ & no \\
$6$ & $2$ & $6$ & $32$ & $32$ & no \\
$6$ & $3$ & $6$ & $8$ & $8$ & no \\
$6$ & $4$ & $6$ & ${\bf{4}}$ & ${\bf{5 \tfrac{1}{3}}}$ & no \\
$6$ & $5$ & $6$ & $2 \tfrac{2}{5}$ & $2 \tfrac{2}{5}$  & yes \\
$7$ & $3$ & $7$ & $16$ & $16$ & no \\
$7$ & $4$ & $7$ & $8$ & $8$ & no \\
$7$ & $5$ & $7$ & ${\bf{3}}$ & ${\bf{3 \tfrac{5}{9}}}$ & no \\
$7$ & $6$ & $7$ & $2 \tfrac{2}{5}$ & $2 \tfrac{2}{5}$ & no \\[0.05cm]
\hline
\end{tabular}
\end{table}

Let $\Gr{G}$ be a graph on~64 vertices, and label its vertices by the 6-length binary strings $\{0,1\}^6$.
Let any two vertices be adjacent if and only if their corresponding strings are of Hamming distance that
is at most~3. Let $\Gr{H} \eqdef \CGr{G}$ denote the complement graph, so any two of its vertices are adjacent if their corresponding
strings are of Hamming distance of at least~4. Replacing the binary strings with $\pm 1$ vectors in the standard way
($0 \mapsto +1$ and $1 \mapsto -1$) and then normalizing these vectors to be unit vectors in the Euclidean space, the adjacency of vertices in $\Gr{H}$
corresponds to two vectors having an inner product that is less than or equal to $-\tfrac13$.
By Definition~\ref{definition: vector chromatic number}, it follows that $\vchrnum{\Gr{H}} \leq 4$. It can be
verified that there exists a clique in $\Gr{H}$ of 4~vertices, so $\clnum{\Gr{H}} \geq 4$; indeed, the vertices
that refer to the vectors $(+1,+1,+1,+1,+1,+1)$, $(-1,-1,-1,-1,+1,+1)$, $(-1,-1,+1,+1,-1,-1)$, and $(+1,+1,-1,-1,-1,-1)$
form a clique in $\Gr{H}$ since any pair of vectors differ in four coordinates. Hence,
$\clnum{\Gr{H}} = 4 = \vchrnum{\Gr{H}}$. Numerical calculation of the strict vector chromatic number of
the graph $\Gr{H}$ as a solution of the SDP problem in \eqref{eq: dual SDP problem - strict vector chromatic number}, which
relies on first constructing the adjacency matrix of the graph $\Gr{H}$ by its above definition, gives that
$\svchrnum{\Gr{H}} = 5 \tfrac{1}{3}$; it is consistent with Schrijver's paper \cite{Schrijver79}
where it was claimed that $\vartheta(\Gr{G}) = 5 \tfrac{1}{3}$ (by definition, $\Gr{H} = \CGr{G}$, so
$\svchrnum{\Gr{H}} = \vartheta(\Gr{G})$). To conclude, it is obtained that
\begin{align}
\label{eq3:17.11.23}
\svchrnum{\Gr{H}} = 5 \tfrac{1}{3} > 4 = \vchrnum{\Gr{H}}.
\end{align}
The regular graph $\Gr{H}$ exhibits vertex-transitivity due to a property in which adding a constant vector to all vertices,
using arithmetic over $\mathrm{GF}(2)$, maintains the Hamming distance between any pair of vertices. This implies that the
adjacency relations among the vertices in $\Gr{H}$ remain consistent under such transformations. Consequently, any vector
$u \in \{0,1\}^6$ can be mapped to any other vector $v \in \{0,1\}^6$ by adding the fixed vector $v - u \in \{0,1\}^6$ to all
vertices in the graph. However, $\Gr{H}$ does not possess edge-transitivity. If it did, according to
Theorem~\ref{theorem: vchrmum and svchrnum of vt+et graphs}, the vector and strict vector chromatic numbers of $\Gr{H}$
would be identical. The SageMath software tool \cite{SageMath} has confirmed that $\Gr{H}$ is not an edge-transitive graph.

As an extension of the above construction of $\Gr{H}$, let $\ell, d_{\mathrm{L}}, d_{\mathrm{H}}$
be integers such that $1 \leq d_{\mathrm{L}} \leq d_{\mathrm{H}} \leq \ell$. Let $\Gr{H}$ be a graph on $n = 2^\ell$ vertices
whose vertices are labeled by all the $\ell$-length binary strings $\{0,1\}^\ell$, and let any pair of vertices in $\Gr{H}$
be adjacent if and only if their corresponding strings are of Hamming distance $d$ that satisfies
$d_{\mathrm{L}} \leq d \leq d_{\mathrm{H}}$. As it is explained above, the regular graph $\Gr{H}$ is vertex-transitive.
The above graph $\Gr{H}$ from \cite{Schrijver79} corresponds to the special case where $d_{\mathrm{L}} = 4, d_{\mathrm{H}} = 6$,
and $\ell=6$. We next compare the vector and strict vector
chromatic numbers of the graph $\Gr{H} \eqdef \Gr{H}_{\ell, d_{\mathrm{L}}, d_{\mathrm{H}}}$ by numerically solving the SDP problems
in \eqref{eq: SDP problem - vector chromatic number} and \eqref{eq: dual SDP problem - strict vector chromatic number} for different
selections of the parameters $(\ell, d_{\mathrm{L}}, d_{\mathrm{H}})$, showing the existence of other possible combinations of these
graph parameters, for which the strict vector chromatic number of the vertex-transitive graph $\Gr{H}$ is strictly greater than its vector chromatic number.
The numerical results of the vector and strict vector chromatic numbers of the graph $\Gr{H} \eqdef \Gr{H}_{\ell, d_{\mathrm{L}}, d_{\mathrm{H}}}$
are presented in Table~\ref{table: generalized Schrijver's example} for different values of parameters $(\ell, d_{\mathrm{L}}, d_{\mathrm{H}})$.
It is shown that, for such vertex-transitive graphs, these two types of chromatic numbers may be distinct, in contrast to graphs that are
both vertex- and edge-transitive whose vector and strict vector chromatic numbers coincide by Theorem~\ref{theorem: vchrmum and svchrnum of vt+et graphs}.
\end{example}

\begin{example}[The variant of the $\vartheta$-function by Schrijver is not an upper bound on the Shannon capacity]
\label{example: counterexample}
By the sandwich theorem in \eqref{eq1a: sandwich}, the inequality $\indnum{\Gr{G}} \leq \vartheta'(\Gr{G}) \leq \vartheta(\Gr{G})$
holds for every finite, undirected, and simple graph $\Gr{G}$.
It is next shown by a counterexample that, unlike the Lov\'{a}sz $\vartheta$-function of a graph $\Gr{G}$, which forms an upper bound
on the Shannon capacity of $\Gr{G}$ (see \eqref{eq: Lovasz 79 - theorem 1}), the variant of the $\vartheta$-function by Schrijver
does not possess that property, i.e., $\Theta(\Gr{G}) \not\leq \vartheta'(\Gr{G})$. It settles a query since the introduction of the latter
$\vartheta$-function in the late seventies \cite{McElieceRR78,Schrijver79}, which was also posed as an open question in
the last paragraph of \cite{BiT19}.

As a counterexample, let $\Gr{G} = \CGr{H}$ with the graph $\Gr{H} \eqdef \Gr{H}_{5, 3, 5}$ that is defined in
Example~\ref{example: vchrnum and svchrnum, regular graphs - part 2}.
In other words, $\Gr{G}$ is a regular graph on~32 vertices that are labeled by the 5-length binary vectors $\{0,1\}^5$,
and any two vertices in $\Gr{G}$ are adjacent if and only if the Hamming distance between their corresponding vectors
is either equal to 1 or~2. The degree of each vertex in $\Gr{G}$ is therefore equal to $\binom{5}{1}+\binom{5}{2} = 15$.
By equalities \eqref{eq1:17.11.23}, \eqref{eq: svchrnum vs. Lovasz theta-function}, and
Table~\ref{table: generalized Schrijver's example}, it follows that
\begin{align}
\label{eq2:20.11.2023}
\vartheta'(\Gr{G}) = 4 < 5 \tfrac{1}{3} = \vartheta(\Gr{G}).
\end{align}
In regard to the different values of $\vartheta'(\Gr{G})$ and $\vartheta(\Gr{G})$ in \eqref{eq2:20.11.2023}, it should be noted that
the regular graph $\Gr{G}$ is vertex-transitive, and the complement graph $\Gr{H} = \CGr{G}$ is not edge-transitive.
Indeed, it is proved in Example~\ref{example: vchrnum and svchrnum, regular graphs - part 2} that all graphs
$\Gr{H}_{\ell, d_{\mathrm{L}}, d_{\mathrm{H}}}$, with $1 \leq d_{\mathrm{L}} \leq d_{\mathrm{H}} \leq \ell$ and $\ell \in \naturals$,
are vertex-transitive; their complements are therefore vertex-transitive as well. On the other hand, the graph $\Gr{H}$ cannot be
edge-transitive. Otherwise, due to its vertex-transitivity, the strict inequality in \eqref{eq2:20.11.2023} could not hold by
Theorem~\ref{theorem: vchrmum and svchrnum of vt+et graphs} and equalities \eqref{eq1:17.11.23} and \eqref{eq: svchrnum vs. Lovasz theta-function}.
As a side note, the graph $\Gr{G}$ is edge-transitive, unlike its complement $\Gr{H}$; those properties can be verified by the SageMath
software tool \cite{SageMath}.

The independence number of $\Gr{G}$ satisfies $\indnum{\Gr{G}} = 4$, so inequality \eqref{eq2:17.11.23} holds with equality
for the considered graph $\Gr{G}$. This gives an example where the Schrijver $\vartheta$-function of a graph provides a tight upper bound
on its independence number, in contrast to the Lov\'{a}sz $\vartheta$-function of that graph.
The computation of the Shannon capacity of the graph $\Gr{G}$ presents a challenge in this context. However, obtaining a lower bound on
$\Theta(\Gr{G})$ would suffice to demonstrate the claim that $\Theta(\Gr{G}) \not\leq \vartheta'(\Gr{G})$. To that end, programming with
the SageMath software tool gives
\begin{align}
\label{eq3:20.11.2023}
& \indnum{\Gr{G} \boxtimes \Gr{G}} = 20,  \\
\label{eq4:20.11.2023}
\Rightarrow \, & \Theta(\Gr{G}) \geq \sqrt{\indnum{\Gr{G} \boxtimes \Gr{G}}} = \sqrt{20} > 4 = \vartheta'(\Gr{G}),
\end{align}
where the leftmost inequality in \eqref{eq4:20.11.2023} holds by \eqref{eq1:graph capacity}, and the rest of \eqref{eq4:20.11.2023}
holds by \eqref{eq2:20.11.2023} and \eqref{eq3:20.11.2023}. It should be noted that the computation of the independence
number of the strong product graph $\Gr{G} \boxtimes \Gr{G}$ appears to be significantly more efficient
by computing instead the clique number of the complement graph $\CGr{\Gr{G} \boxtimes \Gr{G}}$ with the algorithm
{\tt Cliquer}; this has been done with the SageMath software, which wraps a set of C routines named {\tt Cliquer} \cite{NO2003}.
An example of such a maximum independent set in the graph $\Gr{G} \boxtimes \Gr{G}$, which is a maximum clique of size~20
in its complement graph, is given by
\begin{align}
\set{I} = \Bigl\{ &(3, 31), \, (5, 24), \, (6, 6), \, (7, 1), \, (9, 5), \, (10, 16), \, (11, 10), \, (12, 11),
\, (13, 22), \, (14, 29), \nonumber \\
\label{eq5:20.11.2023}
& (17, 2), \, (18, 9), \, (19, 20), \, (20, 21), \, (21, 15), \, (22, 26), \, (24, 30), \, (25, 25), \, (26, 7), \, (28, 0) \Bigr\}.
\end{align}
It appears that there are 368,640 such maximal independent sets of size~20 in the strong product graph $\Gr{G} \boxtimes \Gr{G}$.
The interested reader is referred to the comprehensive review paper \cite{WuH15} on algorithms and applications
of the maximum clique problem.
\end{example}

\begin{example}[Three-tower Hanoi graphs]
\label{example: Hanoi Tower Graphs}
Three-tower Hanoi graphs, denoted $\Gr{H}_3^n$ for $n \in \naturals$, form a family of graphs that can be
constructed recursively as follows (see Figure~\ref{figure: Hanoi_Tower_graphs}).
\begin{itemize}
\item For $n=1$, the graph is the complete graph $\CoG{3}$ (a triangle).
\item For $n > 1$, the graph $\Gr{H}_3^n$ is constructed by converting each vertex at each corner of a triangle
in the graph $\Gr{H}_3^{n-1}$ into a new triangle.
\end{itemize}
\begin{figure}[h!t!]
\centering
\includegraphics[width=4.5cm]{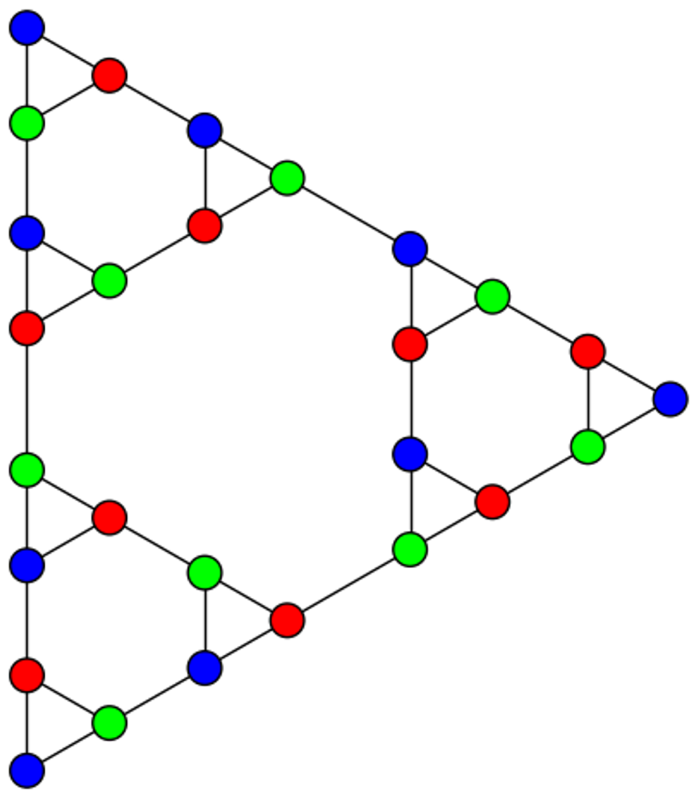}
\hspace{0.6cm}
\includegraphics[width=4.5cm]{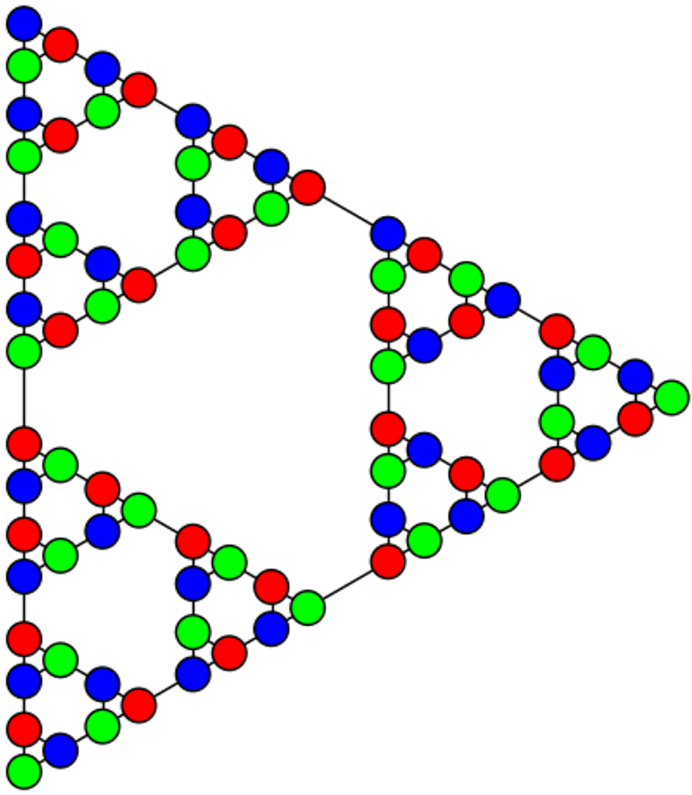}
\vspace*{0.2cm}
\caption{\label{figure: Hanoi_Tower_graphs} \centering{The three-tower Hanoi graphs $\Gr{H}_3^3$ (left)
and $\Gr{H}_3^4$ (right), whose clique and chromatic numbers are equal to~3.}}
\end{figure}

By construction,
\begin{itemize}
\item
The graph $\Gr{H}_3^n$ comprises $3^n$ vertices. It exhibits irregularity as each vertex, except for the three at the corners of
the outer triangle, has a degree of 3, while those corner vertices have a degree of 2 (see Figure~\ref{figure: Hanoi_Tower_graphs}).
These graphs are neither edge-transitive nor vertex-transitive.
\item A three-tower graph is composed of triangles, and it satisfies $\chrnum{\Gr{H}_3^n} = \clnum{\Gr{H}_3^n} = 3$ for all $n \in \naturals$.
\item By the sandwich theorem \eqref{eq1b: sandwich}, it follows that $\fchrnum{\Gr{H}_3^n} = \vartheta(\CGr{\Gr{H}_3^n}) = 3$ for all $n \in \naturals$.
\item Table~\ref{Table II: 3 graphs} compares lower bounds on the fractional chromatic numbers of these graphs.
The lower bound on the right-hand side of \eqref{eq1:04.10.23}, which is the Lov\'{a}sz $\vartheta$-function of the graph complement,
is tight for the three-tower Hanoi graphs. This is in contrast to the other lower bounds on the right-hand sides of
\eqref{eq2:04.10.23} and \eqref{eq3:04.10.23} that are not tight for the examined graphs.
\end{itemize}
\begin{table}[h!b!]
\caption{\label{Table II: 3 graphs}
\centering{The fractional chromatic numbers, and their lower bounds in \eqref{eq1:04.10.23}--\eqref{eq3:04.10.23},
for the two Hanoi graphs in Figure~\ref{figure: Hanoi_Tower_graphs}, and the third graph $\Gr{H}_3^5$.}}
\renewcommand{\arraystretch}{1.5}
\vspace*{0.1cm}
\centering
\begin{tabular}{||c|c|c|c|c||}
\hline
\text{Graph} $\Gr{G}$ &$\fchrnum{\Gr{G}}$ &\text{RHS of} \eqref{eq1:04.10.23} & \text{RHS of}
\eqref{eq2:04.10.23} & \text{RHS of} \eqref{eq3:04.10.23} \\[0.05cm] \hline
$\Gr{H}_3^3$ & $ 3 $ & $ 3 $ & $ 2.4677 $ & $ 2.4334 $ \\
$\Gr{H}_3^4$ & $ 3 $ & $ 3 $ & $ 2.4927 $ & $ 2.4048 $ \\
$\Gr{H}_3^5$ & $ 3 $ & $ 3 $ & $ 2.4984 $ & $ 2.3957 $ \\[0.05cm]
\hline
\end{tabular}
\end{table}
\end{example}

\begin{example}[Windmill graphs]
\label{example: windmill graphs}
Let $k \geq 2$ and $n \geq 2$ be integers.
The windmill graph $\Gr{W}_{k,n}$ is an undirected, connected, and simple graph, constructed
by joining $n$ copies of the complete graph $\CoG{k}$ at a shared
universal vertex (see Figure~\ref{fig:windmill graphs}). Specifically, if $k=2$, then $\Gr{W}_{2,n}$
is a star graph with $n$ leaves (see the leftmost graph in Figure~\ref{fig:windmill graphs}).
The windmill graphs are known to be perfect graphs \cite{Golumbic04}. The chromatic and fractional chromatic
numbers of the graph $\Gr{G} = \Gr{W}_{k,n}$ coincide, being equal to $k$.

\begin{figure}[h!t!]
\centering
\includegraphics[width=3.90cm]{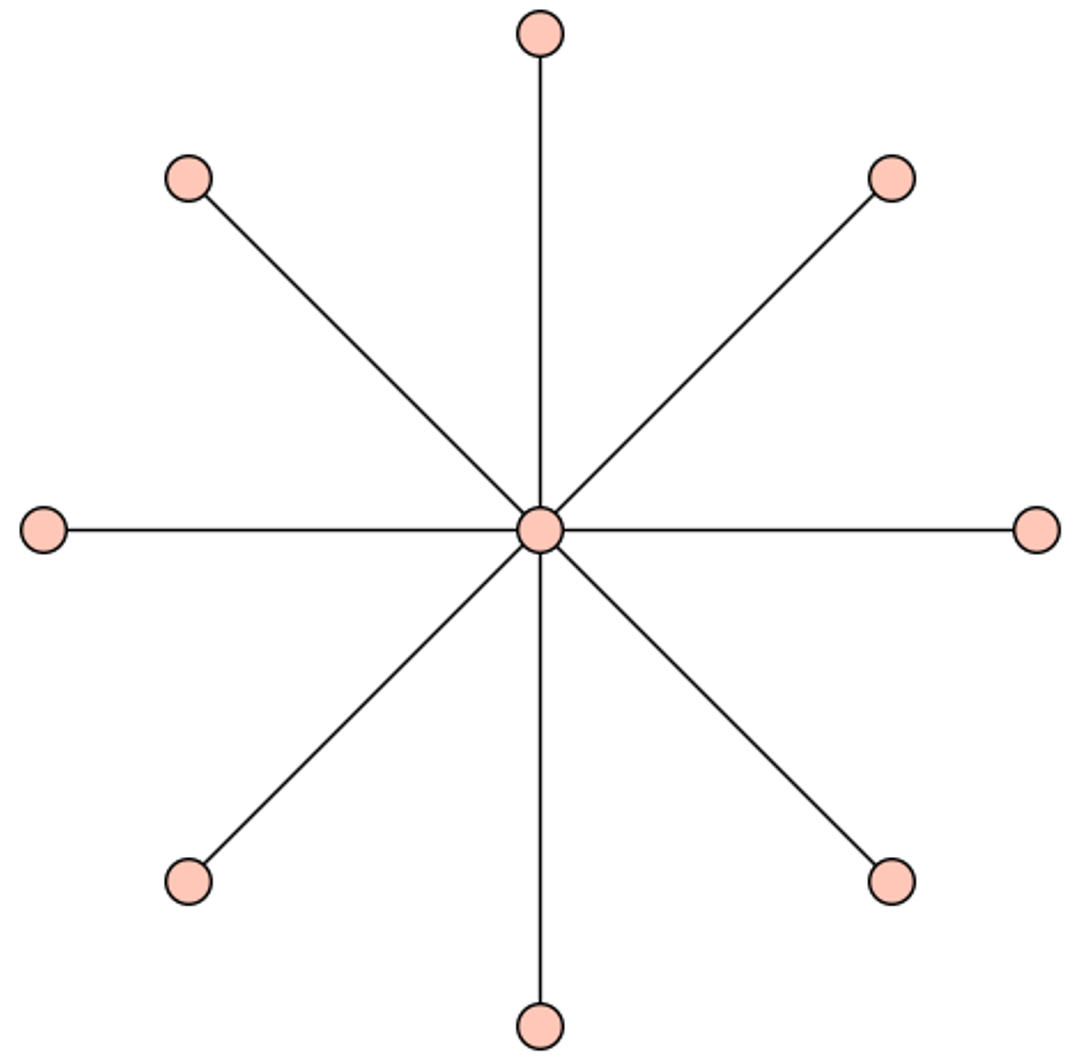}
\includegraphics[width=3.88cm]{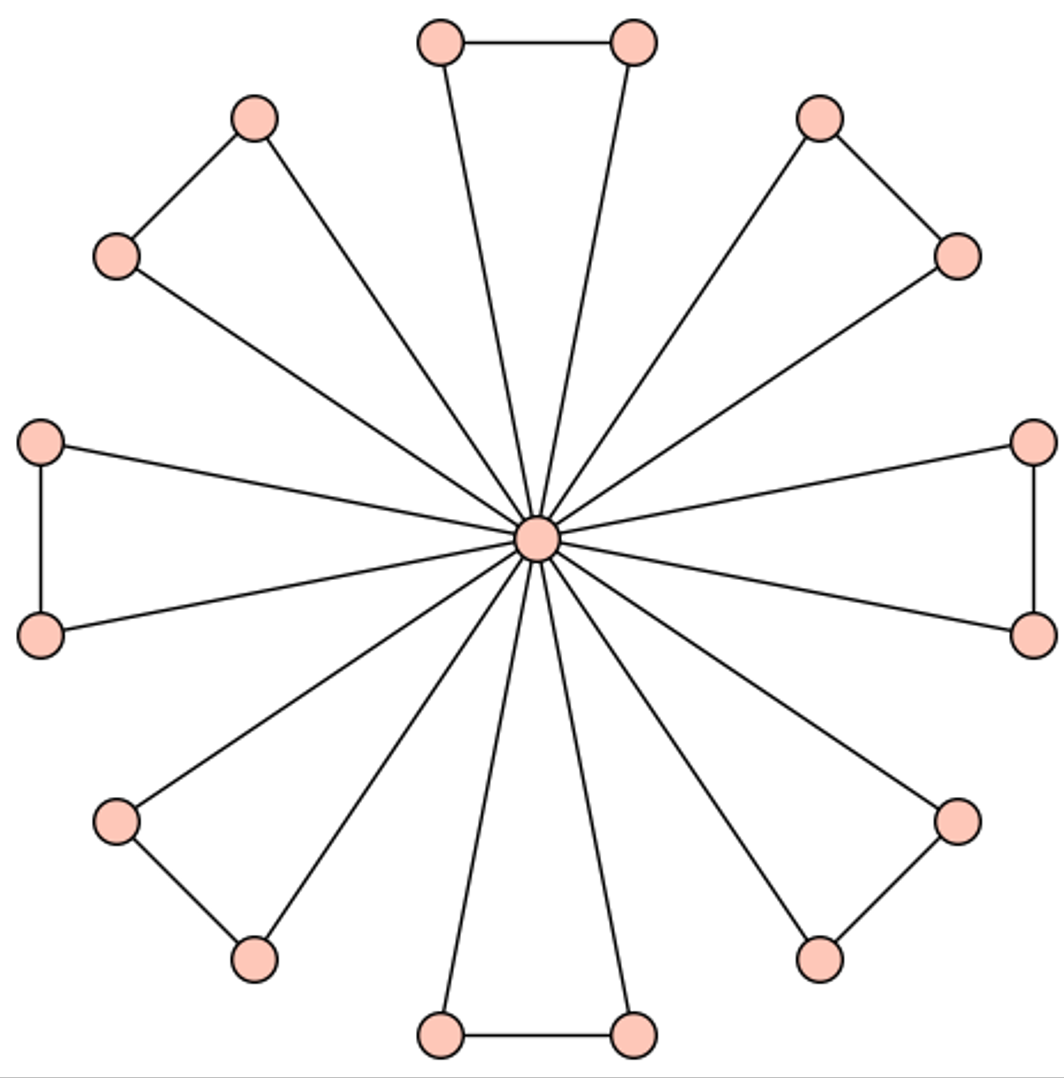}
\includegraphics[width=3.90cm]{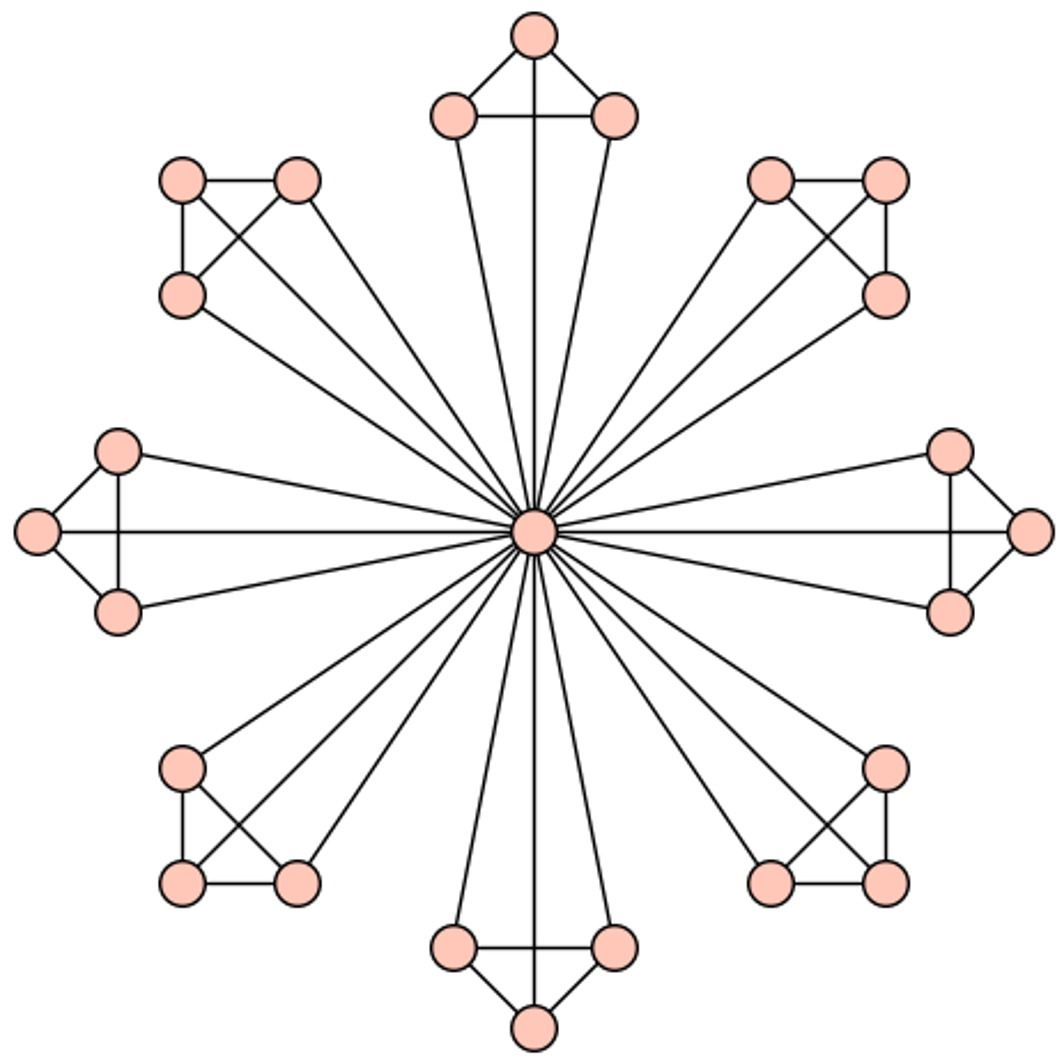}
\includegraphics[width=3.89cm]{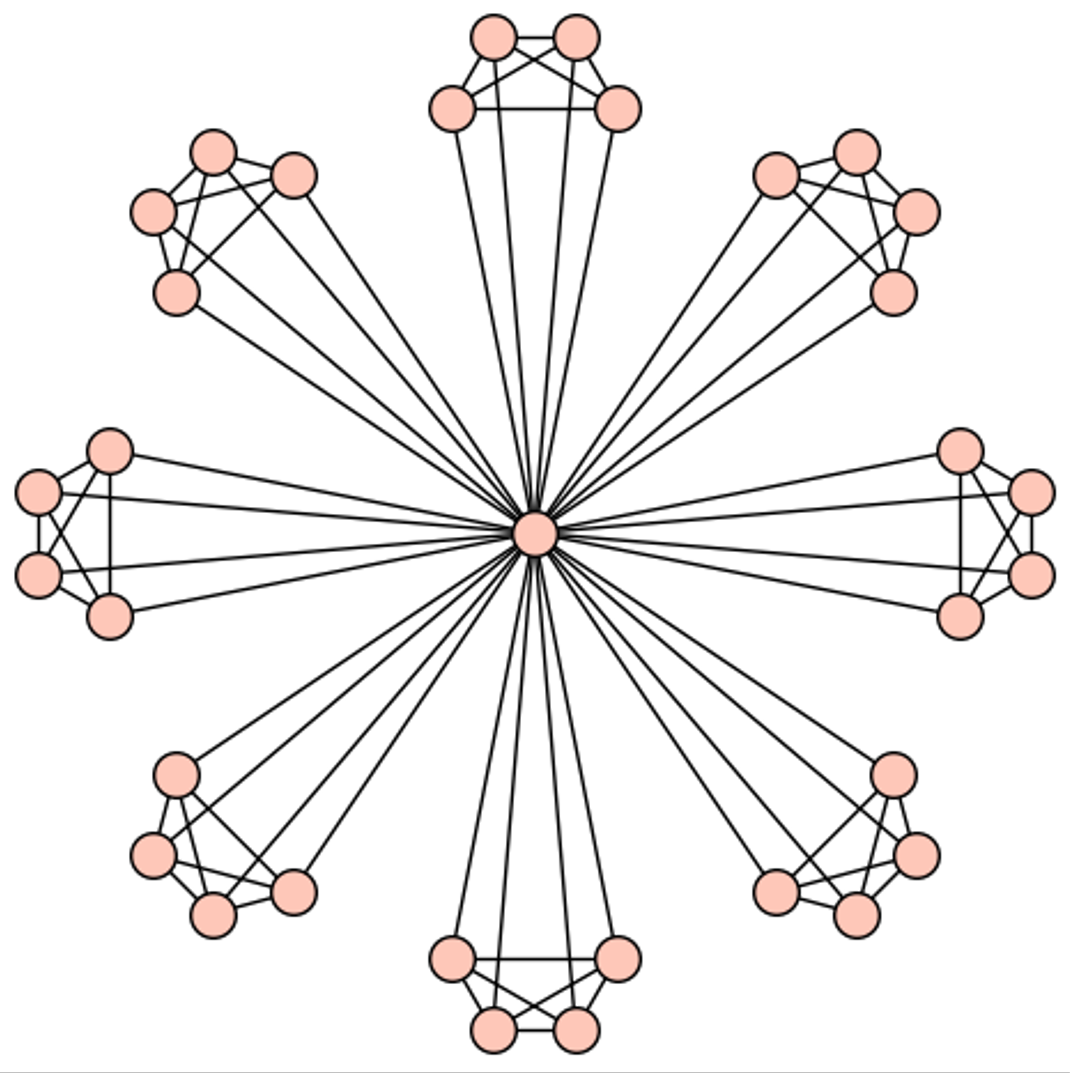}
\vspace*{0.2cm}
\caption{\label{fig:windmill graphs} \centering{From left to right: $\Gr{W}_{2,8}, \Gr{W}_{3,8}, \Gr{W}_{4,8}$, and $\Gr{W}_{5,8}$.}}
\end{figure}
\begin{table}[h!t!]
\caption{\label{Table: windmill graphs}
\centering{The chromatic and fractional chromatic numbers of windmill graphs, which are perfect graphs, and their lower bounds in \eqref{eq1:04.10.23}--\eqref{eq3:04.10.23}.
The graphs are shown in Figure~\ref{fig:windmill graphs}.}}
\renewcommand{\arraystretch}{1.5}
\vspace*{0.1cm}
\centering
\begin{tabular}{||c|c|c|c|c|c||}
\hline
\text{Graph} $\Gr{G}$ &$\chrnum{\Gr{G}}$ &$\fchrnum{\Gr{G}}$ &\eqref{eq1:04.10.23} & \eqref{eq2:04.10.23} & \eqref{eq3:04.10.23} \\[0.05cm] \hline
$\Gr{W}_{2,8}$ & $2$ & $2$ & $2$ & $2$ & $2$ \\
$\Gr{W}_{3,8}$ & $3$ & $3$ & $3$ & $2.2832$ & $2.3450$ \\
$\Gr{W}_{4,8}$ & $4$ & $4$ & $4$ & $2 \tfrac{1}{2} $ & $3$ \\
$\Gr{W}_{5,8}$ & $5$ & $5$ & $5$ & $2.6893$ & $3.7259$ \\
\hline
\end{tabular}
\end{table}
Table~\ref{Table: windmill graphs} compares the fractional chromatic numbers
of the windmill graphs $\Gr{W}_{k,n}$, for several values of $k,n \geq 2$, with their
corresponding lower bounds on the right-hand sides of \eqref{eq1:04.10.23},
\eqref{eq2:04.10.23}, and \eqref{eq3:04.10.23}.
That comparison supports the fact that, by Theorem~\ref{theorem: graph invariants of perfect graphs},
the lower bound on the right-hand side of \eqref{eq1:04.10.23} is tight for perfect graphs, in contrast
to the other two lower bounds on the right-hand sides of \eqref{eq2:04.10.23} and \eqref{eq3:04.10.23}
that are not necessarily tight for the examined windmill graphs.
\end{example}

\begin{example}[Independence, capacity and clique numbers of random graphs]
\label{example: random graphs}
The standard model of a random graph is the binomial model $\Gr{G}(n,p)$, pioneered by
Erd\"{o}s and R\'{e}nyi. We let $0 < p = p(n) < 1$ be a number that may depend on the
number of vertices $n$. Then, the random graph $\Gr{G}(n,p)$ is obtained by independently
including each of the $\binom{n}{2}$ possible edges with probability $p$. The
random graph $\Gr{G}(n,p)$ is said to have a certain property with high probability, if the
probability that such a random graph has that property tends to~1 as we let $n$ tend
to infinity.
If $n \to \infty$ and $p \in (0,1)$ is kept fixed, then with high probability (see
Theorem~2 in \cite{Juhasz82} and Example~11.17 in \cite{Lovasz19})
\begin{align}
\label{eq1:04.11.23}
\frac12 \sqrt{\frac{(1-p)n}{p}} + O\bigl(n^{\tfrac13} \ln n\bigr) \leq \vartheta \bigl(\Gr{G}(n,p)\bigr)
\leq 2 \sqrt{\frac{(1-p)n}{p}} + O\bigl(n^{\tfrac13} \ln n \bigr),
\end{align}
and, by replacing the probability $p$ with $1-p$ for the complement graph, with high probability
\begin{align}
\label{eq2:04.11.23}
\frac12 \sqrt{\frac{pn}{1-p}} + O\bigl(n^{\tfrac13} \ln n\bigr) \leq \vartheta \bigl(\overline{\Gr{G}(n,p)}\bigr)
\leq 2 \sqrt{\frac{pn}{1-p}} + O\bigl(n^{\tfrac13} \ln n\bigr).
\end{align}
The analysis in \cite{Juhasz82} is extended in \cite{Coja-Oghlan05} to the case where
$\frac{1}{n} \, (\ln n)^{\tfrac16} \leq p \leq 1-\frac{1}{n} \, (\ln n)^{\tfrac16}$,
with some concentration inequalities for the Lov\'{a}sz $\vartheta$-function of random graphs $\Gr{G}(n,p)$.
High-probability lower bounds on the independence number and clique number of a random graph $\Gr{G}(n,p)$
can be obtained by combining the bounds in \eqref{eq1:17.10.23}, \eqref{eq1b:17.10.23}, \eqref{eq1:04.11.23},
and \eqref{eq2:04.11.23}. These high-probability lower bounds consequently hold also for the Shannon capacity
of random graphs.

Using \eqref{eq1:04.11.23} and \eqref{eq2:04.11.23}, the inequalities $\indnum{\Gr{G}} \leq \vartheta(\Gr{G})$
and $\clnum{\Gr{G}} \leq \vartheta(\CGr{G})$ yield high-probability upper bounds on the independence and clique
numbers of random graphs. Since $\Theta(\Gr{G}) \leq \vartheta(\Gr{G})$, a high-probability upper bound on their
Shannon capacity is also given by the rightmost term in \eqref{eq1:04.11.23}.

The interested reader is referred to Section~6.1 in \cite{Coja-Oghlan05} in regard to an algorithmic approach
and its probabilistic analysis for the approximation of the independence numbers of random graphs. It
employs a certain greedy algorithm that, for an input graph $\Gr{G} = \Gr{G}(n,p)$, the algorithm most
probably finds an independent set of cardinality of at least $\dfrac{\ln(np)}{2p}$, thus providing a lower
bound on $\indnum{\Gr{G}}$. It then suggests an algorithm that involves the calculation of the Lov\'{a}sz
$\vartheta$-function $\vartheta(\Gr{G})$ as an upper bound on $\indnum{\Gr{G}}$. The expected running time of
the latter algorithm is polynomial in $n$ with high probability, and it provides an upper bound that scales
like $\sqrt{n/p}$.
\end{example}

\section{Summary and Outlook}
\label{section: summary}
The paper primarily focuses on deriving results that hinge on the Lov\'{a}sz $\vartheta$-function of
a graph. These results cover various aspects, including insights into the Shannon capacity of graphs,
exploration of cospectral and nonisomorphic graphs, and derivation of bounds on graph invariants.
Throughout the paper, numerical results are provided to illustrate the usefulness of the derived findings.
This research paper also serves as a tutorial on the subject of the Lov\'{a}sz $\vartheta$-function,
its variants, and its significant roles in algebraic graph theory and zero-error information theory.

This section concludes by posing open questions that are directly relevant to the discussed work.
\begin{enumerate}
\item By Theorem~3.2 of \cite{Alon19}, there exists a sequence of graphs $\{\Gr{H}_n\}$ such that
$\Gr{H}_n$ is a graph on $n$ vertices whose Shannon capacity is at most~3 and its Lov\'{a}sz
$\vartheta$-function is at least $(1+o(1)) \, n^{1/4}$. This shows, in particular, that the Lov\'{a}sz
$\vartheta$-function of a graph cannot be upper-bounded by any function of the independence number.
In light of that result, it is left for further research to study if Schrijver's variant of the
$\vartheta$-function of a graph can be upper-bounded by any function of the independence number.
In other words, by \eqref{eq1:17.11.23}, it is an open question whether the vector chromatic number of
a graph can be upper-bounded by a function of the clique number of the graph. These equivalent questions
are akin to an open question posed in Section~6 of \cite{Alon19}, which asks whether the Shannon capacity of
a graph $\Theta(\Gr{G})$ can be upper-bounded by any function of the independence number $\indnum{\Gr{G}}$.
In particular, it is unknown whether there exists a sequence of graphs whose independence numbers are
equal to~2, while their Shannon capacity is unbounded.

\item The self-complementary vertex-transitive graphs and self-complementary strongly regular graphs are two
distinct subclasses of self-complementary regular graphs. These subclasses have been studied in various works,
including \cite{Farrugia99, Harary69, Mathon88, Muzychuk99, Peisert01, Rao85, Rosenberg82, Ruiz81, Sason23, Zelinka79}.
Notably, self-complementary vertex-transitive graphs do not necessarily fall under the realm of strongly regular
graphs, as exemplified in \cite{Ruiz81}. It remains an open question whether there exists a self-complementary
strongly regular graph that is not vertex-transitive, as indicated on page~88 of \cite{Farrugia99}. This inquiry
is relevant to Item~\ref{item 4: extension of Thm. 12 by Lovasz} of Theorem~\ref{thm:extension of Thm. 12 by Lovasz}.
Furthermore, according to that item, the minimum Shannon capacity among all self-complementary graphs of a fixed
order $n$ is achieved by those that are vertex-transitive or strongly regular, and this minimum is equal to $\sqrt{n}$.
Another unresolved issue is determining the maximum Shannon capacity among all self-complementary graphs of a fixed
order $n$, and identifying which graphs achieve this maximum.

\item By Theorems~\ref{theorem: vchrnum and svchrnum of srg} and~\ref{theorem: vchrmum and svchrnum of vt+et graphs},
vector and strict vector chromatic numbers coincide for every graph that is either strongly regular
or both vertex- and edge-transitive. Additionally, Example~\ref{example: vchrnum and svchrnum, regular graphs}
illustrates that these numbers may also coincide for regular graphs that lack strong
regularity or both vertex- and edge-transitivity. Consequently, it is of interest to establish further
sufficient conditions for the coincidence of the vector and strict vector chromatic numbers of graphs.
\end{enumerate}

\section*{Acknowledgments}
The author gratefully acknowledges the timely and constructive reports of the anonymous referees.
Special thanks are extended to Clive Elphick for sharing the slides of his recent seminar talk and for pointing
out some relevant references, and to Emre Telatar for valuable advice on Latex issues.

\section*{Conflict of interest}
The author declares no conflicts of interest.

\end{document}